\documentclass[11pt]{amsart}
\usepackage{amscd,amsmath,amssymb,amsfonts,verbatim}
\usepackage[cmtip, all]{xy}
\usepackage{comment}

\newtheorem{thm}{Theorem}[section]
\newtheorem{prop}[thm]{Proposition}
\newtheorem{lem}[thm]{Lemma}
\newtheorem{cor}[thm]{Corollary}

\theoremstyle{definition}

\newtheorem{defn}[thm]{Definition}
\theoremstyle{remark}
\newtheorem{remk}[thm]{Remark}
\newtheorem{remks}[thm]{Remarks}

\newtheorem{exm}[thm]{Example}
\newtheorem{exms}[thm]{Examples}
\newtheorem{notat}[thm]{Notation}
\numberwithin{equation}{section}

{\hfill$\square$\end{defn}}
\newenvironment{rem}{\begin{remk}}%
{\hfill$\square$\end{remk}}
{\hfill$\square$\end{remks}}
{\hfill$\square$\end{exm}}
{\hfill$\square$\end{exms}}
{\hfill$\square$\end{notat}}

\newcommand{\thmref}{Theorem~\ref}
\newcommand{\propref}{Proposition~\ref}
\newcommand{\corref}{Corollary~\ref}

\newcommand{\lemref}{Lemma~\ref}

\newcommand{\sB}{{\mathcal B}}
\newcommand{\sC}{{\mathcal C}}

\newcommand{\sE}{{\mathcal E}}
\newcommand{\sF}{{\mathcal F}}
\newcommand{\sG}{{\mathcal G}}

\newcommand{\sK}{{\mathcal K}}
\newcommand{\sL}{{\mathcal L}}
\newcommand{\sM}{{\mathcal M}}

\newcommand{\sO}{{\mathcal O}}

\newcommand{\sQ}{{\mathcal Q}}

\newcommand{\sU}{{\mathcal U}}

\newcommand{\sW}{{\mathcal W}}
\newcommand{\sX}{{\mathcal X}}
\newcommand{\sY}{{\mathcal Y}}
\newcommand{\sZ}{{\mathcal Z}}

\newcommand{\C}{{\mathbb C}}

\renewcommand{\P}{{\mathbb P}}

\newcommand{\Z}{{\mathbb Z}}

\newcommand{\fm}{{\mathfrak m}}
\newcommand{\fg}{{\mathfrak g}}
\newcommand{\fz}{{\mathfrak z}}

\newcommand{\CH}{{\rm CH}}

\newcommand{\surj}{\twoheadrightarrow}
\newcommand{\inj}{\hookrightarrow}

\newcommand{\Spec}{{\rm Spec \,}}

\newcommand{\Ker}{{\rm Ker}}

\newcommand{\Sch}{{\operatorname{\mathbf{Sch}}}}
\newcommand{\qp}{{\operatorname{\mathbf{Qproj}}}}

\renewcommand{\>}{\rangle}

\newcommand{\Sm}{{\mathbf{Sm}}}

\newcommand{\ds}{{/\kern-3pt/}}

\newcommand{\ov}{\overline}

\newcommand{\tuborg}{\left\{\begin{array}{ll}}
\newcommand{\sluttuborg}{\end{array}\right.}

\newcommand{\inv}{{\rm Inv}}

\newcommand{\wh}{\widehat}

\newcounter{elno}

\newcounter{elno-abc}   

\newcounter{elno-abc-prime}   

\usepackage{epigraph}

\begin{document}
    
\title[Atiyah-Segal theorem for Deligne-Mumford stacks]
{Atiyah-Segal theorem for Deligne-Mumford stacks and applications}
\author{Amalendu Krishna and Bhamidi Sreedhar}
\address{School of Mathematics, Tata Institute of Fundamental Research,  
1 Homi Bhabha Road, Colaba, Mumbai, India}
\email{amal@math.tifr.res.in}
\email{sreedhar.bhamidi@gmail.com}



%\maketitle

%\tableofcontents
%This version submitted to ARXIV on 18 January 2017.

\begin{abstract}
We prove an Atiyah-Segal isomorphism for the higher $K$-theory of
coherent sheaves on quotient Deligne-Mumford stacks over $\C$. As an application,
we prove the Grothendieck-Riemann-Roch theorem for such  stacks.
This theorem establishes an isomorphism between the higher $K$-theory of 
coherent sheaves on a Deligne-Mumford stack and the higher Chow groups
of its inertia stack. Furthermore, this isomorphism is covariant for
proper maps between Deligne-Mumford stacks.  
\end{abstract} 
\setcounter{tocdepth}{1}
\maketitle
\tableofcontents  
%{\hypersetup{linkcolor=black} \tableofcontents}

\section{Introduction}\label{sec:Intro}
The Grothendieck-Riemann-Roch theorem provides an isomorphism between
the rational Grothendieck group of coherent sheaves on a regular scheme
and its rational Chow groups. It also shows that this isomorphism is
covariant for proper morphisms between schemes. This result of
Grothendieck was extended to the case of singular schemes by Baum, Fulton
and MacPherson \cite{BFM}. It was later generalized to the level of higher 
$K$-theory and higher Chow groups of all quasi-projective (possibly singular) 
schemes by Bloch \cite{Bloch}.

The Riemann-Roch theorem for higher $K$-theory of quasi-projective 
schemes is arguably one of the most important and deep results
in algebraic geometry. The famous result of Riemann and Roch
that computes the dimension of the space of sections of a line bundle
on a Riemann surface in terms of its topological invariants, is a
special case of this. The celebrated index theorem of Atiyah and Singer
that computes the index of elliptic operators on a compact manifold,
is a differential geometric avatar of the Riemann-Roch theorem.

One reason for the outstanding nature of the Grothendieck-Riemann-Roch
theorem is that it identifies two seemingly very different theories in 
algebraic geometry in a functorial way. One of these (algebraic $K$-theory) is 
abstractly defined as the higher homotopy groups of the classifying space
of a category, and hence, is often very hard to compute
while the  other (higher Chow groups) is described
in terms of explicit generators and relations, and hence, in principle, is 
often more computable.

Many of the recent works in moduli theory seek a version of the
Riemann-Roch theorem for algebraic stacks. For instance, in 
Gromov-Witten theory,
the virtual cycles, whose degrees give rise to the Gromov-Witten invariants,
are the Chern classes of vector bundles on a moduli space (see \cite{Behrend}).
These vector bundles can be lifted to virtual vector bundles on the
moduli stack, and the Riemann-Roch theorem can then be used to
compute the degrees of the virtual cycles. Note that these stacks can be 
highly singular in general. For example, Kontsevich's moduli stack 
${\sM_{g,n}}(X,\beta)$ of stable maps
from $n$-pointed stable curves of genus $g$ to a projective variety $X$ 
is known to be singular even if $X$ is smooth.

In string theory, physicists are often interested in computing some cohomological
invariants such as Euler characteristics of orbifolds
(see \cite{Dixton-1}, \cite{Dixton-2}). They want to compute
these invariants for an orbifold and its resolution of singularities
(the orbifold cohomology). Since the $K$-theory and Chow groups of these
orbifolds are quotients of the corresponding groups for the associated
quotient stacks, the Riemann-Roch for quotient stacks are the natural tools
to analyze these Chow cohomology and Euler characteristics.  

%\enlargethispage{20pt}

The higher $K$-theory of coherent sheaves and vector bundles for algebraic 
stacks is defined exactly as for schemes. But there is less clarity on
the correct definition of higher Chow groups. 
The rational Chow groups of
Deligne-Mumford stacks were defined independently by Gillet \cite{Gillet}
and Vistoli \cite{Vis}. The integral Chow groups of  
algebraic stacks with affine stabilizers were defined by Kresch \cite{Kresch}.
For stacks which occur as quotients of quasi-projective schemes
by actions of linear algebraic groups, the theory of (Borel style) 
higher Chow groups was defined by Edidin and Graham \cite{EG1}. 
This construction is based on an earlier construction of the Chow groups
of classifying stacks by Totaro \cite{TO}. These Chow groups satisfy all
the expected properties of a Borel-Moore homology theory in the
category of quotient stacks and representable maps (see \cite{AK3} for
details). 

When Vistoli constructed intersection theory on Deligne-Mumford
stacks \cite{Vis}, he asked if his Chow groups are related to the $K$-theory 
of the stack, and secondly, if there exists a Riemann-Roch theorem for proper
maps between Deligne-Mumford stacks. Note that for quotient stacks, 
the Chow groups 
of Gillet, Vistoli, Kresch and Edidin-Graham all agree with rational 
coefficients.

When one tries to generalize the Grothendieck-Riemann-Roch theorem to
quotient stacks, one runs into two serious obstacles. The first problem is that
the direct generalization of the Riemann-Roch theorem for schemes actually fails
for quotient stacks.
Already for the classifying stack $\sX =[{\Spec({\C})}/G]$, where
$G = {\Z/n}$, one knows that $G_0(\sX)_{\C} \simeq {\C[t]}/{(t^n-1)}$
while $\CH_*(\sX)_{\C} \simeq \C$. Hence, there can be no Riemann-Roch map
$G_0(\sX)_{\C} \to \CH_*(\sX)_{\C}$ which is an isomorphism.
This shows that the Borel style higher Chow groups of a stack $\sX$ 
are not the right objects which can describe its $K$-theory. A version of 
Grothendieck-Riemann-Roch theorem using the higher Chow groups of a stack was 
established in 
\cite{EG5} for $G_0$ and in \cite{AK1} for higher $K$-theory. 
But these Riemann-Roch maps fail to describe $K$-theory completely. 

The second main problem in Riemann-Roch for stacks is that unlike the category
of schemes, a morphism between stacks is very often not representable. 
This does not allow  the techniques of Riemann-Roch for
schemes to directly generalize to stacks.
This creates a very serious obstacle in proving the covariance of any possible
Riemann-Roch map.

It was observed by Edidin and Graham
in \cite{EG2} that the inertia stack should be the right object for
the Chow groups while studying the Riemann-Roch transformation
for stacks. Using the Chow groups of inertia stack, they proved a version of
the Riemann-Roch theorem for the Grothendieck group of quotient
stacks when we assume
the stack to be smooth Deligne-Mumford and the morphism to be the coarse moduli
space map \cite{EG2}.

Other than this, the Riemann-Roch problem,
identifying the higher $K$-theory and higher Chow groups of quotient stacks
in a functorial way, is completely open. 
The purpose of this work is to fill this gap
in the study of the cohomology theories of separated quotient stacks. 
In particular, we verify the expectation of Edidin and Graham 
that the higher $K$-theory (with complex coefficients) of
a separated quotient stack $\sX$ can be completely described in terms of the
Borel style higher Chow groups of the inertia stack of $\sX$.

\subsection{Main results}\label{sec:MR}
Given a stack $\sX$ which is of finite type over $\C$, let $I_{\sX}$ denote
its inertia stack. Let $G_i(\sX)$ denote the homotopy groups of the
Quillen $K$-theory spectrum of the category of coherent sheaves on $\sX$.
We say that $\sX$ is a quotient stack if it is of the form
$[X/G]$, where $G$ is a linear algebraic group acting on a
separated algebraic space $X$.
If $X$ is a quasi-projective scheme over $\C$ and $\sX = [X/G]$, we 
let $\CH_*(\sX, \bullet)$ denote its higher Chow groups, as defined
in \cite{EG1}. Given a map of stacks $f: \sX \to \sY$, let
$f^I: I_{\sX} \to I_{\sY}$ denote the induced map on the inertia stacks.
In Theorems~\ref{thm:GRR-M}, ~\ref{thm:GRR-Gen-0} and
~\ref{thm:Intro-3} below, all $K$-groups and higher Chow groups
of stacks are taken with complex coefficients.
The main result of this article is the following
extension of the Grothendieck-Riemann-Roch theorem to separated quotient
stacks.    
 
\begin{thm}\label{thm:GRR-M}
Let $f: \sX \to \sY$ be a proper morphism of separated quotient 
stacks of finite type over $\C$ with quasi-projective 
coarse moduli spaces. Then for any integer $i \ge 0$, there exists 
a commutative diagram
\begin{equation}\label{eqn:GRR-st-0}
\xymatrix@C1pc{
G_i(\sX) \ar[r]^-{\tau_{\sX}} \ar[d]_{f_*} & \CH_*(I_{\sX},i)
\ar[d]^{\bar{f^I}_*} \\
G_i(\sY) \ar[r]^-{\tau_{\sY}} & \CH_*(I_{\sY},i)}
\end{equation}
such that the horizontal arrows are isomorphisms.
\end{thm}

The assumption in \thmref{thm:GRR-M} that the coarse moduli spaces are 
quasi-projective is made because the theory of higher Chow groups 
of general schemes is not well behaved and many
fundamental results for this theory are not known in the general context.
If we ignore this assumption, we can prove the following. 

%\enlargethispage{30pt}

\begin{thm}\label{thm:GRR-Gen-0}
Let $f: \sX \to \sY$ be a proper morphism of separated quotient 
stacks of finite type over $\C$.  
Then there exists a commutative diagram
\begin{equation}\label{eqn:GRR-Gen-0-1}
\xymatrix@C1pc{
G_0(\sX) \ar[r]^-{\tau_{\sX}} \ar[d]_{f_*} & \CH_*(I_{\sX})
\ar[d]^{\bar{f^I}_*} \\
G_0(\sY) \ar[r]^-{\tau_{\sY}} & \CH_*(I_{\sY})}
\end{equation}
such that the horizontal arrows are isomorphisms.
\end{thm}
    
As we have already mentioned, a special case of \thmref{thm:GRR-Gen-0}, 
where $\sX$ is smooth and $\sY$ is its coarse moduli space,
is a direct consequence of the main results of \cite{EG2}. 
Note that if $\sX$ is an algebraic space, then the map $I_{\sX} \to \sX$ is an
isomorphism. Therefore, Theorem~\ref{thm:GRR-M}
recovers the Riemann-Roch theorem of Bloch in this case.
Theorem~\ref{thm:GRR-Gen-0} extends the classical Grothendieck-Riemann-Roch
theorem to algebraic spaces.

\vskip .3cm

We deduce the Riemann-Roch theorems for stacks as a consequence of the
following more general Atiyah-Segal isomorphism which describes the higher 
$K$-theory of coherent sheaves on a stack in terms of the 
geometric part of the higher $K$-theory of the inertia stack.
For a stack $\sX$, the higher $K$-theory of coherent sheaves $G_i(\sX)$
is a module over $K_0(\sX)$. Let $\fm_{\sX} \subset K_0(\sX)_{\C}$ denote the ideal 
of virtual vector bundles on $\sX$ whose rank is zero on all connected components
of $\sX$. 
%The Atiyah-Segal correspondence that we prove says the following.

\begin{thm}\label{thm:Intro-3}  
Let $f: \sX \to \sY$ be a proper morphism of separated quotient 
stacks of finite type over $\C$. Let $i \ge 0$ be an integer.
Then there is a commutative diagram
\begin{equation}\label{eqn:ASCS-00}
\xymatrix@C1pc{
G_i(\sX) \ar[r]^-{\vartheta_{\sX}} \ar[d]_{f_*} & G_i(I_{\sX})_{\fm_{I_{\sX}}} 
\ar[d]^{f^I_*} \\
G_i(\sY) \ar[r]^-{\vartheta_{\sY}} & G_i(I_{\sY})_{\fm_{I_{\sY}}}}
\end{equation}
such that the horizontal arrows are isomorphisms.
\end{thm}

Note that unlike \thmref{thm:GRR-M}, the Atiyah-Segal isomorphism makes
no assumption on the coarse moduli spaces.

Another point one needs to note here is that the map
$f^I_*$ in ~\eqref{eqn:ASCS-00} is very subtle.
The proper map $f^I: I_{\sX} \to I_{\sY}$ induces
a push-forward map $f^I_*: G_i(I_{\sX}) \to G_i(I_{\sY})$. However, there is
no guarantee that $f^I_*$ is continuous with respect to the $\fm_{I_{\sX}}$-adic
topology on $G_i(I_{\sX})$ and the $\fm_{I_{\sY}}$-adic topology on $G_i(I_{\sY})$.
In particular, $f^I_*$ does not automatically induce the map on the
localizations.

A very crucial part of the proof of \thmref{thm:Intro-3}
is to show that the map $f^I_*$ on the right side of ~\eqref{eqn:ASCS-00} is
defined. Existence of this map is in fact part of an old conjecture of 
K{\"o}ck \cite{Kock}. As part of the proof of \thmref{thm:Intro-3} therefore,
we also prove a version of K{\"o}ck's conjecture for separated quotient stacks.
We refer to \thmref{thm:Com-compl} for a precise statement of our result.
For a representable map of more general smooth quotient stacks, this 
was proven by Edidin and Graham \cite{EG5}.

When $G$ is a finite group acting on a compact oriented differentiable
manifold, a version of the isomorphism $\vartheta_{\sX}$ of \thmref{thm:Intro-3}
was proven long ago by Segal. 
%\cite{HH}.
When the finite group $G$ acts on a scheme $X$, a finer version of the isomorphism 
$\vartheta_{\sX}$ was obtained for $\sX = [X/G]$ by Vistoli \cite{Vistoli}. 
If we are working with $K$-theory with complex 
coefficients, \thmref{thm:Intro-3} thus provides a complete generalization
of Vistoli's theorem to all separated quotient stacks. Additionally, it also
proves the covariance of this isomorphism.

The equivariant higher Chow groups of schemes with group action are defined
as the ordinary higher Chow groups (see \cite{Bloch}) of the Borel spaces.
In particular, these equivariant Chow groups are actually a non-equivariant
cohomology theory and much more explicitly defined. On the other hand,
the equivariant $K$-theory is a very abstract object and there is no
explicit way to compute them. Novelty of Theorems~\ref{thm:GRR-M}
and ~\ref{thm:GRR-Gen-0} is that they provide a formula for the equivariant
$K$-theory in terms of the ordinary Chow groups of Borel spaces. 
Apart from this, they also provide formula to compute the
Euler characteristics of coherent sheaves and vector bundles on separated
quotient stacks. Note that such a formula was known before only for smooth
quotient stacks (see \cite{G1} and \cite{EG2}, see also \cite{AS} in the 
topological case).

The results obtained above motivate two important questions for future
investigation. The first is related to one of the main
objectives of the paper, namely, to describe
$K$-theory of stacks in terms of more explicit objects like higher Chow groups.
Going beyond the Riemann-Roch theorem, it was proven by 
Bloch-Lichtenbaum \cite{BL}
and Friedlander-Suslin \cite{FS} 
that the integral $K$-theory of quasi-projective schemes
could be described by higher
Chow groups in terms of an Atiyah-Hirzebruch spectral sequence.
Our results provide a strong indication that there should be a similar spectral
sequence consisting of the Borel style higher Chow groups of the inertia stack
that should converge to the $K$-theory of the underlying stack. This
is a very important but a challenging problem in the study of cohomology theory of
stacks. The second question is about extending our results to Artin stacks.
Here, one could ask if there is an  analogue of the Atiyah-Segal isomorphism or the
Riemann-Roch isomorphism for Artin quotient stacks. We hope to come back to 
these questions in future projects.

We end the description of our main results with few comments on their
comparison with the Riemann-Roch theorems for stacks that are available
in the literature. A Riemann-Roch theorem in the equivariant setting was
first established by K{\"o}ck \cite{Kock}, where the target of the 
Riemann-Roch map is the completion of the equivariant $K$-theory with
respect to its $\gamma$-filtration. 
In \cite{Joshua-1} and \cite{Joshua-2}, Joshua proved
a version of Grothendieck-Riemann-Roch theorem for stacks. In his results,
the target of the Riemann-Roch map is a version of abstractly defined 
Bousfield localized 
topological $K$-theory of the underlying stack. In particular, his Riemann-Roch
map is not an isomorphism. An equivariant Riemann-Roch connecting
equivariant algebraic and topological $K$-theory was also
studied by Thomason \cite{TO2}. 
In \cite{Toen}, Toen proved a version of the 
Grothendieck Riemann-Roch theorem for quotient Deligne-Mumford stacks
with quasi-projective coarse moduli space. In Toen's Riemann-Roch 
theorem, the target of the Riemann-Roch map is a generalized cohomology 
theory with coefficients in the sheaves of representations such as, the 
{\'e}tale and de Rham hypercohomology. In particular, the Riemann-Roch map
is not an isomorphism in this case too. To authors' knowledge, a Riemann-Roch 
theorem connecting the $K$-theory of stacks with Chow groups first
appeared in the work of Edidin and Graham \cite{EG2}. 
Our results provide the most general
version of the Riemann-Roch theorem of Edidin and Graham.

%\vskip .3cm

\subsection{Outline of the proofs}\label{sec:Outline}
Since the proof of \thmref{thm:GRR-M} is significantly involved,
we try here to give a brief sketch of its main steps. Our hope is that this
will help the reader keep his/her focus on the final proof and not get
intimidated by the several intermediate steps, which are mostly
of implementational nature.  

The proofs of our Riemann-Roch theorems are deduced from
\thmref{thm:Intro-3}. So the most of this paper is
devoted to the proof of this theorem. 
An outline of its proof is as follows.
Using the fact that we are dealing with quotient 
stacks, we first 
break the map $f:\sX \to \sY$ into the product $f = f_2 \circ f_1$ of two
maps of the following kinds.
\begin{enumerate}
\item
We find a separated quotient Deligne-Mumford stack  $\sX'$ with action by 
an algebraic group $F$ such that the coarse moduli space map 
$p_1: \sX' \to X'$ is $F$-equivariant. The map $f_1: \sX = [\sX'/F] \to [X'/F]$
is then the `stacky' coarse moduli space map or relative moduli space  (see \cite[\S~3]{AOV}, \lemref{lem:GRR-QDM-stacks}).
\item
There is a $F$-equivariant map $p_1: X' \to Y'$ of quasi-projective
schemes such that 
$f_2: [X'/F] \to [Y'/F] = \sY$ is the induced map between the quotient
stacks. In particular, $f_2$ is representable. 
\end{enumerate}

The Atiyah-Segal isomorphism for case (2) is proven in \S~\ref{Sec:GRR-CoaMod}.
Apart from the usage of twisting operators and Morita isomorphisms,
this case crucially relies on the Riemann-Roch theorem of \cite{AK1}.
This shows that a part of the Atiyah-Segal isomorphism actually uses a
weaker version of the Riemann-Roch theorem (see \lemref{lem:Gen-qt}). 

In order to prove (1), we write 
$\sX = [X/(H \times F)]$ so that $\sX' = [X/H]$ and $X' = X/H$. 
We then note that the canonical map $\sX' \to X'$ reduces to the
identity map on taking the coarse moduli spaces. 
This allows us to reduce to the case when
$F$ is trivial so that $f_1$ is just the coarse moduli space map.  
When $\sX = [X/G]$ and $f$ is the coarse moduli space map, we first
prove the case of the free action and then reduce to this case by finding an
equivariant finite surjective cover of $X$ where the group action is free.
The full implementation of these steps is the most subtle part of the
proof of \thmref{thm:Intro-3} and is done in 
\S~\ref{Sec:Product Of Groups}. 

A brief description of the other sections of this paper is as follows.
In \S~\ref{sec:Action}, we recollect some known facts about Deligne-Mumford 
stacks
and prove some geometric properties about them. In \S~\ref{sec:EKT}, we recall
some results about the equivariant $K$-theory for proper action. The 
results here
are mostly recollected from \cite{EG2}. 
In \S~\ref{Sec:Twisting}, we define a twisting operator on the equivariant 
$K$-theory which is used in the construction of the Atiyah-Segal
transformation. We prove some results associated to the Morita isomorphism
in equivariant $K$-theory in \S~\ref{Sec:Morita Isomorphism}.
We construct the Atiyah-Segal map in \S~\ref{sec:ASmap}.
We prove the Atiyah-Segal correspondence for separated quotient stacks
in \S~\ref{sec:RR-DM}. The final section is devoted to the proof of the
Riemann-Roch theorem using \thmref{thm:Intro-3}.

%\enlargethispage{40pt}

\section{Group actions and quotient stacks}\label{sec:Action}
In this section, we set up our assumptions and notations. We also
very briefly recall some definitions related to quotient 
Deligne-Mumford stacks.
We prove some basic properties of these stacks and group actions on algebraic
spaces. These properties will be used throughout the rest of the text.

\subsection{Assumptions and Notations}\label{sec:basicassumptions}
In this text, we work over the base field $k = \C$.
By a {\it{scheme}}, we shall mean a reduced quasi-projective scheme over $\C$
and denote this category by $\Sch_{\C}$. We let $\Sm_{\C}$ be the
full subcategory of $\Sch_{\C}$ consisting of smooth schemes.
All products in the category $\Sch_{\C}$ will be taken over $\Spec(\C)$
unless we specifically decorate them.  In this text, all stacks and
algebraic spaces will be separated and of finite type over $\C$.

A linear algebraic group $G$ is a smooth affine group scheme over $\C$. 
We assume that the action of $G$ on a scheme $X$ is always linear in the sense
that there is a $G$-equivariant ample line bundle on $X$. 
We let $\Sch^G_{\C}$ denote the category of quasi-projective $\C$-schemes
with linear $G$-action. We let $\Sm^G_{\C}$ denote the full subcategory of
$\Sch^G_{\C}$ consisting smooth $\C$-schemes.
We let $\qp^G_{\C}$ denote the category of $\C$-schemes of finite type
$X$ on which $G$ acts properly such that the quotient $X/G$ 
(in the sense of \cite[Definition~0.6(i - iii)]{Mumford}) is
quasi-projective. Note that an object of $\qp^G_{\C}$ is quasi-projective
over $\C$ with linear $G$-action (see \cite[Remark~4.3]{KR})
so that $\qp^G_{\C}$ is a full subcategory of  $\Sch^G_{\C}$.

Recall that the Riemann-Roch transformation for schemes is meaningful
only when we consider the algebraic $K$-theory and higher Chow groups
with rational coefficients. Since our constructions rely on 
\cite{EG2} and \cite{EG4}, we shall actually consider these groups
with complex coefficients and not just rational. We shall therefore assume
throughout this text that all abelian groups are replaced by their
base change by $\C$. In particular, all equivariant $K$-theory groups
$G_i(G, X)$ and all equivariant higher Chow groups $\CH^G_*(X,i)$
will actually mean $G_i(G,X)_{\C}$ and $\CH^G_*(X,i)_{\C}$,
respectively.

\subsection{Deligne-Mumford stacks and their 
coarse spaces}\label{sec:DMS}
Recall that a quotient stack is a stack $\sX$ equivalent to the
stack quotient $[X/G]$ (see \cite[Example~2.4.2]{LM}), 
where $X$ is an algebraic space and $G$ is a
linear algebraic group acting on $X$. 
The stack $[{\Spec(\C)}/G]$ is called the classifying stack of $G$,
which we shall denote by $\sB G$.

Recall that a stack $\sX$ of finite type over $\C$
is called a {\sl Deligne-Mumford stack} if
the diagonal $\Delta_{\sX}: \sX \to \sX \times_{\C} \sX$ is representable, 
quasi-compact and separated, and there is a scheme $U$ with an {\'e}tale 
surjective morphism $U \to \sX$.
The scheme $U$ is called an atlas of $\sX$. The stack $\sX$ is called separated
if $\Delta_{\sX}$ is proper. It is well known that a separated quotient stack 
over $\C$ is a Deligne-Mumford stack.

For a Deligne-Mumford stack $\sX$, we let ${\rm Et}_{\sX}$ denote
the small {\'e}tale site of $\sX$. An object of ${\rm Et}_{\sX}$ 
is a finite type $\C$-scheme $U$ with an {\'e}tale map $u:U \to \sX$. 
This map is representable (by schemes) if $\sX$ is separated.
In this text, we shall be 
dealing with Deligne-Mumford stacks which arise as follows.

\begin{thm}[Deligne-Mumford \cite{DM}]\label{thm:DM-1} 
Let $X \in \Sch_{\C}$ 
and let $G$ be a smooth affine group scheme over $\C$ acting on
$X$ such that the stabilizers of geometric points are finite. 
Then the stack quotient $[X/G]$ (see, for example, \cite[7.17]{Vis}) 
is a Deligne-Mumford stack. If the stabilizers are trivial,
then $[X/G]$ is an algebraic space. 
Furthermore, the stack is separated if and only if the action is proper.
\end{thm}

Recall (see, for instance, \cite[\S~2.2]{AV}) that for 
a Deligne-Mumford stack $\sX$ over $\C$, the coarse moduli space morphism
$\pi : \sX \to \sX_c$ is a morphism to an algebraic space $\sX_c$
such that (i) $\pi$ is initial among morphisms from $\sX$ to algebraic
spaces over $\C$, and (ii) for every algebraically closed field $k$,
the map $[\sX(k)] \to [\sX_c(k)]$ is bijective (where $[\sX(k)]$ denotes the
set of isomorphism classes of objects in the small category $\sX(k)$).
In this case, $\sX_c$ is called {\sl the} coarse moduli space of $\sX$.
We shall often use the common notation for the category
$\sX(k)$ and the set $[\sX(k)]$ if it is clear which
one is meant in a context.
We shall use the following two important facts about the coarse moduli spaces.

The first is the Keel-Mori theorem \cite{KM}, which says that
a separated Deligne-Mumford stack $\sX$ admits
a coarse moduli space $p: \sX \to \sX_c$ such that $\sX_c$ is a separated
algebraic space. This map has some extra properties. Namely,
$\pi$ is a universal homeomorphism, it is proper and quasi-finite,
it commutes with flat base change of $\sX_c$, and $\sO_{\sX_c} \to 
\pi_*(\sO_{\sX})$ is an isomorphism. If $X$ is an algebraic space over $\C$
with a proper action of a linear algebraic group $G$, then one knows 
(see, for instance, \cite[Proposition~2.11]{Vis} and
\cite[Lemma~2.3, Definition~6.8]{Rydh}) that the coarse moduli
space of the quotient stack $[X/G]$ coincides with the geometric
quotient $X/G$ (in the sense of \cite[Definition~0.6]{Mumford}).
It follows therefore from the  Keel-Mori theorem that
the geometric quotient $X/G$ exists.

The second is that the push-forward functor $\pi_*$ from the
category of coherent sheaves on $\sX$ to the analogous category on 
${\sX_c}$ is exact. In particular, it induces a push-forward map
between the $K$-theory of these categories. 
This property is a consequence of the
fact that all Deligne-Mumford stacks in characteristic zero are 
{\sl tame} (see \cite[Lemma~2.3.4]{AV}).

Recall that it is possible in general that a Deligne-Mumford stack may not
be a quotient stack. However, one does not know any example of a separated
stack over $\C$ which is not a quotient stack. Moreover,
there are definite results 
(for example, see \cite{EHKV}, \cite{Totaro}, \cite{KV}) 
which show that most of the separated Deligne-Mumford stacks that occur in 
algebraic geometry are in fact quotient stacks. More specifically,  
we have the following.

\begin{thm}\label{thm:Q-stack}
A separated stack $\sX$ over $\C$ is a quotient stack if one 
of the following holds.
\begin{enumerate}
\item
$\sX$ is smooth and its coarse moduli space is quasi-projective.
\item
$\sX$ satisfies the resolution property, i.e., every coherent sheaf is
a quotient of a locally free sheaf.
\end{enumerate}

If $\sX$ is a quotient stack  with quasi-projective coarse moduli 
space, then it is the quotient of a quasi-projective scheme
by the linear action of a linear algebraic group.
\end{thm}
\begin{proof}
The item (1) is \cite[Theorem~4.4]{KR} while (2) follows from
\cite[Theorem~2.14]{EHKV}. The last assertion is \cite[Remark~4.3]{KR}.
\end{proof}

The following equivalences of two quotient stacks is well known
(see, for instance, \cite[Remark~2.11]{EHKV}). We give a very brief sketch of 
the  proofs.

\begin{lem}\label{lem:Morita-stack}
Let $H \trianglelefteq G$ be a closed normal subgroup of a linear 
algebraic group 
$G$ and let $F = G/H$. Let $f: X \to Y$ be a $G$-equivariant morphism which is 
an $H$-torsor. Then $f$ induces an isomorphism of quotient stacks 
$\bar{f}: [X/G] \xrightarrow{\simeq} [Y/F]$. 
\end{lem}
\begin{proof}
It is enough to show the equivalence of the two functors of groupoids
on $\Sch_{\C}$. Given $B \in \Sch_{\C}$, a $B$-valued point of $[X/G]$ is 
the datum $B \xleftarrow{p} P \xrightarrow{f} X$ in which $P \in \Sch^G_{\C}$,
$f$ is $G$-equivariant and $p$ is a $G$-torsor.
Taking the quotients of $P$ and $X$ by $H \subset G$, we get a datum
$B \xleftarrow{\ov{p}} P/H \xrightarrow{\ov{f}} Y$ 
which is a $B$-valued point of $[Y/F]$. 

Conversely, given a $B$-valued point of $[Y/F]$, given by the datum
$B \xleftarrow{q} Q \xrightarrow{g} Y$, we let $P = X \times_Y Q$ and
easily check that $B \xleftarrow{p} P \xrightarrow{f} X$ gives
a $B$-valued point of $[X/G]$, where $P \to B$ is the composite
$P \to Q \to B$. It is left for the reader to check that this
provides inverse natural transformations between the stacks $[X/G]$ and
$[Y/F]$.
\end{proof}

To apply \lemref{lem:Morita-stack}, we consider the following situation.
Let $G$ be a linear algebraic group and let $H \subset G$ be a closed subgroup
which acts properly on an algebraic space $X$.
We have the $H$-equivariant closed embedding $H \times X \inj G \times X$,
where $h(g,x) = (gh^{-1}, hx)$. This gives rise to the closed embedding
of the quotients $\iota: X \xrightarrow{\simeq} H \times^H X \inj G \times^H X$,
where the first isomorphism is induced by $x \mapsto (e, x)$.
Let $G$ act on $Y:= G \times^H X$ by $g(g',x) = (gg',x)$. One checks that
$\iota: X \inj Y$ is $H$-equivariant and hence it induces the map between
the quotient stacks $\bar{\iota}: [X/H] \to [Y/H]$. Let $\iota^{GH}_X$
denote the composite map $[X/H] \to [Y/H] \to [Y/G]$.

\begin{lem}\label{lem:Morita-stack*}
With the above notations, the following hold.
\begin{enumerate}
\item
$Y \in \Sch^G_{\C}$ if $X \in \Sch^H_{\C}$ and $Y \in \Sm^G_{\C}$ if $X \in \Sm^H_{\C}$.
\item
The map $\iota^{GH}_X: [X/H] \to [Y/G]$ is an isomorphism.
\item
If $f: X' \to X$ is an $H$-equivariant map of algebraic spaces with proper
$H$-actions and $Y' = G \times^H Y$, then $\iota^{GH}_X \circ f =
f \circ \iota^{GH}_{X'}$.
\end{enumerate}
\end{lem}
\begin{proof}
Since $G/H$ is a smooth quasi-projective scheme by \cite[Theorem~6.8]{Borel},
part (1) follows from \cite[Proposition~23]{EG1} and \cite[Lemma~2.1]{AK3}.

To prove (2), let us define an 
$H \times G$ action on $G \times X$ by $(h,g)\star (g',x) = (gg'h^{-1},hx)$. 
Similarly let us define an $H \times G$ action on $X$ by 
$(h,g) \bullet x = hx$. Then the projection map $p:G \times X \to X$
is $(H \times G)$-equivariant and is an $(1_H \times G)$-torsor. 
Similarly, the map 
$q: G \times X \to G \times^H X$ is $(H \times G)$-equivariant 
and an $(H \times 1_G)$-torsor. 

Letting $H = H \times 1_G \subset H \times G$, we get the maps of stacks
\begin{equation}\label{eqn:Morita-crucial}
\xymatrix@C1pc{
[X/H] \ar[r]_-{\iota'} & [{(G \times X)}/{(H \times G)}] 
\ar[r]^-{\bar{q}} \ar@/_.5cm/[l]_-{\bar{p}} &  [Y/G]}
\end{equation}
such that $\iota^{GH}_{X} = \bar{q} \circ \iota'$ and $\bar{p} \circ \iota'$
is identity. Since the maps $\bar{p}$ and $\bar{q}$ are isomorphisms by
\lemref{lem:Morita-stack}, it follows that $\iota'$ and $\iota^{GH}_X$
are isomorphisms too.
The identity $\iota^{GH}_X \circ f = f \circ \iota^{GH}_{X'}$ follows directly from 
the proof of \lemref{lem:Morita-stack}.  
\end{proof}

\subsection{Inertia space and inertia stack}\label{sec:Inertia}
The Atiyah-Segal and Riemann-Roch transformations
for stacks involve the inertia stacks. 
We recall their definitions and prove some basic properties. 
%More details on this
%topic can be found in \cite[\S~4]{EG2}, where they are called the global
%stabilizers.
Given a group action $\mu: G \times X \to X$, the inertia space
$I_{X}$ is the algebraic space defined by the Cartesian square

\begin{equation}\label{diag:Inertiadefn}
\xymatrix@C1pc{ 
I_X \ar[r]^{\phi} \ar[d] & X \ar[d]^{\Delta_X} \\
G \times X \ar[r]_{(Id,\mu)} & X \times X,}
\end{equation}
where $\Delta_X: X \to X\times X$ is the diagonal morphism. 
It is clear that $I_X \in \Sch^G_{\C}$ if $X \in \Sch^G_{\C}$, where the
$G$-action on $I_{X}$ is induced from the above square. More specifically,
the conjugation action of $G$ on itself and its diagonal action on
$G \times X$ keeps $I_X$ a $G$-invariant closed subspace of
$G \times X$. On geometric points, this inclusion is given by
\begin{equation}\label{eqn:Inert-1}
I_X = \{(g,x) \in G \times X|g x = x\} \ \ \mbox{and} \ \ g'\cdot (g,x) =
(g'gg'^{-1}, g'x).
\end{equation} 

One checks that $\phi: I_X \to X$ is a $G$-equivariant morphism
which is finite if $G$-action on $X$ is proper (which means $(Id, \mu)$
is a proper map). Furthermore, $\phi$ makes $I_X$ an affine group scheme
over $X$ whose geometric fibers are the stabilizers of points on $X$ under
its $G$-action. 
The composite map ${\rm fix}_X: I_X \to G \times X \to G$ is $G$-equivariant 
with respect to the conjugation action of $G$ on itself and the fiber
of this map over a point $g \in G$ is the fix point locus $X^g$. 
We let $I^{\psi}_X \subset I_X$ denote ${\rm fix}^{-1}_X(\psi)$ for a
semi-simple conjugacy class $\psi \subset G$ (which is closed). 
It follows from \cite[Remark~4.2]{EG2}
that there is a decomposition $I_X = \amalg_{\psi \in \Sigma^G_X} I^{\psi}_X$
(see Definition~\ref{defn:Support}).    

If we write $\sX = [X/G]$, then the quotient stack $[{I_X}/G]$ is often
denoted by $I_{\sX}$ and is called the Inertia stack of $\sX$.
It is easy to check that for a finite group $G$, one has $I_{\sB G} \simeq
[G/G]$, where $G$ acts on itself by conjugation.

\begin{comment}
The square ~\eqref{diag:Inertiadefn} is then
equivalent to the Cartesian square of stacks

\begin{equation}\label{eqn:Inert-2}
\xymatrix@C1pc{ 
I_{\sX} \ar[r] \ar[d]  \ar@{}[rd]|{\mbox{\fbox{}}} & \sX \ar[d]^{\Delta_{\sX}} 
\\
\sX \ar[r]^-{\Delta_{\sX}}  &  \sX \times \sX .}
\end{equation}
\end{comment}

\begin{lem}\label{lem:$I_X$torsor}
Let $G = H \times F$ be a linear algebraic group acting properly on
an algebraic space $Z$ such that $f: Z \to X$ is an $F$-torsor of algebraic
spaces. Then the diagram 
\begin{equation}
\xymatrix@C1pc{
I_Z \ar[r] \ar[d] & Z \ar[d] \\
I_X \ar[r] & X}
\end{equation}
is Cartesian. In particular, $I_Z \to I_X$ is an $F$-torsor.
\end{lem}
\begin{proof}
We consider the commutative diagram
\begin{equation}\label{eqn:tor-0}
\xymatrix@C1pc{
I_Z \ar[r] \ar[d]   & I_X \ar[d] \ar[r]^{p_{I_X}} 
& I_{\sX} \ar[d] \\
Z \ar[r] &  X  \ar[r]^{p_{X}} & \sX,}
\end{equation}
where $\sX = [X/H] \simeq [Z/G]$. Under this identification, the big outer square 
and the right square are both Cartesian, essentially by definition
(see \cite[Tag 050P]{SP}).
It follows at once that the left square is Cartesian
(for example, see \cite[p.~11]{FGA}).
\end{proof}

\begin{lem}\label{lem:torsor}
Let $G$ be a linear algebraic group acting properly on an algebraic space $X$
and let $\psi \subset G$ be a semi-simple conjugacy class.
Then the following hold.
\begin{enumerate}
\item
For any $g \in \psi$, the $G$-invariant subspace $GX^g$ is closed in $X$.
\item
The map $G \times X^g \to I^{\psi}_X$, given by $(g', x) \mapsto
(g'gg'^{-1}, g'x)$, is a $Z_g$-torsor. In particular, there is a
$G$-equivariant 
isomorphism $p_g:G \times^{Z_g} X^g \xrightarrow{\simeq} I^{\psi}_X$
with respect to the $G$-actions given in \lemref{lem:Morita-stack*}
and ~\eqref{eqn:Inert-1}.
\item
The map $\mu^{\psi}: I^{\psi}_X \to X_{\psi}:= GX^g$ induced by $\phi$
in ~\eqref{diag:Inertiadefn} is finite and surjective $G$-equivariant map.
\end{enumerate}
\end{lem}
\begin{proof}
The parts (1) and (3) of the lemma are direct consequences of the fact that
the quotient stack associated to a proper action has finite inertia. 
The part (2) is the restatement of \cite[Lemma~4.3]{EG2}. 
\end{proof}

\begin{comment}
We first show that $GX^g$ is closed in $X$. We let $x' = g'x$, where
$g' \in G$ and $x \in X^g$. We then have $(g'gg'^{-1}, g'x) \in I^{\psi}_X$
and $g'gg'^{-1} \in \psi$. It follows that $GX^g \subset \phi(I^{\psi}_X)$.
Conversely, any element of $I^{\psi}_X$ can be written as $(g',x)$ such that 
$g'x = x$ and $g' \in \psi$ and therefore $\phi(g',x) = x$.
As $g' \in \psi$, we can write $g' = hgh^{-1}$ for some $h \in G$.
Since $x = g'x = hgh^{-1}x$, we get $gh^{-1}x = h^{-1}x$ so that
$h^{-1}x \in X^{g}$. Equivalently, we get $\phi(g',x) = x \in hX^g \subset GX^g$.
We have thus shown that $\phi: I^{\psi}_X \to X$ is a finite morphism
whose image is $GX^g$. We conclude that $GX^g \subset X$ is closed.

The part (2) of the lemma is the restatement of \cite[Lemma~4.3]{EG2}
and (3) follows from the proof of (1). 
\end{comment}

\enlargethispage{20pt}

\subsection{Some properties of morphisms between stacks}
\label{sec:Mor-prop}
We now prove some properties of the morphisms between Deligne-Mumford
stacks that we shall frequently use.
The first is the following folklore and elementary fact
about quotients by group actions.

\begin{lem}\label{lem:Basic}
Let $G$ be a linear algebraic group acting properly on an algebraic
space $X$. Let $H \trianglelefteq G$ be a normal subgroup with
quotient $F$. Then $F$ is a linear algebraic group which acts on
the geometric quotient $X/H$ such that the quotient map
$q : X \to X/H$ is $G$-equivariant, where $G$ acts on $X/H$ via the
quotient $G \surj F$.
\end{lem}
\begin{proof}
It is a well known fact that $F$ is a linear algebraic group. We have 
also seen above that the geometric quotient $Y: = X/H$ exists as an
algebraic space. To show that it acquires a natural $F$-action,
we let $\mu_{G, X}: G \times X \to X$ denote the action map and
let $q_1 = q \circ \mu_{G, X}: G \times X \to Y$.
 
We let $H$ act on $G \times X$ by $h \bullet (g,x) = (g, hx) =
(g, \mu_{G, X}(h,x))$.
Since $H$ is normal in $G$, one easily checks that
$q_1$ is $H$-equivariant with respect to the trivial $H$-action on $Y$.
It follows from the universal property of the quotient map that
$q_1$ uniquely factors as 
\[
G \times X \xrightarrow{{\rm id}_G \times q} G \times Y 
\simeq (G \times X)/H \xrightarrow{\mu_{G,Y}} Y.
\]
Furthermore, $\mu_{G,Y}$ yields a $G$-action on $Y$ for which
$q: X \to Y$ is $G$-equivariant and 
$\mu_{G,Y}|_{H \times Y}: H \times Y \to Y$ is the projection to $Y$.
The latter condition implies that $\mu_{G,Y}$ factors
through $G \times Y \surj F \times Y \xrightarrow{\mu_{F,Y}} Y$.
It is an elementary verification that this defines an $F$-action on $Y$
satisfying the desired properties.
\end{proof}

\begin{lem}\label{lem:quot*}
Let $G$ be a linear algebraic group acting properly on an 
algebraic space $X$. Let $H \trianglelefteq G$ be a normal 
subgroup with quotient $F$.
Let $Z = X/G$ and $W = X/H$ denote the coarse moduli spaces.  
Then the following hold. 
\begin{enumerate}
\item
The $F$-action on $W$ is proper. 
\item
If $X/G$ is quasi-projective, then $W$ is quasi-projective with a 
linear $F$-action.
\end{enumerate}
\end{lem}
\begin{proof}
As the $G$-action on $X$ is proper, it implies that the $H$-action is also 
proper. Therefore, $[X/H]$ is a separated Deligne-Mumford stack with 
$W$ its coarse moduli space. By the Keel-Mori theorem \cite{KM},
it follows that $W$ is a separated algebraic space over $\C$. 
As $G$ acts properly on $X$, we can find a finite and surjective
$G$-equivariant map $f: X' \to X$ such that $G$-acts freely on $X'$
with $G$-torsor $X' \to {X'}/G$ and the induced maps on the quotients
$\bar{f}: {X'}/G \to X/G$ is also finite and surjective
(see \cite[Proposition~10]{EG1}). 

We first show that $u:T = {X'}/H \to W$ is an $F$-equivariant,
finite and surjective map.
It is clear that the map of stacks $[{X'}/H] \to [X/H]$ is finite and
surjective. Since $[X/H]$ is a separated quotient Deligne-Mumford stack, the 
coarse moduli space map $[X/H] \to W$ is proper, surjective and quasi-finite. 
Therefore, the composite map $T = [{X'}/H] \to [X/H] \to W$ is finite and 
surjective. The $F$-equivariance of this map is clear.

To show (1), we need to show that the $F$-action map
$\Phi_W: F \times W \to W \times W$ is proper. For this, we consider the 
diagram
\begin{equation}\label{eqn:quot*-0}
\xymatrix@C1pc{
F \times T \ar[r]^-{\Phi_T} \ar[d]_{\phi} & T \times T \ar[d]^{u\times u} \\
F \times W \ar[r]^-{\Phi_W} & W \times W,}
\end{equation} 
where $\phi = (Id_F, u)$.
From what we have shown above, it follows that this diagram commutes and
the vertical arrows are finite and surjective. 
As $F$ acts freely on $T$, the map $\Phi_T$ is a closed immersion.
In particular, the composite $\Phi_W \circ \phi = (u \times u) \circ \Phi_T$
is finite. Since $\phi$ is finite and surjective, it easily follows that
$\Phi_W$ is in fact finite. 

Suppose now that $X/G$ is quasi-projective. It follows from (1) and
\thmref{thm:DM-1} that $[W/F] = [(X/H)/F]$ is a separated 
quotient Deligne-Mumford stack with quasi-projective coarse moduli space 
$X/G$. Since $W \to W/F = X/G$ is affine 
(for example, see \cite[Proposition~0.7]{Mumford}),
it follows from \cite[Remark~4.3]{KR} that $W$ must be quasi-projective
with linear $F$-action. 
\end{proof}

We next recall a result of Abramovich-Olsson-Vistoli
(see \cite[Theorem 3.1]{AOV}). 
It says that if $f: \sX \to \sY$ is a finitely presented morphism of tame
(not necessarily Deligne-Mumford) stacks with finite relative inertia, 
then there is a
factorization $\sX \xrightarrow{\pi} \sY' \xrightarrow{\ov{f}} \sY$
of $f$ such that $\pi$ is a relative coarse moduli space 
(see {\sl loc. cit.}, Definition~3.2) and $\ov{f}$ is representable.
To use this result in our setting, we need to make the factorization
of $f$ more explicit in the special case when it is a proper
morphism between separated quotient (Deligne-Mumford) stacks over $\C$.

So we let $\sX = [X/H]$ and $\sY = [Y/F]$ be two separated quotient
stacks and $f: \sX \to \sY$ a proper morphism. 
We let $\sX' = \sX \times_{\sY} Y$ and $Z = \sX' \times_{\sX} X$.
This yields a Cartesian diagram
\begin{equation}\label{eqn:Fin-coh-0-new}
\xymatrix@C1pc{
Z \ar[r]^-{s} \ar[d]_{t} & \sX' \ar[r]^-{f'} \ar[d]^{p'} & Y \ar[d]^{p} \\
X \ar[r]^-{q} & \sX \ar[r]^-{f} & \sY}
\end{equation}
such that $t$ is an $F$-torsor and $s$ is an $H$-torsor. In particular,
$Z$ is an algebraic space with $H$ and $F$-actions. 
We let $\nu:H \times Z \to Z$ and $\mu:F \times Z \to Z$ denote these
action maps. We let $\gamma = f \circ q$ and $\beta = f' \circ s$.
Note that if $X$ and $Y$ are schemes, then so is $Z$.

\begin{lem}\label{lem:GRR-QDM-stacks}
With the above notations, the following hold.
\begin{enumerate}
\item 
$F$ and $H$ act on $Z$ such that their actions commute. In
particular, $G = H \times F$ acts on $Z$. 
\item 
The action of $G$ on $Z$ is proper. 
\item 
The maps $t$ and $\beta$ are $G$-equivariant.
\item 
There is a factorization (see \cite[\S~3]{AOV})
\begin{equation}\label{eqn:st-gen*-0}
f: \sX \cong [Z/G] \surj [(Z/H)/F] \to [Y/F] = \sY 
\end{equation} 
such that the two maps are proper.  
\item
If $X \in \Sch^H_{\C}$ and $Y \in \Sch^F_{\C}$, then
$Z \in \Sch^G_{\C}$.
\item
If $X \in \qp^H_{\C}$, then
$Z \in \qp^G_{\C}$.
\end{enumerate}
\end{lem}
\begin{proof}
The first assertion is standard and known to the experts. 
We give a sketch of the
construction.
We have a commutative diagram
\begin{equation}\label{eqn:st-gen*-2}
\xymatrix@C1pc{
F \times H \times Z \ar[r]^-{\ov{v}} \ar[d] & F \times Z 
\ar[r]^-{\mu} \ar[d]_{p^F_Z} & Z \ar[d]^{t} \\
H \times Z \ar[r]^-{v} \ar[d]_{p^H_Z} & Z \ar[r]^-{t} \ar[d]_{s} & X 
\ar[d]^{q} \\
Z \ar[r]^-{s} & \sX' \ar[r]^-{p'} & \sX.}
\end{equation}

In the above diagram, the vertical arrows on the left are projections.
The top square on the right and the bottom 
square on the left are Cartesian by definition of $Z$.
If we let $\ov{v} = (id_F, v)$, then the top left square commutes.
In particular, we get maps 
$F \times H \times Z \to (H \times Z) \times_Z (F \times Z) \to H \times Z$ 
both of which are $F$-equivariant (with respect to the trivial $F$-action
on $H$ and $Z$) and the second map and the composite map are both $F$-torsors.
It follows that the top left square is Cartesian.
It can now be checked that $\theta_1 = \mu \circ \ov{v}:
F \times H \times Z \to Z$ defines a $G$-action on $Z$ (with $G = H \times F$)
which makes the above diagram commute.
As the composite vertical arrow on the left is just the projection
and $p' \circ s = q \circ t$, it 
follows that $[Z/G] = \sX$ and $Z \to \sX$ is the stack quotient map 
for the $G$-action.

We next consider the diagram
\begin{equation}\label{eqn:st-gen*-3}
\xymatrix@C1pc{
H \times F \times Z \ar[r]^-{\ov{\mu}} \ar[d] & H \times Z   
\ar[d]^-{p^H_Z} \ar[r]^-{v} & Z \ar[d]^{s} \\
F \times Z \ar[d] \ar[r]^-{\mu} & Z \ar[d]^{t} \ar[r]^-{s} &  
\sX' \ar[d]^{p'} \\
Z \ar[r]^-{t} & X \ar[r]^-{q} & \sX,}
\end{equation} 
where the left vertical arrows are the projection maps.
By a similar argument, one checks that $\theta_2 = v \circ \ov{\mu}:
H \times F \times Z \to Z$ defines a group action making 
~\eqref{eqn:st-gen*-3} commute.
Furthermore, ~\eqref{eqn:st-gen*-2} and ~\eqref{eqn:st-gen*-3} together show 
that the stack quotients $[{Z}/{(H \times F)}]$ for actions via
$\theta_1$ and $\theta_2$ both coincide with $\sX$.
It follows that $\theta_1$ and $\theta_2$ define the same action
on $Z$ and this implies that 
\[
v(h, \mu(f, z)) = v \circ \ov{\mu}(h, f, z) = \mu \circ \ov{v}(f,h,z)
= \mu(f, v(h,z))
\]
and this shows that the actions of $H$ and $F$ on $Z$ commute.
Since $\sX$ is separated, it follows from \thmref{thm:DM-1} that
the $G$-action on $Z$ is proper. 
It is also clear that the maps $t$ and $\beta = f' \circ s$ are 
$G$-equivariant. This proves (1) and (2) and (3).

By the Keel-Mori theorem, the coarse moduli space map
$[Z/H] \to Z/H$ is proper. If we let $W = Z/H$ and $\sW = [W/F]$,
it follows that the map $\sX = [Z/G] \to [W/F]$ is also proper.
Since $Y$ is an algebraic space, we have a
factorization $[Z/H] = \sX' \to W \to Y$.
Hence $\beta': W \to Y$ is the map induced on the coarse moduli spaces
by the proper map of stacks $\sX' \to Y$,
it follows that $\beta'$ is proper. This proves (4). We remark here that 
this factorization is precisely the one constructed in 
\cite[Theorem 3.1]{AOV}. To see this, let  $X'$ be as in the notation of 
{\it{loc.cit}}. 
Then it follows from the Cartesian diagram \eqref{eqn:Fin-coh-0-new} that 
$X'$ is the quotient of the coarse moduli space of $\sX'$ 
(which is $(Z/H)$) by the natural $F$-action induced on the quotient, 
which is precisely the 
factorization that is constructed in \eqref{eqn:st-gen*-0}.

We now prove (5) and (6).
It follows from ~\eqref{eqn:Fin-coh-0-new} that there is a Cartesian
square
\begin{equation}\label{eqn:st-gen*-1}
\xymatrix@C1pc{
Z \ar[r] \ar[d]_{(t, \beta)} & \sY \ar[d]^{\delta_{\sY}} \\
X \times Y \ar[r]^-{\gamma \times p} & \sY \times \sY.}
\end{equation}

Since $\sY$ is a separated Deligne-Mumford stack, its diagonal $\delta_{\sY}$
is representable, proper and quasi-finite. In particular, $(t, \beta)$ is a
representable proper and quasi-finite map of algebraic spaces. 
Hence, it is finite. 

If $X$ is $H$-quasi-projective and $Y$ is $F$-quasi-projective.
Then $X \times Y$ is $G$-quasi-projective
with respect to the coordinate-wise action. 
Let $\sL$ be a $G$-equivariant ample line bundle on $X \times Y$.
Since $Z \xrightarrow{(t, \beta)} X \times Y$ is finite,
we get $Z\in \Sch_{\C}$.
Since this map is furthermore $G$-equivariant,
we conclude from \cite[Exc.~III.5.7]{Hart} that
the pull-back of $\sL$ on $Z$ is a $G$-equivariant ample
line bundle. This proves (5).
If $X \in \qp^H_{\C}$, then
$Z/G = X/H$ is quasi-projective and hence $Z \in \qp^G_{\C}$.
This proves (6). 
\end{proof}

\begin{lem}\label{lem:properpushforward}
Let $f:\sX=[X/H] \to \sY=[Y/F]$ be a proper map of separated quotient 
stacks with coarse moduli spaces $M$ and $N$, respectively. 
Then the induced map $\tilde{f}: M \to N$ is proper. If $f$ is finite, then
so is $\tilde{f}$.
\end{lem}
\begin{proof}
The lemma is a consequence of the following basic property of separated 
algebraic spaces of finite type over $\C$.
Let $W_1 \xrightarrow{f_1} W_2 \xrightarrow{f_2} W_3$ be morphisms of 
separated algebraic spaces of finite type over $\C$ such that
$W_1$ is a scheme, $f_1$ is finite and surjective and
$f_2 \circ f_1$ is proper. Then $f_2$ is proper.
%Let $W_1 \xrightarrow{f_1} W_2 \xrightarrow{f_2} W_3$ be the morphisms of separated $\C$-schemes of finite type such that $f_2 \circ f_1$ is proper. If $f_1$ is surjective, then $f_1$ and $f_2$ are both proper. 
%see Liu's book page~104, Prop.~3.16 

We now prove the lemma using the above fact as follows.
Let $p_X: \sX \to M$ and $p_Y : \sY \to N$ denote the coarse moduli space
maps. As in \lemref{lem:quot*}, we let $g: Z \to [X/H]$ be a finite and
surjective map where $Z$ is a finite type separated $\C$-scheme. 
Let us consider the commutative diagram of stacks
\begin{equation}\label{eqn:proper-coarse}
\xymatrix@C1pc{
Z \ar[r]^-{g} \ar[dr]_{\tilde{g}} & [X/H]    \ar[r]^-{f}  \ar[d]^{p_X}  & 
[Y/F] \ar[d]^{p_Y} \\
& M  \ar[r]^-{\tilde{f}}  &  N.}
\end{equation}

Since $g$ is finite and surjective, it follows that $\tilde{g}$ is a
finite and surjective morphism of algebraic spaces.
Since ${f}$ and $p_Y$ are proper, 
it follows that $\tilde{f} \circ \tilde{g}$ is proper.
Since $\tilde{g}$ is finite and surjective, $\tilde{f}$ must be
proper. If $f$ is finite, then so is $f \circ g$. In particular,
the composite map $\tilde{f} \circ \tilde{g} = p_Y \circ f \circ g$
is a finite morphism of finite type separated algebraic spaces with
$Z$ a scheme.
Since $\tilde{g}$ is finite and surjective, $\tilde{f}$ must
be quasi-finite and proper. Hence, it is finite by 
\cite[Appendix]{LM} (see also \cite[Tag 05W7]{SP}).
\end{proof}

\subsection{Cohomological dimension of morphisms of stacks}
\label{sec:C-stack}
For a stack $\sX$, let ${\sC oh}_{\sX}$ denote the abelian category of
coherent sheaves on $\sX$. Let ${\sQ\sC oh}_{\sX}$ denote the abelian category 
of quasi-coherent sheaves on $\sX$ and let $D(\sX)$ denote its unbounded derived
category. 
Given a linear algebraic group $G$ and an algebraic space $X$ with $G$-action,
we let ${\sC oh}^G_{X}$ denote the abelian category of $G$-equivariant
coherent sheaves on $X$ and let $D^G(X)$ denote the unbounded derived
category of $G$-equivariant coherent sheaves on $X$. 
It is well known that ${\sC oh}^G_{X} \simeq {\sC oh}_{[X/G]}$.

Recall that a proper map $f: \sX \to \sY$ of Deligne-Mumford stacks
is said to have finite cohomological dimension if
the higher direct image functor ${\bf R}f_*: D(\sX) \to D(\sY)$
has finite cohomological dimension. Recall also from 
\cite[Theorem~1.2]{Olsson} that ${\bf R}^if_*(\sF)$ is a coherent sheaf on
$\sY$ if $\sF$ is coherent sheaf on $\sX$.

\begin{lem}\label{lem:Fin-coh}
Let $f: \sX \to \sY$ be a proper morphism of separated 
quotient stacks. Then $f$ has finite cohomological dimension.
\end{lem}
\begin{proof}
This is known in more general setting (see \cite[\S~2, Theorem~1.4.2]{DG}).
We give a proof here in our special case.
We write $\sX = [X/H]$ and $\sY = [Y/F]$ 
and follow the notations of ~\eqref{eqn:Fin-coh-0-new}
and \lemref{lem:GRR-QDM-stacks} in this proof.
Let $\phi: \sX' \to W$ be the coarse moduli space map
and let $\beta': W \to Y$ be the induced map on the coarse moduli spaces
such that $f' = \beta' \circ \phi$.
We have the factorization of $f$ (see ~\eqref{eqn:st-gen*-0}):
\begin{equation}\label{eqn:Fin-coh-1}
\sX = [Z/G] \xrightarrow{\bar{\phi}} \sW \xrightarrow{\bar{\beta'}} [Y/F] 
= \sY.
\end{equation}

Since $\phi: \sX' \to W$ is the coarse moduli space map, it is proper.
It follows from the fppf-descent that $\bar{\phi}$ is also proper.
Since the functor $\phi_*: {\sQ\sC oh}_{\sX'} \to {\sQ\sC oh}_W$ is exact by 
\cite[Lemma~2.3.4]{AV}, it follows again from the fppf-descent that
$\bar{\phi}_*: {\sQ\sC oh}_{\sX} \to {\sQ\sC oh}_{\sW}$ is exact. Since
$\sX$ and $\sW$ have finite diagonals, this is same as saying that
$\bar{\phi}_* = {\bf R}\bar{\phi}_*$ at the level of the derived categories
of quasi-coherent sheaves.

Next, $\beta': W \to Y$ is a finite type morphism of algebraic spaces of finite
type over $\C$. Hence, it is of finite cohomological dimension (see 
\cite[Tag 073G]{SP}).
Since $W = Y \times_{\sY} \sW$, it follows from the fppf-descent again that
$\bar{\beta'}: \sW \to \sY$ is of finite cohomological dimension.
We conclude using the Grothendieck spectral sequence that 
$f = \bar{\beta'} \circ \bar{\phi}$ has finite cohomological
dimension. 
\end{proof}

\begin{comment}
and consider the Cartesian diagram
\begin{equation}\label{eqn:Fin-coh-0}
\xymatrix@C1pc{
Z \ar[r]^-{s} \ar[d]_{t} & \sX' \ar[r]^-{f'} \ar[d]^{p'} & Y \ar[d]^{p} \\
X \ar[r]^-{q} & \sX \ar[r]^-{f} & \sY.}
\end{equation}

It follows that $t$ is an $F$-torsor and $s$ is an $H$-torsor. In particular,
$Z$ is an algebraic space with commuting $H$ and $F$-actions so that
$G = H \times F$ acts on $Z$ with $\bar{t}: [Z/G] \xrightarrow{\simeq} \sX$.
It follows from \thmref{thm:DM-1} that $G$ acts properly on $Z$.
We identify $f \circ \bar{t}: [Z/G] \to \sY$ with $f$ since $\bar{t}$ is
an isomorphism.

Since $Y$ is an algebraic space, $f' \circ s$ factors as 
$Z \to \sX' = [Z/H] \xrightarrow{\phi} Z/H \xrightarrow{g} Y$, where $F$ acts 
on $Z/H$ and $g$ is $F$-equivariant. Letting $W = Z/H$ and $\sW = [W/F]$, we
thus get a factorization of $f$ as:
\begin{equation}\label{eqn:Fin-coh-1}
\sX = [Z/G] \xrightarrow{\bar{\phi}} \sW \xrightarrow{\bar{g}} [Y/F] = \sY.
\end{equation}
\end{comment}

\section{Localization  in equivariant $K$-theory}\label{sec:EKT}
Given a linear algebraic group $G$
and an algebraic space $X$ with $G$-action, we let $G(G,X)$
denote the $K$-theory spectrum of the category of $G$-equivariant coherent 
sheaves on $X$. We let $G_i(G,X)$ denote the homotopy groups of $G(G,X)$
for $i \ge 0$ and let $G_*(G, X) = \oplus_{i \ge 0} G_i(G,X)$.
The equivalence of abelian categories ${\sC oh}^G_{X} \simeq {\sC oh}_{[X/G]}$
induces a homotopy equivalence of their $K$-theory spectra. 
This equivalence will be used throughout this text.

\subsection{Representation ring}\label{sec:Rep}
Let $R(G)$ denote the representation ring of $G$,
which in our notation is same as $K_0(G,\Spec \C)$. One knows that
$R(G)$ is a finitely generated $\C$-algebra (see \cite[Corollary~3.3]{Segal}).
In particular, it is Noetherian. The ideal $\fm_{1_G} \subset R(G)$,
defined as the kernel of the rank map $R(G) \surj \C$, is called 
the augmentation ideal of $G$. We refer to \cite[Appendix~A]{AK1} for
a list of basic properties of equivariant $K$-theory that we shall mostly 
use in this text.

For a linear algebraic group $G$ over $\C$ and any $g \in G(\C)$, we let 
$C_{G}(g)$ denote the conjugacy class of $g$ and 
$Z_g:= \sZ_G(g)$ the centralizer of $g$. A point $g \in G$ will always denote a 
closed point, unless we particularly specify it. 
A conjugacy class $\psi = C_{G}(g)$ is called a semi-simple conjugacy class 
if $g \in G$ is a semi-simple element of $G$. We recall below some well known 
facts about the structure of $R(G)$ for which the reader is referred to
\cite[\S~2.1]{EG2}.

For an affine scheme $X$, let $\C[X]$ denote its coordinate ring. 
Let $G$ be a reductive group and let $\C[G]^G$ denote the ring of 
regular functions on $G$ which are invariant under the conjugation action. 
There is a natural ring homomorphism $\chi: R(G) \to \C[G]^G$ defined by
$\chi([V]) = \chi_V$, where $\chi_V(g) = {\rm trace}(g|_V)$. 
%The following result is from \cite{EG2}.

\begin{prop}$($\cite[Propositions~2.2, 2.4]{EG2}$)$
\label{EGNAL:representationring(reductive)}
For a reductive group $G$, the map $\chi: R(G) \to \C[G]^G$, taking a 
virtual representation to its character, is an isomorphism. 
Furthermore, $\fm \subset R(G)$ is a maximal ideal if and only 
if there is a unique semi-simple conjugacy class $\psi$ in $G$ such that 
$\fm = \{v \in R(G)|~~\chi_v(g) = 0 \ \forall \ g \in \psi\}$.
\end{prop}

Let $G$ be any linear algebraic group which is not necessarily reductive.
There is a Levi decomposition $G = U \rtimes L$, where $U$ is the unipotent 
radical of $G$ and $L$ is reductive. 
If $X \in \Sch^G_{\C}$, the Morita isomorphism 
(see \S~\ref{Sec:Morita Isomorphism}) shows that there is a weak equivalence of
spectra $G(G, U \times X) \simeq G(L, X)$ and the homotopy invariance of
the equivariant $G$-theory tells us that $G(G, X) \xrightarrow{\simeq} 
G(G, U \times X)$.
It follows that the restriction map $G(G,X) \to G(L,X)$ is a weak equivalence.
The same holds for equivariant higher Chow groups too 
(see \cite[Proposition~B.6]{AK1}).
This shows that we can assume our group to be reductive in order to study the 
equivariant $K$-theory and  equivariant higher Chow groups.
The special case $R(G) \xrightarrow{\simeq} R(L)$ will be frequently used
in this text without any further explanation.

For a linear algebraic group $G$ as above
and $g \in G$, there is a ring homomorphism $\chi^g : R(G) \to \C$ given 
by $\chi^g([V]) = \chi_{[V]}(g)$. As $\chi^g(1) = 1$, this map is 
surjective. Hence $\fm_{g} := {\rm Ker}(\chi^g)$ is a maximal ideal of $R(G)$. 
Note that $\chi^g$ depends only on $\psi = C_G(g)$ and hence there is 
no ambiguity in writing  $\fm_g$ and $\fm_{\psi}$ interchangeably. 
If $G$ is not necessarily reductive, there is no identification
between $R(G)$ and $\C[G]^G$ in general. But we can still say the following.

\begin{prop}$($\cite[Proposition 2.5]{EG2}$)$
\label{prop:conjugacyclassandmaximalideal}
For a linear algebraic group $G$, the assignment $\psi \mapsto \fm_{\psi}$
yields a bijection between the semi-simple conjugacy classes in $G$ and the 
maximal ideals in $R(G)$. 
\end{prop}

For a surjective morphism of algebraic groups, the following is the
correspondence between the conjugacy classes. We refer to 
\cite[Theorem 2.4.8]{Springer} for a proof of the first part of the 
lemma and the remaining part follows easily by direct verification.

\begin{prop}\label{prop:grp-hom}
Let $f: G \rightarrow F$ be a surjective morphism of algebraic groups. 
Then $f$ takes a semi-simple (unipotent) conjugacy class to a semi-simple 
(unipotent) conjugacy class. 
%Moreover, for any semi-simple conjugacy 
%class $\psi \subset F$, $f ^{-1}(\psi)$ is a disjoint union of semi-simple 
%conjugacy classes in $G$. This union is finite if $\Ker(f)$ is finite. 
If $f_*: R(F) \rightarrow R(G)$ is the restriction map, then 
$f_*^{-1}(\fm_{\psi}) = \fm_{\phi}$ if and only if $f(\psi) = \phi$. 
%In particular, $f_*$ is faithfully flat.
If $\Ker(f)$ is finite, then the inverse image of a semi-simple conjugacy 
class in $F$ is a finite disjoint union of semi-simple conjugacy classes in $G$.
\end{prop}

For a closed immersion of algebraic groups, we have the following
correspondence between the conjugacy classes of the two groups.
We refer to \cite[Propositions~2.3, 2.6, Remark~2.7]{EG2} for a proof.

\begin{lem}\label{lem:basicclosedembedding}
Let $H \hookrightarrow G$ be a closed embedding of linear algebraic 
groups. Then the following hold.
\begin{enumerate}
\item 
The ring homomorphism $R(G) \rightarrow R(H)$ is finite. 
\item 
For a semi-simple conjugacy class $\psi \subset G$, $\psi \cap H = 
\psi_1 \amalg \dots \amalg \psi_n$ is a disjoint 
union of semi-simple conjugacy classes in $H$.
\item 
Given a semi-simple conjugacy class $\psi \subset G$ and $\phi 
\subset  H$, $\mathfrak{m}_{\psi}R(H) \subset \mathfrak{m}_{\phi}$ if 
and only if $\phi \subset \psi \cap H$.

\item 
For a semi-simple class $\psi \subset G$, $R(H)_{\mathfrak{m}
_{\psi}}$ is a semi-local ring with maximal ideals $\mathfrak{m}_{\psi_1} ,
\dots , \mathfrak{m}_{\psi_{n}}$.
\end{enumerate}
\end{lem}

\subsection{Fixed point loci for action of semi-simple elements}
\label{sec:FPS}
If a linear algebraic group $G$ acts on an algebraic space $X$ and if
$g \in G$, let $X^g \subset X$ denote the maximal closed subspace of $X$
which is fixed by $g$. In general, $X^g$ is a $Z_g$-invariant closed 
subspace of $X$. The following results (see \cite[Theorem~5.4]{VV})
show that there are only finitely many
semi-simple conjugacy classes in $G$ for which the invariant subspaces $X^g$
are non-empty, provided the action is proper.

\begin{lem}\label{lem:diagonalizablefiniteness}
Let $G$ be a diagonalizable torus acting properly on an algebraic space
$X$. Then there are only finitely many elements $g \in G$ such that 
$X^g \neq \emptyset$.
\end{lem} 
\begin{proof}
%We prove the lemma using the Noetherian induction on $X$. If $X$ is zero 
%dimensional, then $X$ is a scheme 
%(stacks algebraic spaces over a field pg 6 stacks project)%
%and  as we are assuming that the $G$-action is proper, it implies $G$ is a 
%finite group. The lemma is obvious in this case.
Since $X$ is a separated algebraic space, there 
is a largest open dense subspace which is the scheme locus of $X$. We denote it
by $V$. 
It follows from Thomason's generic slice theorem 
\cite[Proposition 4.10]{TO4} that 
there is a $G$-invariant non-empty open subset $U \subset V \subset X$ and a 
subgroup $H \subset G$ such that $H$ acts trivially 
on $U$ and $F =G/H$ acts freely on $U$ with quotient algebraic space 
$U/F = U/G$. The freeness of $F$-action on $U$ implies that 
if there is $g \in G$ such that $U^g \neq \emptyset$, then $g$ must lie in $H$.
Furthermore, the properness of $G$-action implies that
the stabilizer of any $x \in X$ is finite. It follows that $H$ is finite. 
The result thus follows for $U$. 

Since $Z = X \setminus U$ is a proper $G$-invariant 
closed subscheme of $X$,
the Noetherian induction implies that there are only finitely many 
$g \in G$ such that  $Z^g \neq \emptyset$.
As $X^g \neq \emptyset$ if and only if either $U^g \neq \emptyset$ or 
$Z^g \neq \emptyset$, the lemma follows. 
\end{proof}

\begin{prop}\label{prop:finite conjugacy class}
Let $G$ be a linear algebraic group acting properly on an algebraic space   
$X$. Then there exists a finite set of semi-simple conjugacy classes 
$\Sigma^G_X = \{\psi_1, \cdots , \psi_n \}$ in $G$ such that $X^g \neq 
\emptyset$ if and only if $g \in \psi_i$ for some $1 \leq i \leq n$.
\end{prop}
\begin{proof}
If $G$ is a diagonalizable torus, then the proposition follows from 
\lemref{lem:diagonalizablefiniteness}. 
Let us assume that $G$ is  connected but not necessarily diagonalizable. 
Let $T$ be a fixed  maximal torus of $G$. Then the $T$-action on $X$
is also proper. 
Therefore, it follows from \lemref{lem:diagonalizablefiniteness} that
there are finitely many $g_1, \cdots , g_n \in T$ such that 
$X^{g_i} \neq \emptyset$.  As each 
$g_i \in T$, it is semi-simple and hence $\psi_i :=  C_{G}(g_i)$ is a 
semi-simple conjugacy class in $G$.  We let $\psi_i = C_{G}(g_i)$ and
$\Sigma^G_X = \{\psi_1, \cdots , \psi_n \}$. Note that $\Sigma^G_X$
is never empty because it always contains the conjugacy class
of the identity element.

Suppose now that $g \in G$ is such that $X^g \neq \emptyset$. The properness 
of the $G$-action implies that $g$ must be an element of 
finite order. In particular, $g$ must be a semi-simple element of $G$ and 
hence must belong to a maximal torus of $G$. Since all maximal tori of $G$ 
are conjugate (for example, see \cite[Theorem 6.4.1]{Springer}), 
there is $h \in T$ 
such that $g \in C_{G}(h)$. It is then clear that $X^h \neq \emptyset$
and hence $h \in \psi_i$ for some $1 \le i \le n$. But this implies that
$g \in \psi_i$. This proves the proposition when $G$ is connected.

If $G$ is not necessarily connected, then any element $g \in G$ such that 
$X^g \neq \emptyset$ may not lie in a maximal torus and so the above
proof breaks. In this case, we reduce to the previous case as follows.
We choose a closed embedding $G \subseteq GL_n$ and 
let $Y = X \times^G GL_n$, where $G$ acts on $X \times GL_n$ by 
$g(x, g') = (gx,g'g^{-1})$ and $GL_n$ acts on $Y$ which is
induced by the action $h(x,g') = (x , hg')$. 
It follows from \lemref{lem:Morita-stack} that $[X/G] \simeq [Y/GL_n]$.
Since $G$ acts properly on $X$, it follows from \thmref{thm:DM-1}
that $GL_n$ acts properly on $Y$.

We prove the proposition by contradiction. Suppose that
there are infinitely many semi-simple conjugacy classes 
$S = \{{\psi_i}\}$ in $G$ such that $X^g \neq \emptyset$ for $g \in \psi_i$. 
Let $g_i \in \psi_i$ be representatives of the conjugacy classes and  
set $\psi'_i = C_{GL_n}(g^{-1}_i)$. It follows from 
\lemref{lem:basicclosedembedding} that $S' =  \{\psi_i' | \psi_i \in S\}$
is an infinite set.
The proposition will now follow from the case of connected groups
if we can show that $Y^{g^{-1}} \neq \emptyset$ whenever $g \in \psi_i$
and $\psi_i \in S$. 
 
So fix any $\psi_i \in S$ and $g \in \psi_i$. Since $X^g \neq \emptyset$,
we choose a closed point $x \in X^g$ and let
$y$ denote the image of $(x, e)$ under the quotient map
$\pi: X \times GL_n \to Y$. We then have
\[
g^{-1}y = \pi(\{(hx, g^{-1}h^{-1})| h \in G\})   
=  \pi(\{(hgx, g^{-1}h^{-1})| h \in G\})  = 
\pi(\{(hx, h^{-1})| h \in G\}) = y,\]
where the second equality follows from the fact that $x \in X^g$.
It follows that $Y^{g^{-1}} \neq \emptyset$ and we are done.  
\end{proof}

\subsection{Support of equivariant $K$-theory for proper action}
\label{sec:Supp-pr}
\begin{defn}\label{defn:Support}
Given an algebraic space $X$ with proper action by an algebraic group $G$,
we let $\Sigma^G_X$ denote the set of semi-simple
conjugacy classes in $G$ such that $X^g \neq \emptyset$ if and only if 
$g \in \psi$ for some $\psi \in \Sigma^G_X$. 
It follows from Proposition \ref{prop:finite conjugacy class} that 
$\Sigma^G_X$ is a non-empty finite set.  We let $S_G$ denote the
set of semi-simple conjugacy classes in $G$. 
\end{defn}
%\nolinebreak
%\begin{lem}$($\cite[Proposition~5.1]{EG5}, \cite[Lemma~7.3]{AK1}$)$
The following two results are generalizations of
\cite[Proposition~5.1, Remark~5.1]{EG5} to higher $K$-theory.

\begin{lem}$($\cite[Remark~5.1]{EG5}$)$\label{lem:FiniteJDiagonalizable}
Let $G$ be a diagonalizable torus. Given any $G$-space $X$ 
with proper action, there is an ideal $J \subset R(G)$ such that $R(G)/
J$ has finite support containing $\Sigma_X^G$ and $JG_i(G,X) = 0$ for every 
$i \geq 0 $.
\end{lem}
\begin{proof}
By Thomason's generic slice theorem, there exists a non-empty $G$-invariant
smooth open dense subspace $U \subset X$ which is an affine scheme, and
there is a finite subgroup $H \subset G$ such that $F = G/H$ acts freely on 
$U$ with quotient algebraic space $U/G \simeq U/F$. As $F$ acts 
freely on $U$, we have $G_i(F,U) \simeq G_i(U/F)$. Since the augmentation ideal 
($\Ker(K_0(U/F) \to \C)$) of $K_0(U/F)$ is nilpotent and since each $G_i(U/F)$
is an $K_0(U/F)$-module, it follows that there exists an integer $n \gg 0$
(depending only on $K_0(U/F)$) 
such that $\fm^n_{1_F}G_i(U/F) = 0$  for all $i \ge 0$.

We have a short exact sequence of character groups
\[
0 \to T(F) \to T(G) \to T(H) \to 0
\]
and since the representation ring of a diagonalizable group over $\C$
is the group ring over its character group, it follows that
the map $R(G)/{\fm_{1_F}R(G)} \to R(H)$ is an isomorphism.
In particular, $J_1 := \fm^n_{1_F}R(G)$ is an ideal in $R(G) \simeq 
\C[t^{\pm 1}_1, \cdots , t^{\pm 1}_r]$ (where $r = {\rm rank}(G)$) whose
support is $H$.

Since the maps $R(G) \otimes_{R(F)} G_i(F, U) \xrightarrow{\simeq}
R(G) \otimes_{R(F)} G_i(U/F) \xrightarrow{\simeq} G_i(G, U)$
are isomorphisms by \cite[Lemma~5.6]{TO2}, we conclude that
$J_1G_i(G,U) = 0$ for all $i \ge 0$.
Note that as $F$ acts freely on $U$, we must have $\Sigma^G_X \subseteq H$.
So the lemma is proven for $U$.

We now finish the proof of the lemma using the Noetherian induction. 
This induction shows that there is an ideal 
$J_2 \subset R(G)$ whose support is a finite subset of $G$ containing 
$\Sigma_{X \setminus U}^G$ such that $J_2 G_i(G,X \setminus U) = 0$ for all 
$i \ge 0$. 
It follows that $J = J_1 J_2$ annihilates $G_i(G,X \setminus U)$ and 
$G_i(G,U)$ for all $i \ge 0$. 

The support of $J$ is a finite set containing $\Sigma_X^G$ 
as it is the union of $\Sigma_{X \setminus U}^G$ and $\Sigma_U^G$. 
All we are left to show is $J G_i(G,X) = 0$ for all $i \ge 0$.
But this follows immediately from the localization exact sequence 
$G_i(G,X \setminus U) \to G_i(G,X) \to G_i(G,U)$ 
(see \cite[Proposition~A.1]{AK1}). 
\end{proof}

\begin{lem}$($\cite[Remark~5.1]{EG5}$)$\label{lem:FiniteJGeneral}
Let $G$ be a linear algebraic group acting properly on an algebraic space $X$. 
Then there is an ideal $J \subset R(G)$ such that $R(G)/J$ has finite 
support containing $\Sigma_X^G$ and $JG_i(G,X)$ = 0 for every $i \geq 0$.
\end{lem}
\begin{proof}
If $G$ is a torus it follows from \lemref{lem:FiniteJDiagonalizable}. 
We now assume $G = GL_n$ and let $T$ be a maximal torus of $G$. 
It follows from \lemref{lem:FiniteJDiagonalizable} that there is an ideal 
$I \subset R(T)$ which satisfies the claim of the lemma for the 
$T$-action on $X$. 

Since $R(G) \inj R(T)$ is a finite map, it follows that 
$J = I \cap R(G)$ is an ideal $R(G)$ with finite  support. As
shown in Proposition~\ref{prop:finite conjugacy class},
the image of $\Sigma^T_X$ is $\Sigma_X^G$ under the map $\Spec(R(T)) \to 
\Spec(R(G))$. 
The split monomorphism of $R(G)$ modules $G_i(G,X) \inj G_i(T,X)$, 
induced by the inclusion $T \subset G$ (see \cite[(1.9)]{TO4}),
now shows that $JG_i(G,X) = 0$ for all $i \ge 0$. This proves the lemma 
for $GL_n$. 

To prove the general case, we embed $G \inj GL_n$ and set 
$Y = X \times^G GL_n$, where the $G$ and $GL_n$ actions on $X \times GL_n$ are
described in the proof of Proposition~\ref{prop:finite conjugacy class}. 
Let 
\[
p: X \times GL_n \to X, q:GL_n \to pt, \pi_G: {GL_n}/G \to pt, 
\pi_X : X \to pt, \ r: X \times GL_n \to GL_n
\]
be the obvious projection maps. Then we see that
for any $G$-representation $V$, one has $q^*(V) = V \stackrel{G}{\times} GL_n$.
In particular, for any $GL_n$-representation $W$, we get
\[
q^* \circ {\rm res}(W) = W \stackrel{G}{\times} GL_n = W \times ({GL_n}/G)
= \pi^*_G(W).
\]
In particular, the left triangle in the diagram
\begin{equation}\label{eqn:decom-Gen1}
\xymatrix@C2pc{
R(GL_n) \ar[r]^{\rm res} \ar[rd]_{\pi^*_G} & R(G) \ar[d]^{q^*}
\ar[r]^{\pi^*_X} & K_0(G,X) \ar[d]^{p^*} \\
& K_0(GL_n,{GL_n}/G) \ar[r]_>>>>{r^*} & K_0(GL_n,Y)}
\end{equation}
commutes. Since the right square commutes by the functoriality of the 
Morita isomorphism (see \S\ref{Sec:Morita Isomorphism}), we see that 
the $R(GL_n)$-module structure on $G^{GL_n}_*(Y)$ is same as the restriction
of its $R(G)$-module structure acquired via the Morita isomorphism
$p^*: G_*(GL_n,Y) \xrightarrow{\simeq} G_*(G,X)$. 
Using this compatibility and the proof of the lemma for $GL_n$,
we conclude that there is an ideal $I \subset R(GL_n)$ with finite support
containing $\Sigma^{GL_n}_Y$ such that $IG_*(GL_n, Y) = JG_*(G,X) = 0$,
where $J:=  IR(G) \subset R(G)$.

Since $R(GL_n) \to R(G)$ is finite by \lemref{lem:basicclosedembedding}, 
we see that $J$ has finite support containing the 
inverse image of $\Sigma^{GL_n}_Y$. Moreover, we have shown in the proof 
of Proposition~\ref{prop:finite conjugacy class}
that $\Sigma^G_X$ is contained in the inverse image of $\Sigma^{GL_n}_Y$.
This finishes the proof.
\end{proof}

Let $G$ act properly on an algebraic space $X$. For any semi-simple
conjugacy class $\psi \subset G$, let $j^{\psi}:X_{\psi} \inj X$ denote the
image of the map $\phi: I^{\psi}_X \to X$ (see ~\eqref{diag:Inertiadefn}).
We know from \lemref{lem:torsor} that $X_{\psi} \subset X$ is a closed
$G$-invariant closed subspace. 
We shall repeatedly use the following localization theorem in the rest of the 
text.

\begin{thm}$($\cite[Theorem 3.3]{EG2}$)$\label{thm:closedimmiso}
Let $G$ be a linear algebraic group acting properly on an algebraic space $X$. 
Then for any semi-simple conjugacy class $\psi \in G$ and $i \ge 0$, 
the push-forward map $j^{\psi}_* : G_i(G,X_{\psi}) \rightarrow G_i(G,X)$ 
induces  an isomorphism of $R(G)$-modules
\begin{equation}
G_i(G,X_{\psi})_{\fm_{\psi}} \rightarrow  G_i(G,X)_{\fm_{\psi}}.
\end{equation}
\end{thm}

In \cite{EG2}, $X_{\psi}$ is defined as the closure of $GX^g$ in $X$.
However, when $G$ acts with finite inertia (e.g., when the
action is proper), then $GX^g$ is already closed
(see \lemref{lem:torsor}).

The following is a direct consequence of \cite[Proposition~3.6]{EG2},
\lemref{lem:torsor} and \thmref{thm:closedimmiso}.

\begin{prop}\label{prop:direcsumdecomp}
Let $G$ be a linear algebraic group acting properly on an algebraic space $X$. 
Then for every $i \geq 0$, the natural map 
\begin{equation}
G_i(G,X) \rightarrow \oplus_{\psi \in S_G} G_i(G,X)_{\mathfrak{m}
_{\psi}} = \oplus_{{\psi \in \Sigma_X^G}} G_i(G,X)_{\mathfrak{m}
_{\psi}}
\end{equation}
is an isomorphism of $R(G)$-modules.
\end{prop}

\begin{prop}\label{prop:decom-surj}
Let $G$ act properly on an algebraic space $X$. The following hold.
\begin{enumerate}
\item
Given a closed subgroup $H \subset G$ and a semi-simple conjugacy class
$\psi \subset G$ with $\psi \cap H = \{\psi_1, \cdots , \psi_r\}$,
we have 
\begin{equation}\label{eqn:decom-inj}
G_*(H, X)_{\fm_{\psi}} \xrightarrow{\simeq} 
\stackrel{r}{\underset{i = 1}\oplus} G_*(H, X)_{\fm_{\psi_i}},
\end{equation}
where the localization on the left and the right sides are with respect to the
$R(G)$-module (via the map $R(G) \to R(H)$) 
and $R(H)$-module structures on $G_*(H,X)$, respectively.
The map in ~\eqref{eqn:decom-inj} is product of various localizations.
\item 
Given an epimorphism  $u:G \surj F$, the map $R(F) \to R(G)$
induces an isomorphism of $R(F)$-modules
\begin{equation}\label{eqn:decom-surj-0}
u_*: G_*(G,X) \xrightarrow{\simeq} \oplus_{\phi \in S_F} 
G_*(G, X)_{\fm_{\phi}}.
\end{equation}
\end{enumerate}
\end{prop}
\begin{proof}
The first part follows directly from \propref{prop:direcsumdecomp} 
in combination with \lemref{lem:basicclosedembedding},
which says that $R(H)_{\fm_{\psi}}$ is the semi-local ring with 
maximal ideals $\{\fm_{\psi_1}, \cdots , \fm_{\psi_r}\}$.

To prove (2), we fix a $\phi \in S_F$ and consider the commutative diagram of 
$R(F)$-modules
\begin{equation}\label{eqn:decom-surj-1}
\xymatrix@C1pc{
G_*(G,X) \ar[r] \ar[d] & \oplus_{\psi \in S_G} G_*(G,X)_{\fm_{\psi}} \ar[d] \\
G_*(G,X)_{\fm_{\phi}} \ar[r]   & 
(\oplus_{\psi \in S_G} G_*(G,X)_{\fm_{\psi}})_{\fm_{\phi}},}
\end{equation}
where the bottom row is the localization of the top row at the maximal ideal 
$\fm_{\phi}$ of $R(F)$.
We now show that the bottom row in ~\eqref{eqn:decom-surj-1} 
induces an isomorphism
\begin{equation}\label{eqn:decom-surj-2}
G_*(G,X)_{\fm_{\phi}} \xrightarrow{\simeq} \oplus_{u(\psi) = \phi} 
G_*(G,X)_{\fm_{\psi}}.
\end{equation}

Suppose first $\psi \in S_G$ is such that $u(\psi) = \phi$.
Then $\fm_{\phi} = u^{-1}_*(\fm_{\psi})$ by \propref{prop:grp-hom},
and therefore the map $G_*(G,X) \to G_*(G,X)_{\fm_{\psi}}$ factors through 
$G_i(G,X)_{\fm_{\phi}}$. This implies in particular that
$(G_*(G,X)_{\fm_{\psi}})_{\fm_{\phi}} = G_*(G,X)_{\fm_{\phi}}$.

Suppose next that $u(\psi) \neq \phi$. To show that 
$(G_*(G,X)_{\fm_{\psi}})_{\fm_{\phi}} = 0$, we can assume $\psi \in \Sigma^G_X$.
Since $u(\psi) \neq \phi$,  \propref{prop:grp-hom}
says that there exists $a \in \fm_{\psi} \setminus \fm_{\phi}$ , which 
must act invertibly on $G_*(G, X)_{\fm_{\phi}}$ and hence
on $(G_*(G,X)_{\fm_{\psi}})_{\fm_{\phi}} = (G_*(G,X)_{\fm_{\phi}})_{\fm_{\psi}}$.
On the other hand, Lemma~\ref{lem:FiniteJGeneral} says that there is an ideal 
$J \subset \fm_{\psi} \subset R(G)$ such that $JG_*(G,X) = 0$ and 
$J$ has finite support. Hence, $\fm_{\psi}$ acts nilpotently on 
$G_*(G,X)_{\fm_{\psi}}$ and hence on $(G_*(G,X)_{\fm_{\psi}})_{\fm_{\phi}}$.
But this implies that this module must be zero.
\end{proof}

\subsection{The functor of invariants}\label{sec:Inv}
Let $G$ be a linear algebraic group acting properly on an algebraic space 
$X$. Let $p: X \to X/G$ be the quotient map. 
For a $G$-equivariant coherent sheaf $\sF$ on $X$, one knows
that $p_*(\sF)$ is a coherent sheaf on $X/G$ with $G$-action.
Moreover, the subsheaf of $G$-invariant sections 
$(p_*(\sF))^G \subset p_*(\sF)$ is a coherent sheaf on $X/G$. It is shown
in \cite[Lemma~6.2]{EG2} that this is an exact functor.
Our goal here is to prove a generalization of this construction
in the setting of higher $K$-theory. We prove the following.

\begin{thm}\label{thm:invariants}
Let $H \trianglelefteq G$ be a normal subgroup with quotient $F$.
Let $G$ act properly on an algebraic space $X$ and let $W = X/H$ be
the quotient for the $H$-action.
Then, there is a functor of `$H$-invariants' which induces an $R(F)$-linear map 
\begin{equation}\label{eqn:invarinats-0}
{\rm Inv}^H_X : G_*(G,X) \rightarrow G_*(F,W).
\end{equation}
Given a $G$-equivariant proper map $u: X' \rightarrow X$ with $W' = {X'}/H$,
there is a commutative diagram 
\begin{equation}\label{eqn:invariants-1}
\xymatrix@C1pc{
G_*(G,X') \ar[r]^{u_*} \ar[d]_{{\rm Inv}^H_{X'}} & G_*(G,X) \ar[d]^{{\rm Inv}^H_X} \\
G_*(F,W') \ar[r]_{\bar{u}_*} & G_*(F,W).}
\end{equation}
\end{thm}
\begin{proof}
Using the Keel-Mori theorem and \lemref{lem:Basic}, we first observe that 
$W$ and $W'$ are separated algebraic spaces with $F$-actions
and there are proper and quasi-finite maps
$f^{X,H}: [X/H] \to W$ and $f^{X',H}: [{X'}/H] \to W'$ which are $F$-equivariant.

We first consider the case when $H = G$ as in \cite[\S~6.1]{EG2}. 
Since $[X/G]$ is a separated Deligne-Mumford stack over $\C$, it follows from
\cite[Lemma~2.3.4]{AV} that $f^{X, G}_*$ is an exact functor that takes coherent 
sheaves on $[X/G]$ to coherent sheaves on $W$. 
This induces a proper push-forward map $f^{X,G}_*: G_*(G,X) \simeq 
G_*([X/G]) \to G_*(W)$. We denote this map by ${\rm Inv}^G_X$.

In the general case, we have the proper maps of separated Deligne-Mumford stacks
\begin{equation}\label{eqn:invariants-2}
[X/G] \xrightarrow{g} [W/F] \xrightarrow{\ov{g}} W/F = X/G.
\end{equation}

Using \lemref{lem:Fin-coh}, the Thomason-Trobaugh (see \cite[\S~3.16.1]{TT})
construction of the push-forward map on $K$-theory gives us the maps
$G([X/G]) \xrightarrow{g_*} G([W/F]) \xrightarrow{\bar{g}_*} G(X/G)$,
where the first map is $K_0([W/F])$-linear and the second map is
$K_0(W/F)$-linear. In particular, $g_*$ is $R(F)$-linear.
We let $g_*$ be denoted by ${\rm Inv}^H_X$.
The second part of the theorem follows from the first part and the 
commutative diagram of proper maps
\begin{equation}\label{eqn:invariants-3}
\xymatrix@C1pc{
[X'/G] \ar[r]^-{u} \ar[d]_{g'} & [X/G] \ar[d]^{g} \\
[W'/F]  \ar[r]^-{\bar{u}_*}  &  [W/F]}
\end{equation}
and the covariant functoriality of the push-forward maps
$(g \circ u)_* = g_* \circ u_*$ at the level of $K$-theory spectra.
\end{proof}

\begin{remk}\label{remk:Inv-iso}
If $G$-action on $X$ is not necessarily proper but
the $H$-action is free, then the morphism of stacks
$g: [X/G] \to [W/F]$ is an isomorphism.
This yields an equivalence of abelian categories
$g_*: {\sC oh}([X/G]) \xrightarrow{\simeq} {\sC oh}([W/F])$
and hence an isomorphism
$g_* = {\rm Inv}^H_X: G([X/G]) \xrightarrow{\simeq} G([W/F])$.
\end{remk}

\section{Twisting in  equivariant $K$-theory}
\label{Sec:Twisting}
This section is the starting point of our construction of the
Atiyah-Segal correspondence.
Here, we define the twisting action on  higher 
equivariant $K$-theory. This action was introduced for $G_0$ in 
\cite[\S~6.2]{EG2}. We shall generalize it to higher equivariant $K$-theory
by constructing twisting type functors
at the level of the exact categories of sheaves.
This will play a crucial role in the
construction of the Atiyah-Segal map in \S~\ref{sec:ASmap}.

Let $Q$ be a linear algebraic group acting properly on a separated 
algebraic space $T$. Let $P \subset Q$ be a finite central subgroup
which acts trivially on $T$. Note that $P$ is then a finite abelian
group. In particular, $P$ is semi-simple.
For any $p \in P$, we want to define an automorphism
$t_p: G_*(Q,T) \to G_*(Q,T)$. Notice that our notations here for the group
and the algebraic space deviate from
the ones used in the previous sections. The reason for this deviation
can be seen in the following situation where we are going to 
apply the twisting action.

%The situation in which we shall be interested in applying the twisting action 
%is as follows. 
Let $G$ act properly on an algebraic space $X$. 
Let $g \in G$ be a closed point of finite order. 
Then $X^g$ is not a $G$-invariant closed subspace 
but it naturally is  $Z_g$-invariant. In this situation, one would 
like to define a twisting action on $G_*(Z_g,X^g)$. In the above notation,
it would translate as $Q := Z_g$, $P = \<g\>$, the finite closed subgroup
generated by $g$. As $g$ is a semi-simple element 
of finite order, $P = \<g\> \subset Z_g$ is a finite closed central subgroup of 
$Z_g$ which acts trivially on $X^g$.

We now return to the notations of the first paragraph. Since $p \in P$ is
a semi-simple element of $Q$, 
Proposition~\ref{prop:conjugacyclassandmaximalideal} says that there is a 
unique maximal ideal in $R(Q)$ corresponding to $p$ which we denote by 
$\fm_p$. We are interested in defining a twisting action $t_p$ on
$G_*(Q,T)$  such that this is an isomorphism and takes the summand 
$G_*(Q,T)_{\fm_{p^{-1}}}$ to $G_i(Q,T)_{\fm_1}$.

\subsection{Decomposing the category of $Q$-equivariant coherent 
sheaves}\label{sec:Decomp-t}
As $P$ acts trivially on $T$, the action of $P$ on a $Q$-equivariant 
coherent sheaf $\mathcal{F}$ is fiber-wise over the points of $T$. 
Equivalently, this action is given in terms of
a group homomorphism $P \to {\rm Aut}_T(\sF)$.
Further note that every {\'e}tale open subset of $T$ can be 
considered as a $P$-invariant open subset with the trivial action of $P$.
We let $\wh{P}$ denote the dual of the finite abelian group $P$ 
(the group of characters of $P$).
For each $\chi \in \wh{P}$, we let $\sC^{\chi}_T$ be the 
full subcategory of $Q$-equivariant coherent sheaves defined by
\begin{multline}\label{eqn:Decom-0}
Obj(\mathcal{C}^{\chi}_T) =\{\mathcal{F} \in {\sC oh}^Q_T| \
h.f =\chi(h)f  \ \forall \ U \in  {\rm Et}_T,  f \in \mathcal{F}(U), 
h \in P\}.
\end{multline}

Given $\sF \in {\sC oh}^Q_T$ and $\chi \in \wh{P}$, we let $\sF_{\chi}$
be the subsheaf of $\sF$ defined by 
\begin{equation}\label{eqn:Decom-1}
\sF_{\chi}(U) : = \{f \in \sF(U) | h.f = \chi(h)f \ \forall \ h \in P\}.
\end{equation} 

\begin{lem}\label{lem:definitionoftwist}
With the above notations, the following hold.
\begin{enumerate}
\item
For each $\sF \in {\sC oh}^Q_T$ and $\chi \in \wh{P}$, 
the subsheaf $\sF_{\chi}$ is coherent and $Q$-equivariant.
\item
The natural map ${\underset{\chi \in \wh{P}}\oplus} \sF_{\chi} \to \sF$
is an isomorphism in ${\sC oh}^Q_T$.
\item
$\mathcal{C}^{\chi}_T$ is a full abelian subcategory of ${\sC oh}^Q_T$.
\item
The natural inclusions $\sC^{\chi}_T \inj {\sC oh}^Q_T$
induce an equivalence of categories 
\[
{\underset{\chi \in \wh{P}}\amalg} 
\sC^{\chi}_T \xrightarrow{\simeq} {\sC oh}^Q_T.
\]
\end{enumerate}

In particular, the natural map of spectra ${\underset{\chi \in \wh{P}}\amalg}
K(\sC^{\chi}_T) \to G(Q,T)$ is a homotopy equivalence,
so that $G_i(Q,T) \simeq {\underset{\chi \in \wh{P}}\oplus} K_i(\sC^{\chi}_T)$.
\end{lem}
\begin{proof}
Since $P$ acts trivially on $T$, we can assume that $T$ is affine in order to
prove that $\sF_{\chi}$ is coherent. If we let $T = \Spec(A)$, it suffices to
show that $\sF_{\chi}$ is an $A$-submodule of $\sF$ (since $A$ is
Noetherian). But this is immediate from ~\eqref{eqn:Decom-1}.
To prove that $\sF_{\chi}$ is $Q$-equivariant, it suffices to show that 
for every open $U \subset T$ and every $q \in Q$, the isomorphism
$q^*: \sF(U) \to \sF(qU)$ (induced by the $Q$-action on $\sF$)
preserves $\sF_{\chi}$. That is, we need to show that for every $h \in P, \
\chi \in \wh{P}$ and $f \in \sF_{\chi}(U)$, the equality $h\cdot q^*(f) = 
\chi(h) q^*(f)$ holds. But
\[
\begin{array}{lllll}
\chi(h) q^*(f) & {=}^{1} & q^*(\chi(h)f) & {=}^2 & q^*(h^*(f)) \\
& = & (qh)^*(f) & {=}^3 & (hq)^*(f) \\
& = & h^*(q^*(f)) & = & h \cdot q^*(f),
\end{array}
\]
where ${=}^1$ holds because $q^*$ is $\C$-linear, ${=}^2$ holds because
$f \in \sF_{\chi}(U)$ and ${=}^3$ holds because
$P$ is central in $Q$. We have thus proven (1).

We now prove (2). Since each $\sF_{\chi} \subset \sF$ is $Q$-equivariant,
it suffices to show the decomposition as $P$-modules. But this
is a direct consequence of the diagonalizability of $P$ and triviality of its
action on $T$ (see \cite[\S~6.2]{EG2}).
We briefly outline its proof. Since $P$ acts trivially on $T$, it suffices
to show that the map ${\underset{\chi \in \wh{P}}\oplus} \sF_{\chi} \to \sF$
is an isomorphism on affine open subsets of $T$. We can thus assume
$T$ is affine. Since $P$ is finite abelian, it is diagonalizable. The
isomorphism then follows from \cite[Lemma~5.6]{TO2}.

Since each $\sC^{\chi}_T$ is a full subcategory of ${\sC oh}^Q_T$ by 
definition, we need to show only that it is closed
under taking kernels and cokernels in order to prove (3). For this, we need
to show that the kernel and cokernel of a map $\sF \to \sG$ in ${\sC}^{\chi}_T$
also lie in ${\sC}^{\chi}_T$.  We can prove this also by
considering $\sF$ and $\sG$ as $P$-equivariant coherent sheaves. 
In this case, it is enough to check this at each affine open in $T$, 
where one can check directly. 

We now prove the first part of (4). 
As each $\sC^{\chi}_T$ is a full abelian subcategory, 
all we have to show is that the inclusion functors 
induce the desired equivalence of categories. From (2), it follows that the 
functor $\oplus_{\chi \in \hat{P}} i_{\chi}$ is essentially surjective, 
where $i_{\chi}: \sC^{\chi}_T \inj {\sC oh}^Q_T$ is the inclusion.
To show it is fully faithful, we can work in the category of 
$P$-equivariant coherent sheaves. 
But $P$ is diagonalizable and acts trivially on $T$.
We can thus restrict to affine open subsets of $T$. In this case,
the assertion follows from \cite[Lemma~5.6]{TO2}.

The second part of (4) follows from its first part and
\cite[\S~1, (4), \S~2, (8)]{Quillen} since $P$ is finite.
\end{proof}

%\enlargethispage{40pt}

\subsection{The twisting map and its properties}\label{sec:Twist}
Let $P \subset Q$ and $T$ be as above. We let $G^{\chi}(P,T)$ denote the
spectrum $K(\sC^{\chi}_T)$ and let $G^{\chi}_*(P,T)$ denote the sum of its
homotopy groups. We define the twisting action of
$p \in P$ on $G^{\chi}_i(P,T)$ by 
\begin{equation}\label{eqn:Twist-0}
t_{p}(\alpha) = \chi(p^{-1}) \alpha \text{ for $\alpha \in G^{\chi}_i(P,T)$}.
\end{equation}

Note that this makes sense since we consider $K$-groups with complex 
coefficients. We extend this action to all of $G_i(Q,T)$ using
\lemref{lem:definitionoftwist}. Since $P$ is finite, every $\chi$ is a 
multiplicative homomorphism $\chi: P \to \C^{\times}$ and this implies that
~\eqref{eqn:Twist-0} defines a $\C$-linear action of $P$ on $G_*(Q,T)$
which keeps each $G^{\chi}_*(P,T)$ invariant.
When we restrict to the case $i = 0$, this action is given by
\begin{equation}\label{eqn:Twist-1}
t_p([\sF]) = t_p(\Sigma_{\chi \in \hat{P}} [\sF_{\chi}]) = \Sigma_{\chi \in \hat{P}} 
\chi(p^{-1}) [\sF_{\chi}].
\end{equation}

This is same as the one considered in \cite[\S~6.2]{EG2}.
Note that if $\sF$ is a $Q$-equivariant vector bundle, then
each $\sF_{\chi}$ is also a $Q$-equivariant vector bundle by
\lemref{lem:definitionoftwist}. In particular, the twisting map
$t_p$ in ~\eqref{eqn:Twist-1} is defined for vector bundles as well and there 
is an automorphism $t_p: K_0(Q,T) \to K_0(Q,T)$.  

Now, for any $\sF \in \sC^{\chi}_T$ and $\sG \in \sC^{\chi'}_T$, 
we have $\sF \otimes_{\sO_T} \sG \in \sC^{\chi \cdot \chi'}_T$. Since
$G_i(Q,T)$ has a structure of a $K_0(Q,T)$-module via the tensor product of 
$\sO_T$-modules, it follows that 
\begin{equation}\label{eqn:Twist-2}
t_p(\alpha \cdot \beta) = t_p(\alpha) \cdot t_p(\beta)
\end{equation}
for $\alpha  \in K_0(Q,T)$ and $\beta \in G_i(Q,T)$.
In particular,  $t_p: K_0(Q,T) \to K_0(Q,T)$ is a $\C$-algebra automorphism.
The following are some more properties of the twisting maps.

\begin{prop}\label{prop:invariantsfunct}
Let $Q$ be a linear algebraic group acting properly on algebraic spaces $T$ 
and $T'$.  Let  $P \subset Q$ be a finite central subgroup of $Q$ which acts 
trivially on $T$ and $T'$. Let $f : T' \rightarrow T$ be a $Q$-equivariant 
map. Then the following hold.
\begin{enumerate}
\item The map $f^*: K_0(Q,T) \rightarrow K_0(Q,T')$ commutes with the 
twisting action. 
\item If $f$ is flat, the map $f^*:G_*(Q,T) \to G_*(Q,T')$ commutes with the 
twisting action.
\item For $p \in P$, we have $t_{p^{-1}}(G_*(Q,T)_{\fm_p}) = 
G_*(Q,X)_{\fm_1}$ under the decomposition of
$G_*(Q,T)$ given in \propref{prop:direcsumdecomp}. 
\end{enumerate}
\end{prop}
\begin{proof}
The first and the second properties are immediate from ~\eqref{eqn:Twist-0}
since $t_p$ is just a scalar multiplication
on $\C$-vector spaces $G_*(Q,-)$. 
On $R(Q)$, it is an easy verification from ~\eqref{eqn:Twist-1} that 
$t^{-1}_p(\fm_p) = \fm_1$ (for example, see \cite[\S~6.2]{EG2}). Using this,
the last property follows directly
from ~\eqref{eqn:Twist-2} and (1), which together say that
$t_p(\alpha \cdot \beta) = t_p(\alpha) \cdot t_p(\beta)$ for $\alpha \in R(Q)$
and $\beta \in G_*(Q,T)$.
\end{proof}

\begin{prop}\label{prop:proper-T}
Let $f: T' \to T $ be a $Q$-equivariant proper morphism of algebraic spaces over
$\C$ with proper $Q$-actions. 
Let $P \subset Q$ be a finite central subgroup of $Q$ which acts trivially on 
$T$ and $T'$. Then $f$ induces a 
push-forward map $f^{\chi}_*: G^{\chi}_*(P,T') \to G^{\chi}_*(P,T)$
and further the following diagram commutes 
\begin{equation}\label{eqn:propertwist}
\xymatrix@C1pc{ 
\oplus_{\chi \in \wh{P}} G^{\chi}_*(P,T') 
\ar[rr]^-{\oplus_{\chi \in \wh{P}} i'_{\chi}} 
\ar[d]_{\oplus_{\chi \in \wh{P}} f_*^{\chi}}  &&
G_*(Q,T') \ar[d]^{f_*}      \\
\oplus_{\chi \in \wh{P}} G^{\chi}_*(P,T) \ar[rr]_-{\oplus_{\chi \in \wh{P}} i_{\chi}}  
&& G_*(Q,T).}
\end{equation}  
 
\end{prop} 
\begin{proof}
For the benefit of the reader we provide a detailed proof when $T$ and $T'$ are 
quasi-projective schemes with linear $Q$-actions. The general case follows using similar
arguments and the construction of proper pushforward as in \cite[\S~1.11, 1.12]{TO2}.	
We shall mimic Quillen's construction \cite{Quillen} 
of the push-forward map in $K$-theory, which was adapted to the
equivariant set up in \cite[\S~1.11, 1.12]{TO2}.
We first observe as an immediate consequence of the definition of
$\sC^{\chi}_{(-)}$ and $f_*$ that for any coherent 
sheaf $\sF \in \sC^{\chi}_{T'}$, one has 
$f_*(\sF) \in \sC^{\chi}_T$. 

Let us first assume that $T'$ and $T$ are quasi-projective $\C$-schemes.
In this case, $f$ must be projective. Since $G$ acts linearly on $T'$
(and also on $T$), we can find a $Q$-equivariant line bundle $\sO(1)$ on $T'$
which is very ample. In particular, this is very ample relative to the
$Q$-equivariant map $f: T' \to T$. In this case,
there is a $Q$-equivariant factorization $T' \inj \P_T(\sE) \xrightarrow{p} 
T$ of $f$, where the first map is a closed embedding and $\sE$ is 
a $Q$-equivariant vector bundle on $T$.  
This gives a $Q$-equivariant surjection $f^*(\sE) \surj \sO(1)$,
and hence for every integer $n$, a surjection $f^*(\sE)^{\otimes n} \surj \sO(n)$.

Dualizing this surjection and twisting by $\sO(n)$, we get a $Q$-equivariant
short exact sequence of vector bundles
\begin{equation}\label{eqn:proper-T-0}
0 \to \sO_{T'} \to \sE'(n) \to \sE'' \to 0,
\end{equation}
where we take $\sE' = (f^*(\sE)^{\otimes n})^{\vee}$. Tensoring this with any
given $\sF \in {\sC oh}^Q_{T'}$, we get an inclusion $\sF \inj 
(\sF \otimes\sE')(n)$ for every integer $n$. Now, we know that 
there exists $n \gg 0$ such that ${\bf R}^if_*((\sF \otimes\sE')(n)) = 0$ 
for all $i \ge 1$. We let $\sG = (\sF \otimes\sE')(n)$.
Since the direct summand $\sF_{\chi} \inj \sF$ is functorial,
we also get an inclusion $\sF_{\chi} \inj \sG_{\chi}$ and
${\bf R}^if_*(\sG_{\chi}) = 0$ for $i \ge 1$.
In particular, if $\sF \in \sC^{\chi}_{T'}$ so that $\sF = \sF_{\chi}$,
we get an inclusion $\sF \inj \sG$ with $\sG \in \sC^{\chi}_{T'}$
such that ${\bf R}^if_*(\sG) = 0$ for $i \ge 1$.

If we let ${\sC oh}^Q_{T'}(f)$ (resp. $\sC^{\chi}_{T'}(f)$)
denote the subcategory of ${\sC oh}^Q_{T'}$ (resp. $\sC^{\chi}_{T'}$) which are
$f$-acyclic, then we have shown that every $\sF \in {\sC oh}^Q_{T'}$ 
(resp. $\sC^{\chi}_{T'}$) admits a $Q$-equivariant injection into an object of 
${\sC oh}^Q_{T'}(f)$ (resp. $\sC^{\chi}_{T'}(f)$).
We conclude from \cite[\S~3, Corollary~3]{Quillen} that 

\begin{equation}\label{eqn:proper-T-1}
K({\sC oh}^Q_{T'}) \simeq K(({\sC oh}^Q_{T'})^{\rm op}) \simeq
K(({\sC oh}^Q_{T'}(f))^{\rm op}) \simeq K({\sC oh}^Q_{T'}(f)).
\end{equation}
\[
K({\sC}^{\chi}_{T'}) \simeq K(({\sC}^{\chi}_{T'})^{\rm op}) \simeq
K(({\sC}^{\chi}_{T'}(f))^{\rm op}) \simeq K({\sC}^{\chi}_{T'}(f)).
\]

In particular, the inclusions ${\sC oh}^Q_{T'}(f) \subset {\sC oh}^Q_{T'}$
and ${\sC}^{\chi}_{T'}(f) \subset {\sC}^{\chi}_{T'}$ induce a 
commutative diagram of $K$-theory spectra 

\begin{equation}\label{eqn:proper-T-2}
\xymatrix@C1pc{
{\underset{\chi \in \wh{P}}\amalg} K({\sC}^{\chi}_{T'}(f))
\ar[rr]^-{\simeq} \ar[d]_{\amalg_{\chi \in \wh{P}} i'_{\chi}}  & &
{\underset{\chi \in \wh{P}}\amalg} K(\sC^{\chi}_{T'}) 
\ar[d]^{\amalg_{\chi \in \wh{P}} i'_{\chi}} \\
K({\sC oh}^Q_{T'}(f)) \ar[rr]^-{\simeq} & & K({\sC oh}^Q_{T'})}
\end{equation}
where the horizontal arrows are homotopy equivalences.

On the other hand, as $f_*$ is an exact functor on ${\sC oh}^Q_{T'}(f)$ and 
$\sC^{\chi}_{T'}(f)$,
there is a commutative diagram of push-forward maps
\begin{equation}\label{eqn:proper-T-3}
\xymatrix@C1pc{
{\underset{\chi \in \wh{P}}\amalg} K({\sC}^{\chi}_{T'}(f)) 
\ar[rr]^-{\amalg_{\chi \in \wh{P}} i'_{\chi}} 
\ar[d]_{\oplus_{\chi \in \wh{P}} f_*^{\chi}} 
& & K({\sC oh}^Q_{T'}(f)) \ar[d]^{f_*} \\ 
{\underset{\chi \in \wh{P}}\amalg} K(\sC^{\chi}_T) 
\ar[rr]^-{\amalg_{\chi \in \wh{P}}i_{\chi}}  & & K({\sC oh}^Q_T).}
\end{equation} 

The proposition follows in the quasi-projective case by combining
~\eqref{eqn:proper-T-2} and ~\eqref{eqn:proper-T-3}.
%As shown in \cite[\S~1.0. 1.11]{TO2}, the restriction of the
%functor $f_*$ on $G_0(Q,T')$ is given by $f_*([\sF]) = 
%\Sigma_i (-1)^i {\bf R}^if_*(\sF)$ (see ~\eqref{eqn:Twist-1}). 
%Since each $f^{\chi}_*$ is simply the
%restriction of $f_*$ on $G^{\chi}_0(P, T')$, the same expression holds
%for $f^{\chi}_*$ too. 

Let us now assume that $T$ and $T'$ are algebraic spaces which are not 
necessarily quasi-projective. In this case, we first observe that
the definition ~\eqref{eqn:Decom-1} makes sense for quasi-coherent
sheaves as well. If we denote this category by ${\sQ \sC}^{\chi}_{T'}$,
then the proof of \lemref{lem:definitionoftwist} (3) also shows 
that ${\sQ \sC}^{\chi}_{T'}$ is abelian. In particular, it has enough
injectives. It follows that the functors ${\bf R}^if_*$ from 
the category of $Q$-equivariant quasi-coherent sheaves on $T'$ to that on $T$
restrict to ${\sQ \sC}^{\chi}_{T'}$. Moreover, ${\bf R}f_*$ has
finite cohomological dimension by \lemref{lem:Fin-coh}.
We conclude that the map $f_*: G_*(Q,T') \to G_*(Q,T)$, is well defined and it
restricts to a similar push-forward map $f^{\chi}_*:G^{\chi}_*(P,T') \to
G^{\chi}_*(P,T)$. This completes the proof of the proposition. 
\end{proof}

As an immediate consequence of \propref{prop:proper-T} and ~\eqref{eqn:Twist-0},
we get

\begin{cor}\label{cor:proper-T-commute} 
Let $T$ and $T'$ be algebraic spaces of finite type over $\C$ with proper $Q$-actions. Let $P$ be a finite central subgroup of $Q$ such that $P$ acts 
trivially on $T$ and $T'$. Let $f: T' \to T$ be a $Q$-equivariant proper map. 
Then $f_*: G_*(Q,T') \to G_*(Q,T)$ is equivariant for 
the twisting by $P$-action. That is, for each $p \in P$,
the diagram
\begin{equation}\label{eqn:proper-T-commute-0}
\xymatrix@C1pc{
G_i(Q,T') \ar[d]_{f_*} \ar[r]^-{t_p} &  G_i(Q,T') \ar[d]^{f_*}  \\
G_i(Q,T)  \ar[r]^-{t_p}  &   G_i(Q,T)}
\end{equation}
commutes. 
\end{cor}

\begin{prop}\label{prop:invtwisst}
Let $Q$ be a linear algebraic group acting properly on an algebraic space 
$T$. Let $P$ be a central subgroup of $Q$ of finite order which acts trivially 
on $T$. Then for any $p \in P$ and $\alpha \in G_*(Q,T)$, we have 
$\inv^Q_T \circ t_p(\alpha) = \inv^Q_T(\alpha)$.
\end{prop}
\begin{proof}
For $G_0(Q,T)$, this is shown in \cite[Lemma~6.6]{EG2}, and we follow a similar
argument.
By \lemref{lem:definitionoftwist}, we can replace $G_*(Q,T)$ by
$G^{\chi}(P,T)$. Now, it follows from ~\eqref{eqn:Decom-0} that
no sheaf $\sF \in \sC^{\chi}_T$ can be $P$-invariant unless $\chi = 1$.
Hence, it can not be $Q$-invariant. If $p: T \to T/Q$ is the quotient map,
it follows that $(p_*(\sF))^G = 0$ unless $\chi = 1$.
In particular, $p_*: G^{\chi}_*(P,T) \to G_*(T/Q)$ is the zero map
unless $\chi = 1$.
When $\chi = 1$, the twisting map $t_p: \sC^{1}_T \to \sC^{1}_T$ is
identity by ~\eqref{eqn:Twist-0} so the assertion is obvious.
\end{proof}

\section{Morita isomorphisms}\label{Sec:Morita Isomorphism}
In this section, we define some Morita isomorphisms for $K$-theory
and prove their functorial properties that will be needed for the 
Atiyah-Segal map in \S~\ref{sec:ASmap}.

\subsection{The map $\mu_g$}\label{sec:Mu-map}
Let $G$ be a linear algebraic group acting properly on an algebraic space 
$X$. Let $\psi$ be a semi-simple conjugacy class in $G$. If $g,h \in \psi$, 
then 
there exists $k \in G$ such that $h = kgk^{-1}$. There exists an isomorphism 
between the centralizers $Z_g$ and $Z_h$ given by
\begin{equation}\label{eqn:cent-iso}
\phi: Z_g \xrightarrow{\simeq} Z_h ; ~~ \phi (z) = kzk^{-1}.
\end{equation}

There is an isomorphism $u: X^g \to X^h$ between the fixed point loci, 
given by $u(x) = kx$. 
If we let $Z_g$ act on $X^h$ via the canonical action of $Z_h$ 
composed with $\phi$, then $u$ is $Z_g$-equivariant.
We call this the $\star$-action of $Z_g$ on a $Z_h$-space.
This allows us to define the isomorphisms at the level of $K$-theory
\begin{equation}\label{diag:defnofphi_*}
u_{*}: G_*(Z_g,X^g) \xrightarrow{\simeq} G_*(Z_g,X^h); \ \ \
\phi_*: G_*(Z_g,X^h) \xrightarrow{\simeq} G_*(Z_h,X^h).
\end{equation} 
We let
\begin{equation}\label{diag:defnofphi_*0} 
\theta_X = \phi_* \circ u_*: G_*(Z_g, X^g) \xrightarrow{\simeq} G_*(Z_h,X^h).
\end{equation}

We let $G \times Z_g$ act on $G \times X$ by 
$(g,z)\cdot (h,x) = (ghz^{-1}, zx)$. Then $1 \times Z_g$ acts on $G \times X$
freely with quotient $G \times^{Z_g} X$. Note that this action of $1 \times Z_g$
coincides with the $Z_g$-action on $G \times X$ given in 
\lemref{lem:Morita-stack*}. Let $p_{X^g}: G \times X^g \to X^g$ and
$p_{X^h}: G \times X^h \to X^h$ be the projections.

\begin{defn}\label{defn:Mor-fun-0}
The Morita equivalence $\mu_g$ is the composition of the functors:
\begin{equation}\label{moritamap}
\mu_g: {\sC oh}^{Z_g}_{X^g}  \xrightarrow{p_{X^g}^*}  
{\sC oh}^{G \times Z_g}_{G \times X^g} 
\xrightarrow{{\rm Inv}^{1 \times Z_g}_{G \times X^g}}
{\sC oh}^G_{G \times^{Z_g} X^g} \xrightarrow{p_g} {\sC oh}^G_{I^{\psi}_X}.
\end{equation}
\end{defn}

In the notations of ~\eqref{eqn:Morita-crucial}, 
$p_{X^g}^*$ is same as $\bar{p}^*$. But we have seen in the proof of
\lemref{lem:Morita-stack*} that $\iota'$ and $\bar{p}$ are inverses to each
other as maps of stacks. It follows that $p^*_{X^g} = \bar{p}^* = \iota'_*$.
Since ${\rm Inv}^{1 \times Z_g}_{G \times X^g}$ is same as $\bar{q}_*$
(see \thmref{thm:invariants}), we see that $\mu_g$ is the push-forward map
\begin{equation}\label{eqn:mu-g-*}
\mu_g = (\iota^{GZ_g}_{X^g})_* = \bar{q}_* \circ \iota'_*: 
{\sC oh}_{[{X^g}/{Z_g}]} \xrightarrow{\simeq} {\sC oh}_{[{I^{\psi}_X}/G]}.
\end{equation}

The functor $\mu_g$ induces weak-equivalence between the $K$-theory spectra
and an hence an isomorphism 
$G_*(Z_g, X^g) \xrightarrow{\simeq} G_*(G, I^{\psi}_X)$.
This is an $R(G)$-linear isomorphism, where $R(G)$ acts on
$G_*(Z_g, X^g)$ via the restriction map $R(G) \to R(Z_g)$
(see \cite[Remark~3.2]{EG2}). 
We shall denote the induced composite map on the $K$-theory also by $\mu_g$. 
Our goal is to show that this map is
compatible with respect to the choice of the representatives of $\psi$.
So we fix $h = kgk^{-1} \in \psi$.

\begin{comment}
Note that $\bar{p}_* = \bar{p}_* \cdot id = \bar{p}_* \circ 
(\iota'_* \circ \iota'^*) = (\bar{p}_* \circ \iota'_*) \circ \iota'^*=
(\bar{p} \circ \iota')_* \circ \iota'^* = id_* \circ \iota'_* = \iota'_*$

Here, the equivalence $p_g$ is given in \lemref{lem:torsor}. The functor
$p_{X^g}^*$ takes any $Z_g$-equivariant coherent sheaf on $X^g$ to 
a $G \times Z_g$-equivariant coherent sheaf on $G \times X^g$ and is an
equivalence as $p_{X^g}$ is a $G$-torsor (see \cite[\S~2.2.2]{MK}).
The functor ${\rm Inv}^{1 \times Z_g}_{G \times X^g}$ is an equivalence 
because $G \times X^g \to G \times^{Z_g} X^g$ is $(G \times Z_g)$-equivariant
and $(1 \times Z_g)$-torsor.
Hence, the induced maps on the $K$-theory spectra are all equivalences.
Moreover, the induced composite map $G_*(Z_g, X^g)  \xrightarrow{\simeq}
G_*(G, G \times^{Z_g} X^g)$ on $K$-theory is $R(G)$-linear, where $R(G)$ acts on
$G_*(Z_g, X^g)$ via the restriction map $R(G) \to R(Z_g)$
(see \cite[Remark~3.2]{EG2}). We shall denote
the induced composite map on the $K$-theory also by $\mu_g$.  
Our goal is to show that this map is
compatible with respect to the choice of representatives of $\psi$.
So we fix $h = kgk^{-1} \in \psi$. 
\end{comment}

\begin{lem}\label{lem:thetahandgindependence}
For an algebraic space $X$ with a proper $G$-action, we have 
$\mu_h \circ \theta_X = \mu_g $.
\end{lem}
\begin{proof}
We let $Z_g$ act on $X^h$ and $G \times Z_g$ act on $G \times X^h$ via 
$\phi$ as 
%\begin{equation}\label{eqn:ind-0}
\[
g_2 \star z = \phi(g_2) \cdot z \ \mbox{and} \
(g_2,z) \star (g' ,x) : =(g_2 g' \phi(z)^{-1},\phi(z)x) =  
(g_2 g' (kzk^{-1})^{-1}, kzk^{-1} x).
\]
%\end{equation}  

Since $Z_h$ acts on $X^h$ and
$G \times Z_h$ acts on $G \times X^h$ by $z \cdot x = zx$ and
$(g_2, z)\cdot (g' ,x) = (g_2g'z^{-1}, zx)$, respectively, it follows that these
two actions correspond to the above $\star$-actions of $Z_g$ and $G \times Z_g$
on $X^h$ and $G \times X^h$ via the isomorphism $\phi$, respectively. 
We thus have a commutative diagram of equivalences
\begin{equation}\label{eqn:ind-1}
\xymatrix{
{\sC oh}^{Z_g}_{X^h}  \ar[r]^-{p_{X^h}^*} \ar[d]_{\phi_*} &  
{\sC oh}^{G \times Z_g}_{G \times X^h} 
\ar[r]^-{{\rm Inv}^{1 \times Z_g}_{G \times X^h}} \ar[d]^{(Id \times \phi)_*} &
{\sC oh}^G_{G \times^{Z_g} X^h} \ar[r]^-{p_h} \ar[d]^{(\ov{Id \times \phi})_*} 
& {\sC oh}^G_{I^{\psi}_X} \ar@{=}[d] \\
{\sC oh}^{Z_h}_{X^h}  \ar[r]^-{p_{X^h}^*} &  
{\sC oh}^{G \times Z_h}_{G \times X^h} 
\ar[r]^-{{\rm Inv}^{1 \times Z_h}_{G \times X^h}} &
{\sC oh}^G_{G \times^{Z_h} X^h} \ar[r]^-{p_h} & {\sC oh}^G_{I^{\psi}_X},}
\end{equation}
where the top horizontal arrow is same as the bottom one and is 
induced by the map
$p_h: G \times X^h \to I^{\psi}_X$, given by $p_h(g', x) = (g'h g'^{-1}, g'x)$. 
This is a $Z_g$-torsor with respect to the $\star$-action.

\vskip .3cm

We now let $\delta: G \times X^g \to G \times X^h$ be the map
$\delta(g',x) = (g'k^{-1}, kx)$. With respect to
the canonical action of $G \times Z_g$ on $G\times X^g$ and $\star$-action on 
$G \times X^h$, we have
\begin{equation}\label{eqn:ind-2}
\begin{array}{lllll}
\delta((g_1,z_1) (g',x)) & = & \delta ((g_1g'z_1^{-1}, z_1 x)) & = &
(g_1g'z_1^{-1} k^{-1} ,k z_1 x ) \\
& = & (g_1g'k^{-1}kz_1^{-1}k^{-1} , k z_1 k^{-1} k  x) & = & 
(g_1g'k^{-1} (kz_1k^{-1})^{-1} , (kz_1k^{-1})kx) \\
& = & (g_1,z_1) \star (g'k^{-1}, kx) & = & (g_1,z_1) \star \delta(g',x)
\end{array} 
\end{equation}
and this shows that $\delta$ is $(G\times Z_g)$-equivariant.
Furthermore, one verifies immediately that the diagram
\begin{equation}\label{eqn:ind-3}
\xymatrix@C1pc{
G \times X^g \ar[r]^-{p_{X^g}} \ar[d]_{\delta}   &    X^g \ar[d]^{u} \\
G \times X^h  \ar[r]^-{p_{X^h}} &    X^h}
\end{equation}
commutes in which all maps are $(G \times Z_g)$-equivariant with respect to
the $\star$-action on the bottom row (where we 
give $X^g$ and $X^h$ the trivial $G$-action).

We next claim that the map $p_g: G \times X^g \to I^{\psi}_X$
is same as the map $p_h \circ \delta$.
But this can be directly checked as follows.
\begin{equation}\label{eqn:ind-4}
\begin{array}{lllll}
 p_h \circ \delta (g_1,x) & = & p_h(g_1 k^{-1} ,kx) & = &
((g_1k^{-1}) h (g_1 k^{-1})^{-1} , g_1 x ) \\
& = & (g_1k^{-1}kgk^{-1}k g_1^{-1} , g_1 x )  & = &
(g_1 g g_1^{-1} , g_1 x) \\
& = & p_g(g_1,x). & &                              
\end{array}
\end{equation}

Combining what we have shown above, we get a  commutative diagram
\begin{equation}\label{diag:thetahandgindependence}
\xymatrix@C1pc{ 
G_*(Z_g,X^g) \ar[r]^-{p^*_{X^g}} \ar[d]_{u*} & 
G_*(G \times Z_g, G \times X^g) \ar[d]^{\delta_*} 
\ar[r]^-{{\rm Inv}^{1 \times Z_g}_{G \times X^g}} & 
G_*(G,G\times^{Z_g} X^g) \ar[d]^{{\bar{\delta}_*}} \ar[r]^-{p_g} & 
G_*(G,I^{\psi}_X) \ar[d]^{Id} \\
G_*(Z_g,X^h) \ar[r]^-{p^*_{X^h}}  \ar[d]_{\phi_*}   & 
G_*(G \times Z_g, G \times X^h)  
\ar[r]^-{{\rm Inv}^{1 \times Z_g}_{G \times X^h}}   
\ar[d]^{(Id \times \phi)_*}     & 
G_*(G,G\times^{Z_g} X^h) \ar[r]^-{p_h}        
\ar[d]^{{\ov{(Id \times \phi)}_*}}  & 
G_*(G,I^{\psi}_X) \ar[d]^{Id} \\ 
G_*(Z_h,X^h) \ar[r]^-{p^*_{X^h}}  & 
G_*(G \times Z_h, G \times X^h)  
\ar[r]^-{{\rm Inv}^{1 \times Z_h}_{G \times X^h}}     & 
G_*(G,G \times^{Z_h} X^h) \ar[r]^-{p_h}  & G_*(G,I^{\psi}_X),}
\end{equation}
where the top squares commute by ~\eqref{eqn:ind-2}, ~\eqref{eqn:ind-3} and
~\eqref{eqn:ind-4}, and the bottom squares commute by ~\eqref{eqn:ind-1}.
Since the composite arrows on the top and the bottom are $\mu_g$ and
$\mu_h$, respectively, and the composite vertical arrow on the left is 
$\theta_X$, the lemma follows.
\end{proof}

\begin{remk}\label{remk:Nnn-proper}
As ${1 \times Z_g}$ acts freely on $G \times X^g$, 
the map ${\rm Inv}^{1 \times Z_g}_{G \times X^g}$ exists and is 
an isomorphism even if $G$-action on $X$ is not proper (see
Remark~\ref{remk:Inv-iso}). In particular, the reader can check that
\lemref{lem:thetahandgindependence} holds without the properness
of the $G$-action on $X$. But we shall not need this general version.
\end{remk}

\begin{lem}\label{lem:moritaandrrproper}
Let $f: X \to Y$ be a $G$-equivariant proper morphism of algebraic spaces
with proper $G$-actions and let $f^I: I_X \to I_Y$ be the induced map
on the inertia spaces.
Let $g \in G$ be a semi-simple element of $G$ with $\psi$ its conjugacy 
class. Let $\mu^X_g$ and $\mu^Y_g$ be the Morita isomorphisms for $X$ and $Y$.
Then there is a commutative diagram
\begin{equation}\label{eqn:Mor-commute}
\xymatrix@C1pc{
G_*(Z_g,X^g) \ar[r]^-{\mu^X_g} \ar[d]_{f_*} & G_*(G,I^{\psi}_X) \ar[d]^{f^I_*} \\
G_*(Z_g,Y^g) \ar[r]^-{\mu^Y_g} & G_*(G, I^{\psi}_Y).}
\end{equation}
\end{lem}
\begin{proof}
It is well known and an elementary exercise that $f^I$ is proper 
if and only if $f$ is proper. 
We have the commutative diagram
\begin{equation}\label{eqn:Mor-commute-0}
\xymatrix@C1pc{
[{X^g}/{Z_g}] \ar[r]^{\iota^{GZ_g}_{X^g}} \ar[d] & [{I^{\psi}_X}/G] \ar[d] \\
[{Y^g}/{Z_g}] \ar[r]^-{\iota^{GZ_g}_{Y^g}} & [{I^{\psi}_Y}/G],}
\end{equation}
where the all maps are proper. The lemma now follows from the functoriality of
proper push-forward on $K$-theory of coherent sheaves because
$\mu^X_g$ and $\mu^Y_g$ are just the push-forward on $K$-theory induced by
$\iota^{GZ_g}_{X^g}$ and $\iota^{GZ_g}_{Y^g}$, 
respectively (see ~\eqref{eqn:mu-g-*}).
\end{proof}

\begin{lem}\label{lem:moritaandrrproper-loc}
Let $\iota: X \inj Y$ be a $G$-equivariant closed immersion of algebraic spaces
with proper $G$-actions with open complement $U$.
Let $g \in G$ be a semi-simple element of $G$ with $\psi$ its conjugacy 
class. Then there is a commutative diagram of 
$R(Z_g)$-modules
\begin{equation}\label{eqn:Mor-commute-1}
\xymatrix@C1pc{
G_{i+1}(Z_g,U^g) \ar[r] \ar[d]_{\mu_g} & G_i(Z_g, X^g) \ar[r]^-{\iota_*} 
\ar[d]^{\mu_g} & 
G_i(Z_g, Y^g) \ar[r] \ar[d]^{\mu_g} & 
G_i(Z_g,U^g) \ar[r] \ar[d]^{\mu_g} & G_{i-1}(Z_g,X^g) \ar[d]^{\mu_g} \\
G_{i+1}(G,I^{\psi}_U) \ar[r]  & G_i(G,I^{\psi}_X) \ar[r]^-{\iota_*}  & 
G_i(G,I^{\psi}_Y) \ar[r] & 
G_i(G,I^{\psi}_U) \ar[r]  & G_{i-1}(G,I^{\psi}_X)}
\end{equation}
commutes with rows being exact.
\end{lem}
\begin{proof}
It is clear that $I_X \inj I_Y$ is a closed immersion with complement $I_U$.
We have the following diagram of abelian categories
\begin{equation}\label{eqn:Mor-commute-2}
\xymatrix@C1pc{
{\sC oh}^{Z_g}_{X^g} \ar[r]^-{\iota_*} \ar[d]_{p^*_{X^g}} & 
{\sC oh}^{Z_g}_{Y^g} \ar[d]^{p^*_{Y^g}} &
{\sC oh}^{G \times Z_g}_{G \times X^g} \ar[r]^-{\iota_*}  
\ar[d]_{{\rm Inv}^{1 \times Z_g}_{G \times X}} &
{\sC oh}^{G \times Z_g}_{G \times Y^g} \ar[d]^{{\rm Inv}^{1 \times Z_g}_{G \times X}} &
{\sC oh}^G_{G \times^{Z_g} X^g}  \ar[r]^-{\iota_*} 
\ar[d]_{p^X_g} & {\sC oh}^G_{G \times^{Z_g} Y^g} \ar[d]^{p^Y_g} \\
{\sC oh}^{G \times Z_g}_{G \times X^g} \ar[r]^-{\iota_*} &
{\sC oh}^{G \times Z_g}_{G \times Y^g} &   
{\sC oh}^G_{G \times^{Z_g} X^g}  \ar[r]^-{\iota_*} & 
{\sC oh}^G_{G \times^{Z_g} Y^g} &
{\sC oh}^G_{I^{\psi}_{X}}  \ar[r]^-{\iota_*} & 
{\sC oh}^G_{I^{\psi}_{Y}}.}
\end{equation}

It is shown in \lemref{lem:moritaandrrproper} that all the squares in
~\eqref{eqn:Mor-commute-2} are commutative. 
Moreover, in each square, every horizontal functor
is an inclusion of a localizing Serre subcategory whose localization is
canonically equivalent to the analogous abelian category for 
$U$ via restriction.
In particular, every square gives rise to a commutative diagram
of the localization sequences of $K$-theory spectra.
This implies that each of the maps $p^*_{(-)}, 
{\rm Inv}^{1 \times Z_g}_{G \times (-)}$ and $p^{(-)}_g$ commutes with 
the localization sequences. It follows that the composite
$\mu_g = p^{(-)}_g \circ {\rm Inv}^{1 \times Z_g}_{G \times (-)} \circ  p^*_{(-)}$
commutes with the localization sequence. Equivalently, 
~\eqref{eqn:Mor-commute-1} commutes.
\end{proof}

\subsection{The map $\omega_g$}\label{sec:Lie-alg}
%\subsection{The action of $\lambda_{-1}((\fg/\fz_g)^*)$}\label{sec:Lie-alg}
We continue with the above notations. 
Let $\psi \subset G$ be a semi-simple conjugacy class and let $g \in \psi$
with centralizer $Z_g$.
We have seen in ~\eqref{moritamap} that the Morita map
$\mu_g: G_*(Z_g, X^g) \to G_*(G, I^{\psi}_X)$ is an $R(G)$-linear
isomorphism, where $R(G)$ acts on $G_*(Z_g, X^g)$ 
via the restriction $R(G) \to R(Z_g)$.
This map induces a similar isomorphism of the localizations
$\mu_g: G_*(Z_g, X^g)_{\fm_{\Phi}} \xrightarrow{\simeq} 
G_*(G, I^{\psi}_X)_{\fm_{\Phi}}$ at the maximal ideals of 
$R(G)$ corresponding the semi-simple conjugacy classes $\Phi \in S_G$.

\begin{enumerate}
\item
If we let $\Phi = \psi$, we have 
$\psi \cap Z_g = \{\{g\}, \psi_1, \cdots , \psi_r\}$.
Letting $G_*(Z_g, X^g)_{nc_{\psi}} = \stackrel{r}{\underset{i = 1}\oplus}
G_*(Z_g, X^g)_{\fm_{\psi_i}}$, we get
\begin{equation}\label{eqn:Omega**-0}
G_*(Z_g, X^g)_{\fm_g} \oplus G_*(Z_g, X^g)_{nc_{\psi}}
\xrightarrow{\simeq} G_*(G, I^{\psi}_X)_{\fm_{\psi}}.
\end{equation}

\item
If we let $\Phi = \{e\}$, we get $\Phi \cap Z_g = \{e\}$
and hence an isomorphism
\begin{equation}\label{eqn:Omega**-0-e} 
G_*(Z_g, X^g)_{\fm_1} \xrightarrow{\simeq} G_*(G, I^{\psi}_X)_{\fm_1}.
\end{equation}
\end{enumerate}

We shall denote this isomorphism by $\mu^1_g$. 
One should keep in mind that the localization on the left side of $\mu^1_g$
is at the augmentation ideal of $R(Z_g)$ while on the right, it is
at the augmentation ideal of $R(G)$.

We shall let
$\mu^{\psi}_g$ denote the composite map
\begin{equation}\label{eqn:Omega**-1}
\mu^{\psi}_g: G_*(Z_g, X^g)_{\fm_g} \inj G_*(Z_g, X^g)_{\fm_{\psi}} 
\xrightarrow{\simeq} G_*(G, I^{\psi}_X)_{\fm_{\psi}}. 
\end{equation} 

Note that the inclusion map in ~\eqref{eqn:Omega**-1}
is $R(Z_g)_{\fm_{\psi}}$-linear while the isomorphism is $R(G)_{\fm_{\psi}}$-linear.
In particular, $\mu^{\psi}_g$ is  $R(G)_{\fm_{\psi}}$-linear.

\begin{defn}\label{defn:OMEGA}
We let $\omega_g$ be the composite map
\begin{equation}\label{eqn:w_g-0}
\omega_g = \mu^{\psi}_* \circ \mu^{\psi}_g: G_*(Z_g, X^g)_{\fm_g} 
\to G_*(G, I^{\psi}_X)_{\fm_{\psi}} \to G_*(G, X_{\psi})_{\fm_{\psi}}.
\end{equation}
\end{defn}

We now show that $\omega_g$ is an isomorphism of $R(G)_{\fm_{\psi}}$-modules
and is compatible with the choice of $g \in \psi$.
Recall that the conjugation action of
$G$ on itself fixes the identity element of $G$ and hence the 
tangent space of $G$ at its identity element (which is same as its 
Lie algebra $\fg$) is naturally a representation of $G$. 
This the called the adjoint representation of $G$.
If $\psi$ is a semi-simple conjugacy class in $G$ and $g \in \psi$,
we see that ${\fg}/{\fz_g}$ is representation of $Z_g$.
We let $\lambda_{-1}((\fg/\fz_g)^*) \in R(Z_g)$ be the top Chern  
class of the dual representation $(\fg/\fz_g)^*$, given by 
$\lambda_{-1}((\fg/\fz_g)^*) = \sum_i(-1)^i[\wedge^i((\fg/\fz_g)^*)]$.
We shall denote $\lambda_{-1}((\mathfrak{g}/\mathfrak{z}_g)^{*})$ by $l_g$
in the rest of this text.

If we let $h = kgk^{-1}$, the conjugation automorphism 
$\phi:G \to G$ (sending $g'$ to $kg'k^{-1}$) takes $Z_g$ onto $Z_h$.
In particular, $\phi_*: R(G) \xrightarrow{\simeq} R(G)$ 
restricts to an isomorphism $R(Z_g) \xrightarrow{\simeq} R(Z_h)$
such that 
\begin{equation}\label{eqn:phi-rep}
\phi_*(\fg) = \fg, \ \ \phi_*(\fz_g) = \fz_h \ \ \mbox{and} \ \
\phi_*(\fm_g) = \fm_h.
\end{equation}

To show that $\omega_g$ is an isomorphism when $X$ is smooth, we need
the following representation theoretic fact.

\begin{lem}\label{lem:Adj-rep}
The element $l_g$ is invertible in $R(Z_g)_{\fm_g}$.
\end{lem}
\begin{proof}
Using \propref{prop:conjugacyclassandmaximalideal}, it suffices to show that 
the character of $l_g$ does not vanish at $g$.
Since $g$ is semi-simple, its action on $V = {\fg}/{\fz_g}$ is diagonalizable.
In particular, we have $l_g = \stackrel{d}{\underset{i = 1}\prod} (1 - [V_i])$,
where each $V_i$ is an one-dimensional representation of $\<g\>$,
on which the action is by a character.
It is therefore enough to show that $V^*$ has no non-zero vector
which is left invariant by $g$. 
Since every representation of $\<g\>$ is completely reducible,
it is equivalent to show that $V$ has no non-zero $g$-invariant vector.

Now, it follows from \cite[\S~9.1 ($*$)]{Borel} that
$\fz_g = \{v \in \fg| {\rm Ad}_g(v) = v\}$, where ${\rm Ad}_g$ is the
action of $g$ on $\fg$ under the adjoint representation.
So we can write $\fg = V \oplus \fz_g$ as a representation of $\<g\>$
and it is clear that $V^g = \{0\}$. 
\end{proof}

\begin{lem}\label{lem: w_g}
The map 
\begin{equation}\label{eqn:w_g-0*}
\omega_g: G_*(Z_g,X^g)_{\fm_g} \to G_*(G,X_{\psi})_{\fm_{\psi}}
\end{equation}
is an isomorphism. Moreover, we have $\omega_g = \omega_h \circ \theta_X$
if we let $h = kgk^{-1} \in \psi$.
\end{lem}
\begin{proof}
We first show that $w_g$ is an isomorphism. Suppose that $X$ is smooth.
In this case, it follows from \cite[Theorems~5.3, 5.9]{EG4} that the map
$j^{\psi}_* \circ \mu^{\psi}_* \circ \mu^{\psi}_g \circ l_g: 
G_*(Z_g,X^g)_{\fm_g} \to G_*(G,X)_{\fm_{\psi}}$ is an isomorphism, 
where $j^{\psi}: X_{\psi} \inj X$ is the inclusion and 
$l_g: G_*(Z_g,X^g)_{\fm_g} \to G_*(Z_g,X^g)_{\fm_g}$
is multiplication by $l_g \in R(Z_g)_{\fm_g}$. We conclude from
\thmref{thm:closedimmiso} and \lemref{lem:Adj-rep} that
$\omega_g = \mu^{\psi}_* \circ \mu^{\psi}_g$ is an isomorphism.
We prove the general case by the Noetherian induction.

If $X$ is a reduced $G$-space, there exists a $G$-invariant dense open subspace 
$U \subset X$ which is a smooth scheme.
%[\texttt{Do we have to put the assumption connected or reduce to connected}].
Letting $\iota: Y = X \setminus U \inj X$, we get a diagram 
%\begin{equation}
\[
\xymatrix@C1pc{
G_{i+1}(Z_g,U^g)_{\fm_g} \ar[r] \ar[d]_{\omega_g} & G_i(Z_g,Y^g)_{\fm_g} \ar[r] 
\ar[d]^{\omega_g} & 
G_i(Z_g, X^g)_{\fm_g} \ar[r] \ar[d]^{\omega_g} & 
G_i(Z_g,U^g)_{\fm_g} \ar[r] \ar[d]^{\omega_g} & G_{i-1}(Z_g,Y^g)_{\fm_g} 
\ar[d]^{\omega_g} \\
G_{i+1}(G,U_{\psi})_{\fm_{\psi}} \ar[r] & G_i(G,Y_{\psi})_{\fm_{\psi}} \ar[r]  & 
G_i(G, X_{\psi})_{\fm_{\psi}} \ar[r]  & G_i(G,U_{\psi})_{\fm_{\psi}} 
\ar[r]& G_{i-1}(G,Y_{\psi})_{\fm_{\psi}},}
%\end{equation}  
\]
where the rows are the long exact localization sequences 
of $R(G)_{\fm_g}$-modules (see \cite[Theorem~2.7]{TO1}).
The first and the fourth vertical arrows (from left) are isomorphisms since $
U$ is smooth. The second and the fifth vertical arrows are isomorphisms by
the Noetherian induction. The lemma would now follow if we
know that all the squares in this diagram commute.
All we are therefore left with showing is the commutativity of the
diagram
\begin{equation}\label{diag:w_gfunctorialitylocalization}
\xymatrix@C1pc{
G_{i+1}(Z_g,U^g)_{\fm_g} \ar[r] \ar[d]_{\mu^{\psi}_g} & G_i(Z_g,Y^g)_{\fm_g} \ar[r] 
\ar[d]^{\mu^{\psi}_g} & G_i(Z_g, X^g)_{\fm_g} \ar[r] \ar[d]^{\mu^{\psi}_g} & 
G_i(Z_g,U^g)_{\fm_g} \ar[d]^{\mu^{\psi}_g} \\
G_{i+1}(G,I^{\psi}_U)_{\fm_{\psi}} \ar[r] \ar[d]_{\mu^{\psi}_*} & 
G_i(G,I^{\psi}_Y)_{\fm_{\psi}} \ar[r] \ar[d]^{\mu^{\psi}_*} & 
G_i(G,I^{\psi}_X)_{\fm_{\psi}} 
\ar[r] \ar[d]^{\mu^{\psi}_*} & G_i(G,I^{\psi}_U)_{\fm_{\psi}} \ar[d]^{\mu^{\psi}_*} \\
G_{i+1}(G,U_{\psi})_{\fm_{\psi}} \ar[r] & G_i(G,Y_{\psi})_{\fm_{\psi}} \ar[r]  & 
G_i(G, X_{\psi})_{\fm_{\psi}} \ar[r]  & G_i(G,U_{\psi})_{\fm_{\psi}}.}
\end{equation}  

Now, the bottom squares commute because all the vertical arrows are induced 
by the finite $G$-equivariant map $\mu^{\psi}: I^{\psi}_X  \to X_{\psi}$ 
(see \lemref{lem:torsor})
and its restrictions to $Y$ and $U$. In particular, the diagram of
abelian categories
\[
\xymatrix@C1pc{
{\sC oh}^G_{I^{\psi}_Y} \ar[r]^-{\iota_*} \ar[d]_{\mu^{\psi}_*} & 
{\sC oh}^G_{I^{\psi}_X} \ar[d]^{\mu^{\psi}_*} \\
{\sC oh}^G_{Y_{\psi}} \ar[r]^-{\iota_*} & 
{\sC oh}^G_{X_{\psi}}}
\]
commutes. The top horizontal arrow is the inclusion of a 
localizing Serre subcategory whose localization is
canonically equivalent to ${\sC oh}^G_{I^{\psi}_U}$ via restriction.
Similarly, ${\sC oh}^G_{U_{\psi}}$ is the localization of  
the bottom horizontal arrow.  The top squares commute by 
\lemref{lem:moritaandrrproper-loc}. This proves that ~\eqref{eqn:w_g-0*} is
an isomorphism.

The second part follows directly from Lemma~\ref{lem:thetahandgindependence}
and ~\eqref{eqn:phi-rep} 
which imply for any
$\alpha \in G_i(Z_g,X^g)_{\fm_g}$ that 
$\omega_h \circ \theta_X(\alpha) = \mu^{\psi}_* \circ 
\mu^{\psi}_h \circ \theta_X(\alpha) = \mu^{\psi}_* \circ \mu^{\psi}_g (\alpha)$.
\end{proof}

\section{The Atiyah-Segal map}\label{sec:ASmap}
Let a compact Lie group $G$ act on a compact Hausdorff 
space $X$. Let $I_G$ denote the augmentation ideal of $R(G)$. 
Recall from \cite{AS} that the Atiyah-Segal theorem in topology provides
an isomorphism between the $I_G$-adic completion of the $G$-equivariant 
topological $K$-theory of $X$ and the usual topological $K$-theory of
the Borel space $EG \times^G X$. The algebraic version of Atiyah-Segal theorem was
studied in \cite{AK2arxiv}. 
The $K$-theory of $EG \times^G X$ is often called 
the {\sl geometric part} of the equivariant $K$-theory $K(G,X)$.
As the goal of this text is to functorially describe the full equivariant $K$-theory
of a $G$-space in terms of the geometric part of the $K$-theory of
the inertia space, we call such a connection the {\sl Atiyah-Segal
correspondence}. In this section, we define the Atiyah-Segal map
which will describe this correspondence.  

\subsection{Two completions of equivariant $K$-theory}\label{sec:Comp-compl&}
Let $\sX$ be a separated quotient stack of finite type over
$\C$. Let $p:\sX \to M$ be the coarse moduli space map and let
$\{\sX_1, \cdots , \sX_r\}$ denote the set of inverse images of the connected
components of $M$ under this map. We shall say that $\sX$ is connected if
$r =1$.
Let $K_0(\sX)$ denote the Grothendieck group 
of vector bundles on $\sX$. For each $1 \le i \le r$, let $\fm_{\sX_i} \subset 
K_0(\sX)$ denote the set of virtual vector bundles on $\sX$ whose 
restrictions to $\sX_i$ have rank zero. It is easy to see that
$\fm_{\sX_i}$ is a maximal ideal of $K_0(\sX)$. We let $\fm_{\sX} 
= \stackrel{r}{\underset{i =1}\cap} \fm_{\sX_i}$ and call this the
{\sl augmentation ideal} of $K_0(\sX)$.

Let $X$ be an algebraic space with an action of a linear algebraic group $G$
such that $\sX = [X/G]$ and let $p_X: X \to \sX$ be the quotient map. 
%Since $\sX$ is Deligne-Mumford, it follows that the $G$-action on $X$ has finite stabilizers. 
The pull-back map $p^*_X: K_0(\sX) \to K_0(G,X)$ is 
an isomorphism of rings and $p^*_X(\fm_{\sX})$ is the set of 
$G$-equivariant virtual vector bundles on $X$ whose restrictions to each
$G$-invariant closed subspace $X_i:= (p_X)^{-1}(\sX_i)$ have rank zero.
We shall denote $p^*_X(\fm_{\sX})$ by $\fm_{\sX}$.

The $G$-equivariant maps $X_i \inj X \to \Spec(\C)$ induce the ring 
homomorphisms $R(G) \to K_0(G,X) \to K_0(G,X_i)$ and it follows from this that
$\pi^*(\fm_{1}) \subset p^*_X(\fm_{\sX})$, where we use $\fm_1$ as a shorthand for
the augmentation ideal $\fm_{1_G} \subset R(G)$.
In other words, we have the inclusion of ideals
$\fm_{1}K_0(\sX) \subset \fm_{\sX}$. In particular, given any $K_0(\sX)$-module
$N$, there is a natural map $\wh{N}_{\fm_{1}} \to 
\wh{N}_{\fm_{\sX}}$, where $\wh{N}_{\fm_{1}}$ and $\wh{N}_{\fm_{\sX}}$
are the $\fm_{1}$-adic and $\fm_{\sX}$-adic completions
of $N$, respectively.
Since each $G_i(\sX)$ is a $K_0(\sX)$-module,
we get a natural map $\wh{G_i(\sX)}_{\fm_{1}} \to \wh{G_i(\sX)}_{\fm_{\sX}}$.

Let $G_i(\sX)_{\fm_{1}}$ and $G_i(\sX)_{\fm_{\sX}}$ denote the localizations
of $G_i(\sX)$ at $\fm_{1}$ and $\fm_{\sX}$, respectively.
Since $\fm_{1}$ is a maximal ideal of $R(G)$ and $\fm_1K_0(\sX) \subset
\fm_{\sX}$, it follows that $\fm_1 = R(G) \cap \fm_{\sX}$.
In particular, the localization map $G_i(\sX) \to G_i(\sX)_{\fm_{\sX}}$ factors
through $G_i(\sX)_{\fm_1}$. To compare various localizations and completions of
the $K$-theory of stacks, we shall use the following elementary result from
commutative algebra.

\begin{lem}\label{lem:Elem-CAlg}
Let $A$ be a commutative ring (not necessarily Noetherian) and let $\fm$
be an ideal of $A$ which is intersection of maximal ideals
$\{\fm_1, \cdots , \fm_r\}$. Then for any $A$-module $M$ and any integer 
$n \ge 0$, the map ${M}/{\fm^nM} \to {M_{\fm}}/{\fm^nM_{\fm}}$ is an isomorphism.
\end{lem}
\begin{proof}
We first recall for the reader that $M_{\fm}$ is, by definition, the localization of
$M$ obtained by inverting all elements in $A \setminus 
(\stackrel{r}{\underset{j =1}\cup} \fm_i)$.
We next observe that $A_{\fm}$ must be a semi-local ring with maximal ideals
$\{\fm_1A_{\fm}, \cdots, \fm_rA_{\fm}\}$. In particular, any element of
$A_{\fm}$ which is not in any of these maximal ideals must be a unit.
This immediately implies the lemma when $M = A$. The general case follows from
this because we can now write
${M}/{\fm^nM} \simeq M \otimes_A {A}/{\fm^n}
\simeq (M \otimes_A A_{\fm}) \otimes_{A_{\fm}} {A_{\fm}}/{\fm^nA_{\fm}} 
\simeq {M_{\fm}}/{\fm^nM_{\fm}}$.
\end{proof}

Our main result on the two completions of the $K$-theory is 
the following.

\begin{thm}\label{thm:Com-compl}
Let $\sX$ be a separated quotient Deligne-Mumford stack of finite type over 
$\C$. Given any presentation $\sX = [X/G]$ and integer $i \ge 0$,
the following hold.
\begin{enumerate}
\item
$\fm^n_{\sX}G_i(\sX)_{\fm_{\sX}} = 0$ for all $n \gg 0$.
\item
There is a commutative diagram
\begin{equation}\label{eqn:Com-compl-0}
\xymatrix@C1pc{
G_i(\sX)_{\fm_{1}} \ar[r] \ar[d] & \wh{G_i(\sX)}_{\fm_{1}} \ar[d] \\
G_i(\sX)_{\fm_{\sX}} \ar[r] & \wh{G_i(\sX)}_{\fm_{\sX}},}
\end{equation}
in which all arrows are isomorphisms.
\end{enumerate}
\end{thm}
\begin{proof}
We first observe that
\begin{equation}\label{eqn:Com-compl-2}
\wh{G_i(\sX)}_{\fm_{\sX}} = {\underset{n}\varprojlim} \
{G_i(\sX)}/{\fm^n_{\sX}G_i(\sX)} \simeq 
{\underset{n}\varprojlim} \ {G_i(\sX)_{\fm_{\sX}}}/{\fm^n_{\sX}G_i(\sX)_{\fm_{\sX}}}, 
\end{equation}
where the second isomorphism follows from \lemref{lem:Elem-CAlg}.
Note here that $K_0(\sX)$ may not be a Noetherian ring. We similarly have 
\begin{equation}\label{eqn:Com-compl-3}
\wh{G_i(\sX)}_{\fm_1} = {\underset{n}\varprojlim} \
{G_i(\sX)}/{\fm^n_1G_i(\sX)} \simeq {\underset{n}\varprojlim} \
{G_i(\sX)_{\fm_1}}/{\fm^n_{1}G_i(\sX)_{\fm_{1}}}. 
\end{equation}

We shall prove the theorem by the Noetherian induction on $X$.
We first assume that $X$ is a smooth scheme. In this case, we have $K(\sX) 
\xrightarrow{\simeq} G(\sX)$.
We know from \cite[Lemma~7.3]{AK1} that  
$\fm^n_1K_i(\sX)_{\fm_1} = 0$ for $n \gg 0$. Since $K_i(\sX)_{\fm_{\sX}}$ is
a localization of $K_i(\sX)_{\fm_1}$, we get $\fm^n_{1}K_i(\sX)_{\fm_{\sX}} = 0$ for 
$n \gg 0$. On the other hand, since
each $K_i(\sX)$ is a $K_0(\sX)$-module, it follows from  
\cite[Theorem~6.1]{EG5} that for every $n \ge 0$, one has
\begin{equation}\label{eqn:Com-compl-1}
\fm^{n'}_{1} K_i(\sX) \subseteq \fm^{n'}_{\sX}K_i(\sX) \subseteq 
\fm^n_1K_i(\sX) \ \mbox{for} \  n' \gg 0.
\end{equation}

We conclude from this that $\fm^n_{\sX}K_i(\sX)_{\fm_{\sX}} = 0$ for all $n \gg 0$. 
It follows now from 
~\eqref{eqn:Com-compl-2} and ~\eqref{eqn:Com-compl-3} that the bottom and
the top horizontal arrows in ~\eqref{eqn:Com-compl-0} are isomorphisms.
It follows from ~\eqref{eqn:Com-compl-1} that the $\fm_1$-adic and
$\fm_{\sX}$-adic topologies on $K_i(\sX)$ coincide. Hence, the 
right vertical arrow in ~\eqref{eqn:Com-compl-0} is an isomorphism.
We conclude that the left vertical arrow is also an isomorphism. The theorem is
thus proven when $X$ is a smooth scheme.

In general, there is a non-empty $G$-invariant 
open subspace $U \subset X$ which is a smooth scheme, so that the open substack
$\sU = [U/G] \subset \sX$ is smooth. We let $Z = X \setminus U$ and 
$\sZ = [Z/G]$. Let $\sZ \xrightarrow{\iota} \sX \xleftarrow{j} \sU$ denote the
inclusion maps. 

It is easy to check that $\fm_{\sX}K_0(\sZ) \subset \fm_{\sZ}K_0(\sZ)$.
It is also immediate from the definition of $\fm_{\sX}$ that
$\fm_{\sZ} \cap K_0(\sX) = \fm_{\sX} = \fm_{\sU} \cap K_0(\sX)$.
It follows that we have natural maps of localizations
$G_i(\sZ)_{\fm_1} \to G_i(\sZ)_{\fm_{\sX}} \to G_i(\sZ)_{\fm_{\sZ}}$.
It follows by the Noetherian induction that the composite map
is an isomorphism. In particular, the second map is surjective.
We claim that this map is also injective, and hence an isomorphism.
To prove this, we note that showing injectivity of this map is 
equivalent to showing that given any $a \in G_i(\sZ)$
and $s \in K_0(\sZ) \setminus \fm_{\sZ}$ such that $sa = 0$, there is 
$t \in K_0(\sX) \setminus \fm_{\sX}$ such that $ta = 0$. 
However, the injectivity of the map $G_i(\sZ)_{\fm_1} \to G_i(\sZ)_{\fm_{\sZ}}$
implies that there is $u \in R(G) \setminus \fm_1$ such that $ua = 0$.
As $\fm_{\sX} \cap R(G) = \fm_1$, if we let $t$ be the 
image of $u$ under the map 
$R(G) \to K_0(\sX)$, we get $t \notin \fm_{\sX}$ and $ta = 0$.

We similarly have the natural maps of localizations
$G_i(\sU)_{\fm_1} \to G_i(\sU)_{\fm_{\sX}} \to G_i(\sU)_{\fm_{\sU}}$
and the composite map is an isomorphism since $\sU$ is smooth.
An identical reason as before now shows that these maps are
isomorphisms.

We have the exact sequence of $K_0(\sX)$-modules
\begin{equation}\label{eqn:Com-compl-4}
G_i(\sZ) \to G_i(\sX) \to G_i(\sU)
\end{equation}
and hence this remains exact after localization at $\fm_{\sX}$.
For reader unfamiliar with the fact that the above is a sequence of 
$K_0(\sX)$-modules, recall that $K(\sX)$ is a ring spectrum and 
$G(\sZ) \to G(\sX) \to G(\sU)$
is a fiber sequence in the stable homotopy category of module spectra over the
ring spectrum $K(\sX)$. In particular, the associated long exact sequence of
homotopy groups is a sequence of $\pi_0(\sK(\sX))$-modules. But
$\pi_0(\sK(\sX))$ is $K_0(\sX)$ by definition.

Now, we know that $\fm^n_{\sU} G_i(\sU)_{\fm_{\sU}} = 0$ as $\sU$ is smooth.
Since $G_i(\sU)_{\fm_{\sX}} \xrightarrow{\simeq} G_i(\sU)_{\fm_{\sU}}$
as shown above, and since $\fm^n_{\sX}G_i(\sU) \subset \fm^n_{\sU}G_i(\sU)$,
it follows that $\fm^n_{\sX}G_i(\sU)_{\fm_{\sX}} = 0$ for all $n \gg 0$.
Similarly, we have $\fm^n_{\sX}G_i(\sZ)_{\fm_{\sX}} = 0$ for $n \gg 0$. 
It easily follows from this and ~\eqref{eqn:Com-compl-4} that
$\fm^n_{\sX}G_i(\sX)_{\fm_{\sX}} = 0$ for $n \gg 0$. This proves (1).

We now prove (2). It follows from (1) and ~\eqref{eqn:Com-compl-2}
that the bottom horizontal arrow in ~\eqref{eqn:Com-compl-0} is an isomorphism.
It follows from \cite[Lemma~7.3]{AK1} that $\fm^n_1G_i(\sX)_{\fm_1} = 0$
for $n \gg 0$. Combining this with ~\eqref{eqn:Com-compl-3}, it follows that
the top horizontal arrow in ~\eqref{eqn:Com-compl-0} is an isomorphism.
It is now enough to show that the left vertical arrow is an isomorphism.

To prove this, we consider the commutative diagram 
\begin{equation}\label{eqn:Com-compl-5}
\xymatrix@C1pc{
G_{i+1}(\sU)_{\fm_1} \ar[r] \ar[d] & G_i(\sZ)_{\fm_1} \ar[r] \ar[d] & 
G_i(\sX)_{\fm_1} \ar[r] \ar[d] & G_i(\sU)_{\fm_1} \ar[r] \ar[d] & G_{i-1}(\sZ)_{\fm_1}
\ar[d] \\
G_{i+1}(\sU)_{\fm_{\sX}} \ar[r] & G_i(\sZ)_{\fm_{\sX}} \ar[r]  & 
G_i(\sX)_{\fm_{\sX}} \ar[r] & G_i(\sU)_{\fm_{\sX}} \ar[r] & G_{i-1}(\sZ)_{\fm_{\sX}}.}
\end{equation}

Since the localization exact sequence for the $K$-theory 
(of coherent sheaves) is a sequence of $K_0(\sX)$-modules,
%Since the functors $Coh_{\sZ} \inj Coh_{sX} \to Coh_{\sU}$ commute with the tensor
%product of vector bundles on $\sX$. Hence the corresponding sequence of homotopy
%groups also commute. 
it follows that this sequence remains exact after localization at $\fm_{\sX}$. 
Similarly, it remains exact after localization
at $\fm_1 \subset R(G)$. We conclude that the top and the bottom rows of
~\eqref{eqn:Com-compl-5} are exact. We have shown that all vertical arrows
in this diagram (except possibly the middle one) are isomorphisms.
It follows that the middle one must also be an isomorphism.
\end{proof}

\subsection{K{\"o}ck's conjecture}\label{sec:K-conj}
A conjecture of K{\"o}ck \cite[Conjecture~5.6]{Kock} in equivariant
$K$-theory asserts that if $f: X \to Y$ is a $G$-equivariant proper map
of schemes such that $f$ is a local complete intersection, then it induces
a push-forward map of completions $f_*: \wh{K_i(G,X)}_{\fm_{[X/G]}} \to
\wh{K_i(G,Y)}_{\fm_{[Y/G]}}$. 
The main result of \cite{Kock} is based on the validity of this conjecture.
If $X$ and $Y$ are smooth, this conjecture was
settled by Edidin and Graham \cite{EG5}. 

The following consequence of
\thmref{thm:Com-compl} proves K{\"o}ck's conjecture when 
$Y$ is smooth but
$X$ is possibly singular with proper $G$-actions.
In fact, the result below is stronger than the one predicted by
K{\"o}ck's conjecture because it holds for any proper map of
stacks which may not be representable. 
Recall that a stack $\sX$ satisfies the {\sl resolution property} if
every coherent sheaf on $\sX$ is a quotient of a vector bundle.
If $\sX$ is a quotient Deligne-Mumford stack with quasi-projective coarse
moduli space, then it is known that it satisfies the resolution property
(see \cite[Proposition~5.1]{KR}).  

\begin{cor}\label{cor:Com-compl-cons}
Let $f: \sX \to \sY$ be a proper morphism of 
separated quotient Deligne-Mumford stacks of finite type over $\C$.
Let $i \ge 0$ be an integer. Then there is a push-forward map 
$f_*: \wh{G_i(\sX)}_{\fm_{\sX}} \to \wh{G_i(\sY)}_{\fm_{\sY}}$.
If $\sY$ is smooth and satisfies the resolution property, 
then there is a push-forward map 
$f_*: \wh{K_i(\sX)}_{\fm_{\sX}} \to \wh{K_i(\sY)}_{\fm_{\sY}}$.
\end{cor}
\begin{proof}
By ~\eqref{eqn:st-gen*-0}, we can write $f$ as the product 
$\sX \simeq [Z/G] \xrightarrow{p} [W/F] \xrightarrow{q} [Y/F] = \sY$,
where $G = H \times F$ and $[Z/H] \to W$ is the coarse moduli space
and $W \to Y$ is $F$-equivariant.
We thus get maps $G_i(G, Z)_{\fm_{1_G}} \to G_i(G, Z)_{\fm_{1_F}} \to
G_i(F,W)_{\fm_{1_F}} \to G_i(F,Y)_{\fm_{1_F}}$ whose composite is $f_*$.
It follows now from \thmref{thm:Com-compl} that this composite map is
same as the map $f: G_i(\sX)_{\fm_{\sX}} \to G_i(\sY)_{\fm_{\sY}}$. 
This proves the first part.
If $\sY$ is smooth and satisfies the resolution property, we get maps
$\wh{K_i(\sX)}_{\fm_{\sX}} \to \wh{G_i(\sX)}_{\fm_{\sX}} \xrightarrow{f_*}
\wh{G_i(\sY)}_{\fm_{\sY}} \xleftarrow{\simeq} \wh{K_i(\sY)}_{\fm_{\sY}}$
and this completes the proof.
\end{proof}

\begin{remk}\label{remk:twocompletions1}
	We remark that it is not necessary to assume in \thmref{thm:Com-compl}
	that $\sX=[X/G]$ is separated. It is enough to 
assume that the $G$-action has finite stabilizers. 
We also note that $\fm_{X}$ can be defined without the assumption that 
$[X/G]$ has a coarse moduli space (see \cite[\S~4.2]{AK1}).   
	
\end{remk}

\begin{remk}\label{remk:kocksconjecture}
	Further, in \corref{cor:Com-compl-cons}, we make 
	the separatedness assumption because
 we need to use the factorization constructed in 
\lemref{lem:GRR-QDM-stacks}. If we assume $f$
to be representable (this was the setting of K{\"o}ck's conjecture), 
then \corref{cor:Com-compl-cons} holds more generally when
the stacks $\sX$ and $\sY$ have finite stabilizers. 
%We kept this assumption only to be
%consistent with the universal assumption in the rest of the paper.
\end{remk}

\begin{comment}
\begin{remk}\label{remk:Kock-arb}
We remark that it is not necessary for the ground field to be $\C$ in
\thmref{thm:Com-compl} and its proof. We kept this assumption only to be
consistent with the universal assumption on the ground field in the rest
of this text. In particular, Corollary~\ref{cor:Com-compl-cons} and its 
proof are valid over any field. We also remark that \corref{cor:Com-compl-cons}
is stronger than the one predicted by K{\"o}ck's conjecture since we do not
assume $f$ to be representable.
\end{remk}

Since $f$ is representable, letting $\sY = [Y/G]$, we get a Cartesian
square
\begin{equation}\label{eqn:Com-compl-2}
\xymatrix@C1pc{
X \ar[r] \ar[d]_{g} & \sX \ar[d]^{f} \\
Y \ar[r] & \sY,}
\end{equation}.
where $X$ and $Y$ are algebraic spaces. It follows from the definition 
of the quotient stacks that $X$ is an algebraic space with $G$-action
and $g$ is $G$-equivariant and proper. The first part of the corollary now
follows from \thmref{thm:Com-compl}.
\end{comment}

%\begin{comment}
\subsection{The Atiyah-Segal map}\label{sec:AS-map*}
Let $G$ be a linear algebraic group over $\C$ and let $X$ be an algebraic
space with a proper $G$-action. We now define our Atiyah-Segal map.
The next few sections will be devoted to proving its functorial properties.

\begin{defn}\label{defn:RR-mapconjclass-AS}
Let $\psi$ be a semi-simple conjugacy class in $G$ and let $g \in \psi$. 
For any integer $i \ge 0$, we let $\vartheta^{\psi}_X$ denote the composite map
\begin{equation}\label{eqn:RR-AS}
G_i(G,X_{\psi})_{\fm_{\psi}} \xrightarrow{\omega_g^{-1} } 
G_i(Z_g,X^g)_{\fm_g} \xrightarrow{t_{g^{-1}}} 
G_i(Z_g,X^g)_{\fm_1} \xrightarrow{\mu^1_g} G_i(G,I^{\psi}_X)_{\fm_1}.
\end{equation}

Note that all arrows in ~\eqref{eqn:RR-AS} are isomorphisms.
\end{defn}

\begin{comment}
\begin{defn}\label{defn:RR-mapconjclass}
Let $X$ be a quasi-projective scheme with a proper $G$-action such that 
$Y = X/G$ is quasi-projective. For any integer $i \ge 0$, 
we define the Riemann-Roch transformation $\tau^\psi_X$ to be the composite
\begin{equation}\label{eqn:RR-def-0}
\tau^\psi_X: G_i(G,X_{\psi})_{\fm_{\psi}} \xrightarrow{\vartheta^{\psi}_X}
G_i(G,I^{\psi}_X)_{\fm_1} \xrightarrow{t^G_{I^\psi_X}} \CH_*^G(I^{\psi}_X,i).
\end{equation}

If $i = 0$ and $X$ is an algebraic space with proper $G$-action, we let 
$\tau^{\psi}_X$ be the composite of the maps in ~\eqref{eqn:RR-def-0}.

Note that all arrows in ~\eqref{eqn:RR-def-0} are isomorphisms, where the
second isomorphism follows from \thmref{thm:AKRR}.
\end{defn}
\end{comment}

\begin{lem}\label{lem:Independence of w_g on g}
The map $\vartheta^{\psi}_X$ is independent of the choice 
of $g$ in $\psi$. That is, the composite maps 
$\mu^1_g \circ t_{g^{-1}} \circ \omega_{g^{-1}}$ and 
$\mu^{1}_h \circ t_{h^{-1}} \circ \omega_{h^{-1}}$ 
coincide if $g, h \in \psi$ .
%In particular, the map $\mu^1_g \circ t_{g^-1} \circ 
%\mu_g^{-1}: G_i(G, I^{\psi}_X) \rightarrow G_i(G,I^{\psi}_X)$ 
%coincides with the map $\mu^{\psi}_h \circ t_{h^{-1}} \circ \mu_{h}^{-1}$.
\end{lem}
%\end{comment}
\begin{proof}
We need to show that the diagram
\begin{equation}\label{eqn:Indep-0}
\xymatrix@C1pc{
G_i(G,X)_{\fm_{\psi}} \ar[r]^-{\omega_g^{-1} } \ar[d]_{Id} &  
G_i(Z_g,X^g)_{\fm_g} \ar[r]^-{t_{g^{-1}}} \ar[d]^{\theta_X} & G_i(Z_g,X^g)_{\fm_1} 
\ar[r]^-{\mu^1_g} \ar[d]^{\theta_X} &  G_i(G,I^{\psi}_X)_{\fm_1} \ar[d]^{Id}  \\
G_i(G,X)_{\fm_{\psi}} \ar[r]^-{\omega_h^{-1} } & G_i(Z_h,X^h)_{\fm_h} 
\ar[r]^-{t_{h^{-1}}} & G_i(Z_h,X^h)_{\fm_1} \ar[r]^-{\mu^{1}_h} & 
G_i(G,I^{\psi}_X)_{\fm_1}}
\end{equation}
commutes.
\lemref{lem:thetahandgindependence} says that the right square
commutes. To show that the left square commutes, we
can replace the horizontal arrows in this square by their inverses. 
In this case, the desired commutativity follows from \lemref{lem: w_g}. 
We are now left with showing that the middle square commutes.

Recall from ~\eqref{diag:defnofphi_*}
that $u_*: G_i(Z_g,X^g) \to G_i(Z_g, X^h)$ is an isomorphism, where 
$X^h$ is given the $\star$-action of $Z_g$ via $\phi$. 
As $u_*$ is a $Z_g$-equivariant isomorphism, it follows 
from Proposition~\ref{prop:invariantsfunct} that 
$u_* \circ t_{g^{-1}} = t_{g^{-1}} \circ u_*$.
Since $\theta_X = \phi_* \circ u_*$, it suffices to show that
$t_{g^{-1}}$ commutes with $\phi_*$.
First, one checks directly from the definitions of $t_{g^{-1}}$
and $\phi_*$ that the square
\begin{equation}\label{eqn:Indep-1}
\xymatrix@C1pc{
R(Z_g)_{\fm_g} \ar[r]^-{t_{g^{-1}}} \ar[d]_{\phi_*} & R(Z_g)_{\fm_1}
\ar[d]^{\phi_*} \\
R(Z_h)_{\fm_h} \ar[r]^-{t_{h^{-1}}} & R(Z_h)_{\fm_1}}
\end{equation}
commutes. This can also be deduced from \cite[Lemma~6.5]{EG2}. Indeed,
an element of $R(Z_g)$ is a virtual representation of $Z_g$,
which we can assume to be an actual representation where ${\<g\>}$ acts
by a fixed character $\chi$ (see \lemref{lem:definitionoftwist}).
If $V$ is such a representation, then $\phi_*([V])$ is just $V$ itself,
considered as a representation of $Z_h$ via the isomorphism
$\phi^{-1}: Z_h \xrightarrow{\simeq} Z_g$. We thus have
$t_{h^{-1}} \circ \phi_*([V]) = \chi(\phi^{-1}(h))\cdot [V]$.
But this is clearly same as $\chi(g)\cdot[V] = \phi_* \circ t_{g^{-1}}([V])$.

In general, for $a \in G^{\chi}_i(\<g\>,X^g)$, we have
\[
\phi_* \circ t_{g^{-1}} (a)  = \phi_*(\chi(g) a) = 
\phi_*(t_{g^{-1}}([1_{Z_g}])) \phi_*(a) 
\ {=}^{\dagger} \ t_{h^{-1}}([1_{Z_g}]) \phi_*(a) 
=  t_{h^{-1}} \circ \phi_*(a),
\]
where $1_{Z_g}$ is the rank one trivial representation of $Z_g$  
and ${=}^{\dagger}$ holds by case of $\phi_*: R(Z_g) \xrightarrow{\simeq} R(Z_h)$
shown above. Since $G_i(Z_g,X^g) = \oplus_{\chi \in \wh{\<g\>}} 
G^{\chi}_i(\<g\>,X^g)$ by \lemref{lem:definitionoftwist},
we have shown that the middle square in ~\eqref{eqn:Indep-0} commutes. 
This finishes the proof.
\end{proof}

\begin{comment}
Using $\theta_X \circ t_{g^{-1}} = t_{h^{-1}} \circ \theta_X$, shown in 
\lemref{lem:Independence of w_g on g}, and \lemref{lem:lamdaindependence}, 
we get

\begin{cor}\label{cor:Indep-2}
Given any $g, h \in \psi $, the composite map $\mu^1_g \circ t_{g^-1} \circ 
\mu_g^{-1}: G_i(G, I^{\psi}_X) \rightarrow G_i(G,I^{\psi}_X)$ 
coincides with the map $\mu^{\psi}_h \circ t_{h^{-1}} \circ \mu_{h}^{-1}$.
\end{cor}

\[
\begin{array}{lllll}
\phi_* \circ t_{g^{-1}} (a) & = & \phi_*(\chi(g) a) & = &
\phi_*(t_{g^{-1}}([1_{Z_g}])) \cdot \phi_*(a) \\
& {=}^{\dagger} & t_{h^{-1}}([1_{Z_g}])) \phi_*(a) 
& = & t_{h^{-1}} \circ \phi_*(a),
\end{array}
\]
\end{comment}

Recall from \S~\ref{sec:Inertia} that there is a $G$-equivariant
decomposition $I_X = \amalg_{\psi \in \Sigma^G_X} I^{\psi}_X$ 
and hence we have $\oplus_{\psi \in \Sigma^G_X} G_i(G,I^{\psi}_X) 
\xrightarrow{\simeq} G_i(G,I_X)$. 
Using \thmref{thm:closedimmiso} and Proposition~\ref{prop:direcsumdecomp}, 
we get the following.

\begin{defn}\label{defn:RR-mapfull}
Let $G$ act properly on an algebraic space $X$.
For any integer $i \ge 0$, the Atiyah-Segal map 
$\vartheta^G_X: G_i(G,X) \to G_i(G,I^{\psi}_X)_{\fm_1}$
is defined to be the composite
\begin{equation}\label{diag:RR-mapfull}
\vartheta^G_X : G_i(G,X) \xrightarrow{\simeq} \oplus_{\psi \in \Sigma^G_X} 
G_i(G, X_{\psi})_{\fm_{\psi}} 
\xrightarrow{\oplus {\vartheta^{\psi}_X}}
\oplus_{\psi \in \Sigma^G_X} G_i(G,I^{\psi}_X)_{\fm_1} 
\xrightarrow{\simeq} G_i(G,I_X)_{\fm_1}.
\end{equation}
\end{defn}

Since each $\vartheta^{\psi}_X$ is an isomorphism, it follows that $\vartheta^G_X$
is an isomorphism.

\begin{comment}
\vskip .3cm

%\begin{remk}\label{remk:EGRR}
\subsection{Agreement with the Riemann-Roch map of Edidin 
and Graham}\label{sec:EGRR}
If $X$ is smooth and $i = 0$, then an equivariant Riemann-Roch map was
constructed by Edidin and Graham \cite[Theorem~6.7]{EG2}.
Using the non-abelian localization theorem
of \cite{EG4}, we can check that the map 
$\tau^G_X$ of ~\eqref{diag:RR-mapfull} coincides
with that of \cite{EG2}.
 
It is enough to check this at every $\psi$.
Let $p_g$ denote the projection $G_*(Z_g, X^g)_{\fm_{\psi}} \surj 
G_*(Z_g, X^g)_{\fm_g}$ and let $l_f = \lambda_{-1}(N_f^*)^{-1}$
(see \cite[\S~5.4]{EG4}). To check the agreement, all we need to show is that
$l_f \circ p_g \circ \mu^{-1}_g \circ f^* = \omega^{-1}_g \circ (j^{\psi}_*)^{-1}$,
where $j^\psi: X_{\psi} \inj X$ is the inclusion,
$j^{\psi}_*: G_*(G, X_{\psi})_{\fm_{\psi}} \to  G_*(G, X)_{\fm_{\psi}}$
is the isomorphism of \thmref{thm:closedimmiso} and 
$f = j^\psi \circ \mu^\psi: I^\psi_X \to X$ is a map of smooth algebraic
spaces. Equivalently, we need to show that
$j^{\psi}_* \circ \omega_g \circ l_f \circ p_g \circ \mu^{-1}_g \circ f^*$
is identity. Using our definition of $\omega_g$ and $j^{\psi}_* \circ \mu^\psi_* 
= f_*$, this is equivalent to showing that 
$f_* \circ \mu^\psi_g \circ l_f \circ p_g \circ \mu^{-1}_g \circ f^*$ 
is identity. In the notations of \cite[\S~5.4]{EG4}, this is equivalent to
saying that $f_*((l_f \cap f^*(\alpha))_{c_\psi}) = \alpha$ for
$\alpha \in  G_*(G, X)_{\fm_{\psi}}$. But this is precisely the
identity (13) of \cite{EG4}.
\end{comment}

\section{Atiyah-Segal correspondence for the coarse 
moduli space map}\label{Sec:GRR-CoaMod}
\begin{comment}
From now on and until the last section of this text, we shall assume
that $G$ is a linear algebraic group acting properly on a
$\C$-scheme $X$ such that quotient $X/G$ is quasi-projective.
Recall from \S~\ref{sec:basicassumptions} that 
the category of such schemes is denoted by $\qp^G_{\C}$, which is
a full subcategory of $\Sch^G_{\C}$.

%Note that this assumption implies that $X$ is quasi-projective with
%linear $G$-action (see \cite[Remark~4.3]{KR}).
%Recall that the category of such schemes is denoted by $\qp^G_{\C}$.

It will be immediate from various proofs of the previous sections
and the from the ones that will now follow that all these proofs work
identically for all algebraic spaces with proper $G$ action, if we
work with $G_0(G,X)$ and $\CH^G_*(X)$. This case will be considered in the
last section.  

\subsection{Riemann-Roch for representable morphisms
of quotient stacks}\label{sec:GRR-rep}
\end{comment}
The construction of the Atiyah-Segal correspondence for the $G$-theory of
stacks goes in two steps. The first is to show that the Atiyah-Segal map
is well defined and is an isomorphism. The second part is to show that this
map is covariant with respect to proper maps of separated quotient stacks.
We have shown the first part in the previous section. The more 
intricate second part will be
shown in several steps. In this section, we prove it for the representable
maps and the coarse moduli space maps.

\begin{prop}\label{prop:riemannrochmapfunct}
For any $G$-equivariant proper morphism $f : X \rightarrow Y$ of 
algebraic spaces 
with proper $G$-action and for any integer $i \ge 0$, 
there is a commutative diagram
\begin{equation}\label{eqn:GRR-I-0}
\xymatrix@C1pc{
G_i(G,X) \ar[r]^-{\vartheta^G_X} \ar[d]_{f_*} & 
G_i(G, I_X)_{\fm_1} \ar[d]^{f^I_*} \\
G_i(G,Y) \ar[r]^-{\vartheta^G_Y} & G_i(G,I_Y)_{\fm_1}.}
\end{equation}
\end{prop}
\begin{proof}
Using ~\eqref{diag:Inertiadefn} and ~\eqref{eqn:Inert-1}, 
it is easy to check that 
$f^I: I_X \to I_Y$ is a $G$-equivariant proper map which takes 
$I^{\psi}_X$ to $I^{\psi}_Y$. 
Since $f_*: G_i(G,X) \to G_i(G,Y)$ is $R(G)$-linear, it takes the summand 
$G_i(G,X)_{\fm_{\psi}} \to G_i(G,Y)_{\fm_{\psi}}$. 

Since $f$ is a $G$-equivariant map, it is clear that  
$\Sigma^G_X \subseteq \Sigma^G_Y$.
If $\psi$ is a semi-simple 
conjugacy class of $G$ that does not belong to $\Sigma^G_Y$,
then both $X_{\psi}$ and $Y_{\psi}$ are empty hence there is nothing to prove. 
If $\psi$ belongs to $\Sigma^G_Y \setminus \Sigma^G_X$,
then $G_i(G,X)_{\fm_{\psi}} = 0 $ and $I^{\psi}_X = \emptyset$.
%or $G_i(G,Y)_{\fm_{\psi}} = 0 $ and $I^{\psi}_Y = \emptyset$. 
So $f_*$ and $f^I_*$ are both zero.
We can thus assume that $\psi \in \Sigma^G_X$. 
This reduces us to showing that $\vartheta^{\psi}_Y \circ f_* =
f^I_* \circ \vartheta^{\psi}_X: G_i(G,X)_{\fm_{\psi}} \to 
G_i(G, I^{\psi}_Y)_{\fm_1}$.

Now, going back to ~\eqref{eqn:RR-AS}, we note that 
the map $t_{g^{-1}}$ commutes with $f_*$ by 
\corref{cor:proper-T-commute}. The maps $\mu^{\psi}_g$ and $\mu^1_g$
commute with $f_*$ and $f^I_*$ by Lemma~\ref{lem:moritaandrrproper}. Since
\begin{equation}\label{eqn:GRR-I-1}
\xymatrix@C1pc{
I^{\psi}_X \ar[d]_{\mu^{\psi}} \ar[r]^-{f^I} & I^{\psi}_Y \ar[d]^{\mu^{\psi}} \\
X \ar[r]^-{f} & Y} 
\end{equation}
is a commutative square of $G$-equivariant proper maps, it follows
that $\mu^{\psi}_*$ and $f_*$ commute. Since $\omega_g = \mu^{\psi}_* \circ 
\mu^{\psi}_g$, it follows that $f_* \circ \omega_g = \omega_g \circ f_*$.
Since $\omega_g$ is invertible, we get $f_* \circ \omega^{-1}_g = \omega^{-1}_g 
\circ f_*$.  We thus get $\vartheta^{\psi}_Y \circ f_* =
f^I_* \circ \vartheta^{\psi}_X$ and this finishes the proof.
\end{proof}

\subsection{The case of coarse moduli space map}\label{SubSec:GRR-CoaMod}
Our goal now is to prove that the Atiyah-Segal map is covariant 
with respect to the coarse moduli space map $[X/G] \to X/G$.

Let $X$ be an algebraic space with proper $G$-action and 
let $Y = X/G$ be the quotient. 
Let $\psi \subset G$ be a 
semi-simple conjugacy class. We have the $G$-equivariant closed immersion 
$j^{\psi}_X: X_{\psi} \subset X$ and a $G$-equivariant finite and surjective map
$\mu^{\psi}: I^{\psi}_X \to X_{\psi}$. Let $\nu^{\psi}: {I^{\psi}_X}/G \to  
{X_{\psi}}/G := Y_{\psi}$ denote the induced map on the quotients.
This is clearly finite and surjective. We fix an integer $i \ge 0$.
 
\begin{lem}\label{lem:semisimpleconjclass}
There is a commutative diagram
\begin{equation}\label{eqn:Aug-2}
\xymatrix@C1pc{
G_i(G,X_{\psi})_{\fm_{\psi}} \ar[r]^-{\vartheta^{\psi}_X} \ar[d]_{\inv^G_{X_{\psi}}} & 
G_i(G, I^{\psi}_X)_{\fm_1} \ar[d]^{\nu^{\psi}_* \circ \inv^G_{I^{\psi}_X}} \\
G_i(Y_{\psi}) \ar[r]^-{Id} & G_i(Y_{\psi}).}
\end{equation}
\end{lem}
\begin{proof}
We consider the diagram
\begin{equation}\label{eqn:Aug-3}
\xymatrix@C1pc{
G_i(G,X_{\psi})_{\fm_{\psi}} \ar[r]^-{\omega_g^{-1}} 
\ar[dd]_{\inv^G_{X_{\psi}}} & G_i(Z_g,X^g)_{\fm_g} \ar[r]^-{t_{g^{-1}}} 
\ar[d]^{\inv^{Z_g}_{X^g}} & G_i(Z_g,X^g)_{\fm_1} \ar[r]^-{\mu^1_g} 
\ar[d]^{\inv^{Z_g}_{X^g}} & G_i(G,I^{\psi}_X)_{\fm_1} 
\ar[d]^{\inv^G_{I^{\psi}_X}} \\
& G_i(X^g/Z_g) \ar[r]^-{Id} \ar[ld]_{\nu^{\psi}_* \circ \gamma_{g*}} & G_i(X^g/Z_g) 
\ar[r]^-{\gamma_{g*}} & G_i(I^{\psi}_X/G)  
\ar[d]^{\nu^{\psi}_*} \\
G_i(Y_{\psi}) \ar[rrr]^-{Id} &&& G_i(Y_{\psi}).}
\end{equation}

The lemma is equivalent to the commutativity of the big outer square.
Using the $G$-equivariant isomorphism $p_g: G \times^{Z_g} X^g
\xrightarrow{\simeq} I^{\psi}_X$, we note from ~\eqref{eqn:mu-g-*}
that $\mu^1_g$ is the push-forward isomorphism induced on $K$-theory via the 
isomorphism of quotient stacks 
$\iota^{GZ_g}_{X^g}: [{X^g}/{Z_g}]  \xrightarrow{\simeq} 
[{I^{\psi}_X}/G]$ given by \lemref{lem:Morita-stack*}.
%If we let $W_g = {X^g}/{Z^g}$, we get
%$G \times^{Z_g} W^g \simeq {G/{Z_g}} \times W^g \simeq {I^{\psi}_X}/G$.
We have a commutative square of separated quotient stacks and 
coarse moduli spaces
%It follows from \lemref{lem:Morita-stack*} that the $Z_g$-equivariant map
%$\pi_g: X^g \to {X^g}/{Z^g}$ (with respect to the trivial $Z_g$-action on 
%${X^g}/{Z^g}$) induces a commutative square of stacks 
\begin{equation}\label{eqn:Aug-4}
\xymatrix@C1pc{
[{X^g}/{Z_g}] \ar[r]^-{\iota^{GZ_g}_{X^g}}_{\simeq} \ar[d]_{\pi_g} & 
[{I^{\psi}_X}/G] \ar[d]^{\pi^I_g} \\
{X^g}/{Z_g} \ar[r]^{\gamma_g}_{\simeq} & {I^{\psi}_X}/G,}
\end{equation}
where the vertical arrows are proper. Taking the corresponding maps
on the $K$-theory, we see that the top square from the extreme right 
in ~\eqref{eqn:Aug-3} is commutative.

The middle square on the top in ~\eqref{eqn:Aug-3} commutes by
\propref{prop:invtwisst}. We are now left with showing that the
trapezium on the left side of ~\eqref{eqn:Aug-3} commutes.
For any $a \in G_i(G,X)_{\fm_{\psi}}$, we have
\[
\begin{array}{lllll}
{\inv}_{X_{\psi}}^G (a) & = & {\inv}_{X_{\psi}}^G \circ \omega_g \circ 
\omega_g^{-1} (a) & = & {\inv}_{X_{\psi}}^G \circ 
\mu^{\psi}_* \circ \mu^1_g \circ \omega^{-1}_g (a) \\
& {=}^{1} & \nu^{\psi}_* \circ {\inv}^G_{I^{\psi}_X}  \circ \mu^1_g 
\circ \omega_g^{-1} (a) & {=}^2 & 
\nu^{\psi}_* \circ \gamma_{g*} \circ {\inv}^{Z_g}_{X^g}  \circ 
\omega_g^{-1} (a),
\end{array}
\]
where ${=}^1$ follows from \thmref{thm:invariants} and ${=}^2$ follows from
~\eqref{eqn:Aug-4}. This proves the desired commutativity and
finishes the proof of the lemma.  
\end{proof}

\begin{thm}\label{thm:GRR-CoaMod}
Let $X$ be an algebraic space with proper $G$-action and let
$Y = X/G$ be the quotient. For a semi-simple conjugacy class
$\psi \subset G$, let $j^{\psi}_Y: Y_{\psi} \inj Y$ denote the closed immersion, 
where $Y_{\psi} = {X_{\psi}}/G$. Then there is a commutative diagram
\begin{equation}\label{diag:GRR-CoaMod}
\xymatrix@C1pc{
G_i(G,X) \ar[rr]^-{\vartheta^G_X} \ar[dd]_{\inv^G_X} && \oplus_{\psi \in S_G} 
G_i(G,I^{\psi}_X)_{\fm_1} \ar[d]^{\oplus_{\psi \in S_G}
{\inv^G_{X_{\psi}} \circ \mu_*^{\psi}}} \\
&&  \oplus_{\psi \in S_G} G_i(Y_{\psi}) 
\ar[d]^{\oplus_{\psi \in S_G} j^{\psi}_{Y *}}     \\
G_i(Y) \ar[rr]^-{Id} && G_i(Y).}
\end{equation}
\end{thm}
\begin{proof}
It suffices to show that the diagram commutes when restricted to each 
$G_i(G,X)_{\fm_{\psi}}$ such that $\psi \in \Sigma^G_X$.
Continuing with the notations set up before and in 
\lemref{lem:semisimpleconjclass}, let $a \in G_i(G,X)_{\fm_{\psi}}$.
We then have
\[
\begin{array}{lllll}
\inv^G_{X} (a) & {=}^1 & \inv^G_{X} 
\circ j^{\psi}_{X *}(a) & {=}^2 &  j^{\psi}_{Y *} \circ \inv^G_{X_{\psi}} (a) \\
& {=}^3 & j^{\psi}_{Y *} \circ \nu^{\psi}_* \circ \inv^G_{I^{\psi}_X}
\circ \vartheta^{\psi}_X (a) 
& {=}^4 & j^{\psi}_{Y *} \circ \inv^G_{X_{\psi}} \circ \mu^{\psi}_* \circ 
\vartheta^{\psi}_X (a).
\end{array}
\]
The equality ${=}^1$ follows from \thmref{thm:closedimmiso}, (2) follows from
\thmref{thm:invariants} and
${=}^3$ follows from \lemref{lem:semisimpleconjclass}.
The equality ${=}^4$ follows from the functoriality of the proper push-forward
map in $K$-theory of coherent sheaves, applied to the commutative square
of proper maps of stacks
\begin{equation}\label{eqn:stack-proper}
\xymatrix@C1pc{
[{I^{\psi}_X}/G] \ar[r]^-{\mu^{\psi}} \ar[d] & [{X_{\psi}}/G] \ar[d] \\
{I^{\psi}_X}/G \ar[r]^-{\nu^{\psi}} & Y_{\psi}}
\end{equation}
and noting that $\inv^G_{(-)}$ is the push-forward map on $K$-theory induced
by the coarse moduli space map which is proper (see \thmref{thm:invariants}). 
This proves the theorem.
\end{proof}

\section{Atiyah-Segal correspondence for a partial quotient 
map}\label{Sec:Product Of Groups}
In this section, we shall prove a version of the Atiyah-Segal correspondence
for a map which is the quotient by one of the factors 
of a product of groups acting properly on an algebraic space.
More specifically, we shall work with the following set up.

Let $G = H \times F$ be the product of two linear algebraic groups
which acts properly on an algebraic space $X$ with quotients 
$Y = X/H$ and $Z = X/G$. 
Let $g = (g_1,g_2) \in H \times F = G$ be a semi-simple element. 
Let $Z_{g_1} = \mathcal{Z}_H(g_1)$ and $Z_{g_2} =  \mathcal{Z}_F(g_2)$ be the 
centralizers of $g_1$ in $H$ and $g_2$ in $F$, respectively so that
$Z_{g} = \sZ_G(g) = Z_{g_1} \times Z_{g_2}$. Let $\psi$ denote 
the conjugacy class of $g$ in $G$ and let $\phi$ denote its 
image in $F$. As $g$ is a semi-simple element in $G$, so is $g_2$ and hence 
$\phi$ is a semi-simple conjugacy class in $F$. 

We let $W = {X^g}/{Z_{g_1}}$ and let $k^g: W \to Y^{g_2}$ denote the
map induced by $X \to Y$. Let $\upsilon^\psi: {I^{\psi}_X}/H \to I^{\phi}_Y$
and $\upsilon: {I_X}/H \to I_Y$ denote the maps on the inertia spaces.
Let $j^{\psi}_X : X_{\psi} \rightarrow X$ and $j^{\phi}_Y: Y_{\phi} \to Y$
be the inclusion maps. Let $s^{\psi}_X: I^\psi_X \to I_X$ and
$s^{\phi}_Y : I^{\phi}_Y \to I_Y$ be the inclusion maps.
Note that $W$ and $Y^{g_2}$ have natural actions of 
$Z_{g_2}$ by \lemref{lem:Basic}.
Similarly, $F$ acts on ${I^{\psi}_X}/H$ and ${X_{\psi}}/H$.
Let $u^g: {X_{\psi}}/H \to Y_{\phi}$ denote the 
canonical map induced by quotient map $X \to Y$. This is clearly 
$F$-equivariant. 

\begin{comment}
We have the two commutative
diagrams
\begin{equation}\label{diag:prodstructureofmaps*}
\xymatrix@C1pc{
I^{\psi}_X \ar[rr]^-{\mu^{\psi}} \ar[d]_{\pi^{\psi}} && X  \ar[d] & & & &
X^g \ar[d] & G \times X^g \ar[r] \ar[d] \ar[l] & G \times^{Z_g} X^g \ar[d] \\
I^{\psi}_X/H \ar[r]^-{\upsilon^{\psi}} \ar[d] & I^{\phi}_Y 
\ar[r]^-{\mu^{\phi}} \ar[d] & Y \ar[d]  & & & & W & F \times W \ar[r] \ar[l] & 
F \times^{Z_{g_2}} W. \\
I^{\psi}_X /G \ar[r]^-{\delta^{\psi}} & {I^{\phi}_Y}/F \ar[r]^-{\gamma^{\phi}} & Z.
& & & & & &}
\end{equation}

Note that in the diagram on the right, all arrows are just the   
quotient maps. We shall use this diagram later 
(see, for instance, \lemref{lem:Moritapartialfunctoriality}).
\end{comment}

We have the commutative diagram
\begin{equation}\label{diag:prodstructureofmaps*}
\xymatrix@C1pc{
I^{\psi}_X \ar[rr]^-{\mu^{\psi}} \ar[d]_{\pi^{\psi}} && X  \ar[d] \\
I^{\psi}_X/H \ar[r]^-{\upsilon^{\psi}} \ar[d] & I^{\phi}_Y 
\ar[r]^-{\mu^{\phi}} \ar[d] & Y \ar[d] \\
I^{\psi}_X /G \ar[r]^-{\delta^{\psi}} & {I^{\phi}_Y}/F \ar[r]^-{\gamma^{\phi}} &
Z.}
\end{equation}

\begin{lem}\label{lem:Moritaforpartialquotient}
With respect to the above set up, we have the following.
\begin{enumerate}
\item
All horizontal arrows 
in ~\eqref{diag:prodstructureofmaps*} are finite.
All maps are $G$-equivariant and all maps in the bottom squares are
$F$-equivariant. 
\item
The map $u^g: {X_{\psi}}/H \to Y_{\phi}$ is finite. 
\item
There is an $F$-equivariant isomorphism $p^W_{g_2}: F \times^{Z_{g_2}}W 
\xrightarrow{\simeq} {I^{\psi}_X}/H$.
\item
There is a canonical isomorphism of quotient stacks
$[W/{Z_{g_2}}] \simeq [{({I^{\psi}_X}/H)}/F]$.
\item
The map $k^g: W \to Y^{g_2}$ is $Z_{g_2}$-equivariant and finite.
\end{enumerate}
\end{lem}
\begin{proof}
The $G$-equivariance and $F$-equivariance of the maps in the left diagram in  
~\eqref{diag:prodstructureofmaps*} are clear.
Except possibly ${\upsilon^{\psi}}$ and $\delta^{\psi}$, all the horizontal 
arrows in ~\eqref{diag:prodstructureofmaps*} are finite by 
Lemmas~\ref{lem:torsor} and ~\ref{lem:properpushforward}. In particular,   
$\mu^{\phi} \circ \upsilon^{\psi}$ and $\gamma^{\phi} \circ \delta^{\psi}$ are
finite. It follows that $\upsilon^{\psi}$ and $\delta^{\psi}$ are also finite.
This proves (1).

We next note that there are maps
${I^{\psi}}/H \xrightarrow{\mu^{\psi}/H} {X_{\psi}}/H \xrightarrow{u^g} Y_{\psi}$,
such that $u^g \circ ({\mu^{\psi}/H}) = \mu^\phi \circ \upsilon^{\psi}$.
The first map is finite and surjective by Lemmas~\ref{lem:torsor}
and ~\ref{lem:properpushforward}. We just showed above that
the map $u^g \circ ({\mu^{\psi}/H})$ is finite. It follows that
$u^g$ is finite (see the proof of 
\lemref{lem:properpushforward}). This proves (2).
We now consider the commutative diagrams
\begin{equation}\label{eqn:Mor-prod-0}
\xymatrix@C1pc{
X^g \ar@{^{(}->}[r] \ar[d] & G \times^{Z_g} X^g \ar@{->>}[r] \ar[d] & 
{G/{Z_g}} \ar[d] 
& & G \times^{Z_g} X^g \ar[r]^-{p_g} \ar[d] & I^{\psi}_X \ar[d] \\
W \ar@{^{(}->}[r]  \ar[d]_{k^g}  & F \times^{Z_{g_2}} W \ar@{->>}[r] 
\ar[d]^{\ov{id_F \times k^g}} 
& {F/{Z_{g_2}}} 
\ar[d] & & F \times^{Z_{g_2}} W \ar[r]^-{p^W_{g_2}} \ar[d]_{\ov{id_F \times k^g}} & 
{I^{\psi}_X}/H \ar[d]^{\upsilon^{\psi}} \\
Y^{g_2} \ar@{^{(}->}[r] & F \times^{Z_{g_2}} Y^{g_2} 
\ar@{->>}[r] & {F/{Z_{g_2}}} & &
F \times^{Z_{g_2}} Y^{g_2} \ar[r]^-{p^Y_{g_2}} & I^{\phi}_Y.} 
\end{equation}

We first explain the maps in the diagram on the left.
The horizontal arrows on the left are all of the type $z \mapsto (e, z)$,
where $e$ is the identity element of $G$ or $F$. The horizontal arrows on the
right are the ones induced on the quotients by the projection maps to the 
first factors. It is then clear that the left column is the closed fiber
of these maps over the identity cosets. In particular, horizontal arrows on the 
left are all closed immersions. 

We now explain the maps in the diagram on the right.
The maps $p_g$ and $p^Y_{g_2}$ are given in \lemref{lem:torsor}. 
Since $p_g$ is $G$-equivariant with respect to the action
$(g'_1, g'_2)(h_1, h_2, x) = (g'_1h_1, g'_2h_2, x)$, it follows that
this descends to an $F$-equivariant map ${(G \times^{Z_g} X^g)}/H \to 
{I^{\psi}_X}/H$.
On the other hand, using the way various actions are defined, we
get
\begin{equation}\label{eqn:Mor-prod-1}
\begin{array}{lllllll}
{(G \times^{Z_g} X^g)}/H & \xrightarrow{\simeq} &
\frac{(H \times F) \times^{(Z_{g_1} \times Z_{g_2})} X^g}{H} &   
\xrightarrow{\simeq} &
\frac{(H \times F) \times X^{g}}{H \times Z_{g_1} \times Z_{g_2}} \\
& \xrightarrow{\simeq} &
\frac{F \times X^g}{Z_{g_1} \times Z_{g_2}} 
& \xrightarrow{\simeq} & 
\frac{F \times ({X^g}/{Z_{g_1}})}{Z_{g_2}} \\ 
& \xrightarrow{\simeq} & F \times^{Z_{g_2}}W. & &
\end{array}
\end{equation}

The second isomorphism holds because $H \simeq (H \times 1_F)$ and 
$(Z_{g_1} \times Z_{g_2})$-actions on $(H \times F) \times X^{g}$ commute.
The third isomorphism holds because the $(H \times 1_F)$-action on 
$(H \times F) \times X^{g}$ is free and this action on $X^g$-factor is 
trivial. The fourth isomorphism holds
because $Z_{g_1}$ acts trivially on $F$. One checks easily that all the maps
in ~\eqref{eqn:Mor-prod-1} are $F$-equivariant.
It follows that $p_g$ descends to an $F$-equivariant isomorphism 
$p^W_{g_2}$ and the two squares on the right side of ~\eqref{eqn:Mor-prod-0}
commute. 
This proves (3) and (4) follows from this because of the isomorphism of
stacks $[W/{Z_{g_2}}] \simeq [{(F \times^{Z_{g_2}}W)}/F]$ by 
\lemref{lem:Morita-stack*}.

We now prove (5). Since the map $\upsilon^{\psi}$ is finite 
(see ~\eqref{diag:prodstructureofmaps*}), it follows that 
the map ${\ov{id_F \times k^g}}: F \times^{Z_{g_2}}W \to F \times^{Z_{g_2}}Y^{g_2}$ 
is finite $F$-equivariant. Since $W \to F \times^{Z_{g_2}}W$ and 
$Y^{g_2} \to F \times^{Z_{g_2}}Y^{g_2}$ are 
closed immersions, it follows by looking at the diagram on the left in 
~\eqref{eqn:Mor-prod-0} that $k^g$ is an $Z_{g_2}$-equivariant finite morphism.
\end{proof}

\subsection{Equivariant $K$-theory for action of $H \times F$}
\label{sec:K-prod*}
We shall let 
\begin{equation}\label{eqn:Mor-prod-2}
\mu_g: G(Z_g, X^g) \xrightarrow{\simeq} G(G,I^{\psi}_X), \
\mu'_g: G(Z_{g_2}, W) \xrightarrow{\simeq} G(F,I^{\psi}_X/H) \
\mbox{and} 
\end{equation}
\[
\mu_{g_2}: G(Z_{g_2},Y^{g_2}) \xrightarrow{\simeq} G(F,I^{\phi}_Y)
\]
denote the Morita isomorphisms on the $K$-theory induced by 
\lemref{lem:Moritaforpartialquotient}.
Continuing with the above set up, we now prove the following statements.

\begin{lem}\label{lem:Moritapartialfunctoriality}
There is a commutative diagram
\begin{equation}\label{eqn:Mor-prod-3}
\xymatrix@C1pc{
G_*(Z_g, X^g) \ar[d]_{\mu_g} \ar[r]^-{\inv^{Z_{g_1}}_{X^g}} &  
G_*(Z_{g_2}, W) \ar[r]^-{k^g_*} \ar[d]^{\mu'_g} & G_*(Z_{g_2}, Y^{g_2})
\ar[d]^{\mu_{g_2}} \\
G_*(G,I^{\psi}_X)  \ar[r]^-{\inv^H_{I^{\psi}_X}} &
G_*(F, {I^{\psi}_X}/H) \ar[r]^-{\upsilon^{\psi}_*} & G_*(F, I^{\phi}_Y).}
\end{equation}
\end{lem}
\begin{proof}
The commutativity of the right square follows by applying
\lemref{lem:Morita-stack*} to the $Z_{g_2}$-equivariant morphism
$k^g: W \to Y^{g_2}$. The left square commutes because,
by \thmref{thm:invariants} and ~\eqref{eqn:mu-g-*}, it is induced by the 
commutative diagram of quotient stacks
\begin{equation}\label{eqn:Mor-prod-3-1}
\xymatrix@C1pc{
[{X^g}/{Z_g}] \ar[r] \ar[d]_{\iota^{GZ_g}_{X^g}} & [W/{Z_{g_2}}] 
\ar[d]^{\iota^{FZ_{g_2}}_W} \\
[{I^{\psi}_X}/G] \ar[r] & [{(I^{\psi}/H)}/F],}
\end{equation}
where the horizontal arrows are proper and the vertical arrows are isomorphisms.
\end{proof}

\begin{comment}
The proof of the commutativity of the left square
is essentially a repetition of the proof of \lemref{lem:moritaandrrproper}.
We give a sketch. We consider the diagram
\begin{equation}\label{eqn:Mor-prod-4}
\xymatrix@C1pc{
G_*(Z_g, X^g) \ar[r]^-{pr_{X^g}^*} \ar[d]_{\inv^{Z_{g_1}}_{X^g}} & 
G_*(G \times Z_g, G \times X^g ) \ar[r]^-{\inv^{1 \times Z_g}_{G \times X^g}} 
\ar[d]^{\inv^{H \times Z_{g_1}}_{G \times X^g}} & 
G_*(G,G \times^{Z_g}X^g) \ar[d]^{\inv^H_{I^{\psi}_X}} \\ 
G_*(Z_{g_2},W) \ar[r]^-{pr^*_{W}} & 
G_*(F \times Z_{g_2}, F \times W) 
\ar[r]^-{\inv^{1 \times Z_{g_2}}_{F \times W}} & 
G_*(F, F \times^{Z_{g_2}} W).}
\end{equation}

To prove that the left square commutes, 
we can replace the horizontal arrows by their
inverses, which are just the invariant functors 
$\inv^{G \times 1}_{{G \times X^g}}$ and $\inv^{F \times 1}_{F \times W}$.
Now, the two square commute by the functoriality of the 
`invariants'-functors with respect to various subgroups, given in 
\thmref{thm:invariants} and from the commutativity of the  
diagram the quotient maps given in the right corner of 
~\eqref{diag:prodstructureofmaps*}.
\end{proof}
\end{comment}

\begin{lem}\label{lem:Moritapartialcompatability}
There is a commutative diagram
\begin{equation}
\xymatrix@C1pc{
G_*(Z_g,X^g)_{\fm_g} \ar[r]^-{\mu^{\psi}_g} \ar[d]_{\inv^{Z_{g_1}}_{X^g}} & 
G_*(G,I^{\psi}_X)_{\fm_{\psi}} \ar[r]^-{\mu^{\psi}_{*}} \ar[d]^{\inv^H_{I^{\psi}_X}} & 
G_*(G, X_{\psi})_{\fm_{\psi}} \ar[d]^{\inv^{H}_{X_{\psi}}} \\
G_*(Z_{g_2},W)_{\fm_{g_2}} \ar[r]^-{\mu'^{\phi}_g} 
\ar[d]_{k^g_*} & G_*(F, I^{\psi}_{X}/H)_{\fm_{\phi}} \ar[d]^{\upsilon^{\psi}_*} 
\ar[r]^-{\mu'^{\phi}_*} & G_*(F,X_{\psi}/H)_{\fm_{\phi}} \ar[d]^{u^g_*} \\
G_*(Z_{g_2}, Y^{g_2})_{\fm_{g_2}} \ar[r]^-{\mu^{\phi}_{g_2}} & 
G_*(F ,I^{\phi}_Y)_{\fm_{\phi}} \ar[r]^-{\mu^{\phi}_*} & 
G_*(F,Y_{\phi})_{\fm_{\phi}}.} 
\end{equation}
\end{lem}
\begin{proof}
Using the fact that all maps on the top squares are $R(G)$-linear
and all maps on the bottom squares are $R(F)$-linear, the
commutativity of the two squares on the left follows directly from
\lemref{lem:Moritapartialfunctoriality}. The
commutativity of the two squares on the right follows from
\thmref{thm:invariants}.
\end{proof}

\begin{lem}\label{lem:Gen-qt}
Let $G$ be a linear algebraic group acting properly on an algebraic space $X$
and let $p: X \to X/G$ be the quotient. Then the map
$\inv^G_X: G_*(G,X)_{\fm_1} \to G_*(X/G)$ is an isomorphism.
\end{lem}
\begin{proof}
If $X/G$ is quasi-projective, then so is $X$ and the result follows from
\corref{cor:G-surj}. Note that this corollary does not rely on any result
of this text before \S~\ref{sec:RRT}.

In general, it follows from 
\cite[Proposition~2.4]{TO3} that there is a dense open 
subspace $U \subset X/G$ such that $U$ and $p^{-1}(U)$ are schemes of finite type 
over $\C$. Since the map $p^{-1}(U) \to U$ is affine, we can find an
affine open subscheme $W' \subset X/G$ so that $W = p^{-1}(W') \subset p^{-1}(U)$ 
is $G$-invariant affine and $W' = W/G$. It follows therefore from the
quasi-projective case that the map
$\inv^G_W: G_*(G,W)_{\fm_1} \to G_*(W/G)$ is an isomorphism. 

We now let $Z = X \setminus W$ and consider the commutative diagram
\begin{equation}\label{eqn:Gen-qt-0}
\xymatrix@C1pc{
G_{i+1}(G,W)_{\fm_1} \ar[r] \ar[d] & G_i(G,Z)_{\fm_1} \ar[r] \ar[d] &
G_i(G,X)_{\fm_1} \ar[r] \ar[d] & G_i(G,W)_{\fm_1} \ar[r] \ar[d] &
G_{i-1}(G,Z)_{\fm_1} \ar[d] \\
G_{i+1}(W/G) \ar[r] & G_i(Z/G) \ar[r] &
G_i(X/G) \ar[r] & G_i(W/G) \ar[r] & G_{i-1}(Z/G)}
\end{equation}
of localization exact sequences. We showed above that the first and the fourth 
vertical arrows from the left are isomorphisms. 
The second and the fifth vertical arrows are
isomorphisms by the Noetherian induction on $X$. It follows that the middle
vertical arrow in an isomorphism. 
\end{proof}

\begin{lem}\label{lem:partialinvtwisst}
There is a commutative diagram
\begin{equation}\label{eqn:Mor-prod-5}
\xymatrix@C1pc{
G_*(Z_g, X^g)_{\fm_g} \ar[d]_{t_{g^{-1}}} \ar[r]^-{\inv^{Z_{g_1}}_{X^g}} &
G_*(Z_{g_2}, W)_{\fm_{g_2}} \ar[d]^{t_{g^{-1}_2}} \\
G_*(Z_g, X^g)_{\fm_1} \ar[r]^-{\inv^{Z_{g_1}}_{X^g}} & 
G_*(Z_{g_2}, W)_{\fm_1}.} 
\end{equation}
\end{lem}
\begin{proof}
We consider the diagram
\begin{equation}\label{eqn:Mor-prod-6}
\xymatrix@C1pc{
G_*(Z_g, X^g)_{\fm_g} \ar[d]_{t_{g^{-1}}} \ar[r]^-{\inv^{Z_{g_1}}_{X^g}} &
G_*(Z_{g_2}, W)_{\fm_{g_2}} \ar[d]^{t_{g^{-1}_2}} 
\ar[r]^-{\inv^{Z_{g_2}}_W} & G_*({X^g}/{Z_g}) \ar[d]^{Id} \\
G_*(Z_g, X^g)_{\fm_1} \ar[r]^-{\inv^{Z_{g_1}}_{X^g}} & 
G_*(Z_{g_2}, W)_{\fm_1} \ar[r]^-{\inv^{Z_{g_2}}_W} & G_*({X^g}/{Z_g}).} 
\end{equation}
The right square and the outer rectangle commute by
\propref{prop:invtwisst} and the map ${\inv^{Z_{g_2}}_W}$ is an isomorphism
by \lemref{lem:Gen-qt}. It follows that the left square
commutes.
\end{proof}

\subsection{Atiyah-Segal for partial quotient map}\label{sec:PartialRR}
\begin{comment}
We begin with the following special case.
\begin{lem}\label{lem:aiprod}
There is a commutative diagram
\begin{equation}\label{diag:aiprod}
\xymatrix@C1pc{
G_i(G,X)_{\fm_1} \ar[r]^-{\tau^G_X} \ar[d]_{\inv^H_X} & 
\CH_*^G(X,i) \ar[d]^{\ov{\inv^H_X}} \\
G_i(F,Y)_{\fm_1} \ar[r]^-{\tau^F_Y} & \CH_*^F(Y,i).}
\end{equation}
\end{lem}
\begin{proof}
We consider the diagram
\begin{equation}\label{eqn:aiprod-0}
\xymatrix@C1pc{
G_i(G,X)_{\fm_1} \ar[d]_{\tau^G_X} \ar[r]^-{\inv^H_X} &
G_(F,Y)_{\fm_1} \ar[r]^-{\inv^F_Y} \ar[d]^{\tau^F_Y} &
G_i(Z) \ar[d]^-{\tau_Z} \\
\CH_*^G(X,i) \ar[r]^-{\ov{\inv^H_X}} & \CH_*^F(Y,i) \ar[r]^{\ov{\inv^F_Y}} &
\CH_*(Z,i).}
\end{equation}
It follows from Theorem~\ref{thm:invariants} that $\inv_Y^F \circ \inv^H_X 
= \inv^G_X$ and it follows from \thmref{thm:pcpu} that $\ov{\inv}
^G_X = \ov{\inv}^F_Y \circ \ov{\inv}^H_X$. 
The the outer rectangle and the right square both commute by
\lemref{lem:augmentation ideal}. Since the map $\ov{\inv^F_Y}$
is an isomorphism (see the proof of \thmref{thm:pcpu}), it follows that 
the left square commutes.
\end{proof}
\end{comment}

We continue with our set up of proper action of an algebraic group 
$G = H \times F$
on an algebraic space $X$. We shall now prove the covariance of the 
Atiyah-Segal map with respect to the partial quotient map $X \to X/H = Y$.
Let $i \ge 0$ be a fixed integer.
Recall from \thmref{thm:invariants} that the map 
$\inv^H_{X_{\psi}}: G_i(G,X_{\psi}) \to 
G_i(F,X_{\psi}/H)$ is $R(F)$-linear.
On the localizations, we get the maps of $R(F)$-modules 
\begin{equation}\label{eqn:G-F-loc}
G_i(G,X_{\psi})_{\fm_\psi} \inj \oplus_{\gamma \cap F = \phi} 
G_i(G, X_{\psi})_{\fm_{\gamma}} \xrightarrow{\simeq} G_i(G,X_{\psi})_{\fm_\phi} 
\xrightarrow{\inv^H_{X_{\psi}}} G_i(F,X_{\psi}/H)_{\fm_\phi},
\end{equation}
where the first map is the canonical inclusion as a 
direct summand and the second map is an isomorphism by 
\propref{prop:decom-surj}.
The map $\inv^H_{X_{\psi}}: G_i(G,X_{\psi}) \to 
G_i(F,X_{\psi}/H)$ therefore restricts to an $R(F)$-linear
map $ G_i(G,X_{\psi})_{\fm_\psi} \xrightarrow{\inv^H_{X_{\psi}}}
G_i(F,X_{\psi}/H)_{\fm_\phi}$.
Similarly, we have an $R(Z_{g_2})$-linear map
$\inv^{Z_{g_1}}_{X^g}: G_i(Z_g,X^g)_{\fm_g} \to  G_i(Z_{g_2},W)_{\fm_{g_2}}$.

\begin{prop}\label{prop:conjprod}
There is a commutative diagram
\begin{equation}\label{diag:conjprod}
\xymatrix@C1pc{
G_i(G,X_{\psi})_{\fm_{\psi}} \ar[r]^-{\vartheta^{\psi}_X} 
\ar[d]_{u^g_* \circ \inv^H_{X_{\psi}}}  &  
G_i(G,I^{\psi}_X)_{\fm_1} \ar[d]^{\upsilon^{\psi}_* \circ \inv^H_{I^{\psi}_X}} \\
G_i(G,Y_{\phi})_{\fm_{\phi}} \ar[r]^{\vartheta^{\phi}_Y} & G_i(F, I_Y^{\phi})_{\fm_1}.}
\end{equation}
\end{prop}
\begin{proof}
We consider the diagram
\begin{equation}\label{diag:conjprod-0}
\xymatrix@C.7pc{
G_i(G,X_{\psi})_{\fm_{\psi}} \ar[r]^-{\omega_g^{-1}} \ar[d]_{\inv^{H}_{X_{\psi}}} & 
G_i(Z_g,X^g)_{\fm_g} \ar[r]^{t_{g^{-1}}}  \ar[d]^{\inv^{Z_{g_1}}_{X^g}} & 
G_i(Z_g,X^g)_{\fm_1} \ar[r]^{\mu^1_g} \ar[d]^{\inv^{Z_{g_1}}_{X^g}} &
G_i(G,I^{\psi}_X)_{\fm_{{1}}} \ar[d]^{\inv^H_{I^{\psi}_X}} \\
G_i(F,X_{\psi}/H)_{\fm_{\phi}} \ar[d]_{u^g_*} & G_i(Z_{g_2},W)_{\fm_{g_2}} 
\ar[d]^{k^g_*} \ar[r]^-{t_{{g^{-1}_2}}} &
G_i(Z_{g_2},W)_{\fm_1} \ar[d]^{k^g_*} \ar[r]^-{\mu'^1_g} &
G_i(F,I^{\psi}_X/H)_{\fm_1} \ar[d]^{\upsilon^{\psi}_*} \\
G_i(F,Y_{\psi})_{\fm_{\phi}} \ar[r]^-{\omega_{g_2}^{-1}} & 
G_i(Z_{g_2}, Y^{g_2})_{\fm_{g_2}} \ar[r]^-{t_{g^{-1}_2}}  &
G_i(Z_{g_2},Y^{g_2})_{\fm_1} \ar[r]^-{\mu^1_{g_2}} & 
G_i(F,I^{\phi}_Y)_{\fm_1}.}
\end{equation}

In this diagram, the top and the bottom squares on the right
commute by \lemref{lem:Moritapartialfunctoriality}.
The middle square on the top commutes by
\lemref{lem:partialinvtwisst} and the one on the bottom commutes by 
\corref{cor:proper-T-commute}. To prove that the big rectangle on the 
left commutes, we can replace the horizontal arrows by their inverses.
In this case, the rectangle commutes by \lemref{lem:Moritapartialcompatability}.
We have thus shown that ~\eqref{diag:conjprod} commutes.
\end{proof}

The final consequence of all the above results of this section is
the following.

\begin{thm}\label{thm:GRR-PROD}
There is a commutative square
\begin{equation}\label{eqn:GRR-PROD-0}
\xymatrix@C1pc{
G_i(G,X) \ar[r]^-{\vartheta^G_X} \ar[d]_{\inv^H_X } & 
G_i(G, I_X)_{\fm_1} \ar[d]^{{\upsilon}_* \circ \inv^H_{I_X}} \\
G_i(F,Y) \ar[r]^-{\vartheta^F_Y} & G_i(F,I_Y)_{\fm_1}}
\end{equation}
in which the horizontal arrows are isomorphisms.
\end{thm}
\begin{proof}
We first observe that the $G$-equivariant decomposition
$I_X = \amalg_{\psi \in \Sigma^G_X} I^{\psi}_X$ yields an $F$-equivariant
decomposition ${I_X}/H = \amalg_{\psi \in \Sigma^G_X} {I^{\psi}_X}/H$. 
There is a similar decomposition of $I_Y$ in terms of $I^{\phi}_Y$.
We now look at the diagram
\begin{equation}\label{eqn:GRR-PROD-1}
\xymatrix@C1pc{
{\oplus_{\psi \in S_G} G_i(G,X_{\psi})_{\fm_{\psi}}} 
\ar[rr]^-{\oplus \vartheta^{\psi}_X} \ar[dr]^-{\oplus j^{\psi}_{X  *}}
\ar[dd]_{\oplus u^g_* \circ \inv^H_{X_{\psi}}} & & 
{\oplus_{\psi \in S_G} G_i(G,I^{\psi}_X)_{\fm_1}} 
\ar@/.8cm/[dd]^{\oplus \upsilon^{\psi}_* \circ \inv^H_{I^{\psi}_X}}
\ar[dr]^{{\oplus s^{\psi}_X}_*} & \\
& G_i(G,X) \ar[rr]^>>>>>>>>>>>>{\vartheta^G_X} \ar@/.5cm/[dd]^{\inv^H_X } & & 
G_i(G,I_X)_{\fm_1} \ar[dd]^{\upsilon_* \circ \inv^H_{I_X}} \\
{\oplus_{\phi \in S_F} G_i(F, Y_{\phi})_{\fm_{\phi}}} \ar[rr]^-{\oplus \vartheta^{\phi}_Y}
\ar[dr]_{\oplus j^{\phi}_{Y  *}} & & 
{\oplus_{\phi \in S_F} G_i(F,I^{\phi}_Y)_{\fm_1}} \ar[dr]^{{\oplus s^{\phi}_Y}_*}
& \\
& G_i(F,Y) \ar[rr]^-{\vartheta^F_Y} & & G_i(F,I_Y)_{\fm_1}.}
\end{equation}

The back face of this cube commutes by \propref{prop:conjprod}.
The top and the bottom faces commute by the definition of 
$\vartheta^G_X$ and $\vartheta^F_Y$ (see Definition~\ref{defn:RR-mapfull}).
The left and the right faces commute by the functoriality of the
proper push-forward maps in equivariant $K$-theory (of coherent sheaves).
The maps $\oplus j^{\psi}_{X  *}$ and $\oplus j^{\phi}_{Y  *}$ are
isomorphisms by \thmref{thm:closedimmiso} and \propref{prop:direcsumdecomp}.
The maps ${\oplus s^{\psi}_X}_*$ and ${\oplus s^\phi_Y}_*$ are
isomorphisms by the above decompositions of the inertia spaces.
It follows that the front face commutes.
Since each of $\vartheta^{\psi}_X$ and $\vartheta^{\phi}_Y$ is an isomorphism 
(see Definition~\ref{defn:RR-mapconjclass-AS}), we
conclude that $\vartheta^G_X$ and $\vartheta^F_Y$ are isomorphisms too.
The proof of the theorem is complete.
\end{proof}

\section{Atiyah-Segal correspondence for quotient stacks}\label{sec:RR-DM}
In this section, we shall prove our final results on the
Atiyah-Segal correspondence for separated quotient stacks.
In order to define the Atiyah-Segal map for such stacks,
we need to prove the following result about these maps for the
group action.

Let $H$ and $F$ be two linear algebraic groups.
Let $H$ and $F$ act properly on algebraic spaces $X$ and $Y$, respectively. 
We let $\sX = [X/H]$ and $\sY = [Y/F]$ denote the stack quotients.
Let $f: \sX \to \sY$ be a proper morphism.
If we let $G = H \times F$, then this gives rise to 
the algebraic space $Z$ with $G$-action and the Cartesian diagram
~\eqref{eqn:Fin-coh-0-new}, which we reproduce here for
reader's convenience. 

\begin{equation}\label{eqn:Fin-coh-0-new*}
\xymatrix@C1pc{
Z \ar[r]^-{s} \ar[d]_{t} & \sX' \ar[r]^-{f'} \ar[d]^{p'} & Y \ar[d]^{p} \\
X \ar[r]^-{q} & \sX \ar[r]^-{f} & \sY.}
\end{equation}

Let us now assume that $\sX$ is a separated quotient stack with the 
coarse moduli space $M$. Let $\sX = [X/H] = [Y/F]$ be two
presentations of $\sX$. Letting $Z = X \times_{\sX} Y$,
we get a Cartesian diagram like ~\eqref{eqn:Fin-coh-0-new*}, where
$f$ is the identity map. In particular, \lemref{lem:GRR-QDM-stacks}
applies. 
%Let $\sB \sG, \ \sB \sH$ and $\sB \sF$ denote the classifying stacks
%and $G, H$ and $F$, respectively so that $\sB \sG = \sB \sH \times \sB \sF$.  
Let $I_{\sX}$ denote the inertia stack of $\sX$. Let $\epsilon: {I_Z}/F \to I_Y$
denote the canonical map induced by $s$.
We fix an integer $i \ge 0$.

\begin{lem}\label{lem:independenceofstackmap}
There exists a commutative diagram
\begin{equation}\label{eqn:Ind-RR-st-0}
\xymatrix@C1pc{
G_i(\sX) \ar[r]^-{q^*} & G_i(H,X) \ar[r]^-{\vartheta^H_X} & 
G_i(H,I_X)_{\fm_{1_H}} \ar[r]^-{{q^*}^{-1} \circ v_X} & 
G_i(I_{\sX})_{\fm_{I_{\sX}}} \\
G_i(\sX) \ar[r]^-{(q \circ t)^*} \ar[u]^{Id} \ar[d]_{Id} 
& G_i(G,Z) \ar[r]^-{\vartheta^G_Z} \ar[d]^{\inv^H_Z}
\ar[u]_{\inv^F_Z}  
& G_i(G, I_Z)_{\fm_{1_G}} \ar[r]^{(t^* \circ q^*)^{-1} \circ v_Z}  
\ar[u]_{\upsilon_* \circ \inv^F_{I_Z}} \ar[d]^{\epsilon_* \circ \inv^H_{I_Z}} & 
G_i(I_{\sX})_{\fm_{I_{\sX}}} \ar[u]_{Id} \ar[d]^{Id} \\
G_i(\sX) \ar[r]^-{p^*}  & G_i(F,Y)  \ar[r]^-{\vartheta^F_Y}
& G_i(F, I_Y)_{\fm_{1_F}} \ar[r]^-{{p^*}^{-1} \circ v_Y} & 
G_i(I_{\sX})_{\fm_{I_{\sX}}}.}
\end{equation}
\end{lem}
\begin{proof}
We show that all the top squares commute as the argument for
the commutativity of the bottom squares is identical. 
We first observe that $F$ acts freely on $I_Z$ and the map 
$\upsilon: {I_Z}/F \to I_X$ is an isomorphism by \lemref{lem:$I_X$torsor}.
In particular, the map $\upsilon_* \circ \inv^F_{I_Z}$ is same as the
map $\inv^F_{I_Z}: G_i(G, I_Z) \to G_i(H, I_X)$. 
Since the maps $Z \to X$ and $I_Z \to I_X$ are $F$-torsors,
the induced maps $\inv^F_Z: G_i(G,Z) \to G_i(H,X)$ and $\inv^F_{I_Z}:
G_i(G,I_Z) \to G_i(H,I_X)$ are isomorphisms which are inverses to the 
pull-back maps $t^*:G_i(H,X) \to G_i(G,Z)$ and $\pi^*: G_i(H,I_X) \to
G_i(G,I_Z)$, respectively.
Here, $\pi:I_Z \to {I_Z}/F \xrightarrow{\simeq} I_X$ is the
quotient map (see ~\eqref{diag:prodstructureofmaps*}).
The commutativity of the top squares in ~\eqref{eqn:Ind-RR-st-0}
is therefore equivalent to the commutativity of the diagram
\begin{equation}\label{eqn:Ind-RR-st-1}
\xymatrix@C1pc{
G_i(\sX) \ar[r]^-{q^*} \ar[d] _{Id} & G_i(H,X) \ar[r]^-{\vartheta^H_X} 
\ar[d]^{t^*} & 
G_i(H,I_X)_{\fm_{1_H}} \ar[r]^-{{q^*}^{-1} \circ v_X} \ar[d]^{\pi^*} & 
G_i(I_{\sX})_{\fm_{I_{\sX}}} \ar[d]^{Id} \\
G_i(\sX) \ar[r]^-{(q \circ t)^*}  & 
G_i(G,Z) \ar[r]^-{\vartheta^G_Z}   
& G_i(G, I_Z)_{\fm_{1_G}} \ar[r]^{(t^* \circ q^*)^{-1} \circ v_Z}  & 
G_i(I_{\sX})_{\fm_{I_{\sX}}}.}
\end{equation}

Now, the left square commutes by definition and the middle square commutes 
by \thmref{thm:GRR-PROD}. The right square commutes again by definition
since $t^* \circ v_X = v_Z \circ t^*:G_i(H, I_X)_{\fm_{1_H}} \to 
G_i(G, I_Z)_{\fm_{I_{\sX}}}$. This proves the lemma.
\end{proof}

\begin{defn}\label{defn:RR-QDM}
Let $\sX=[X/G]$ be a separated quotient stack.
We define the Atiyah-Segal map $\vartheta_{\sX}$ to be the composite
\begin{equation}\label{eqn:RR-QDM-0}
\vartheta_{\sX} = v_X \circ \vartheta^G_X \circ p_X^*: 
G_i(\sX) \to G_i(I_{\sX})_{\fm_{I_{\sX}}}. 
\end{equation}
Here, $p_X: X \to \sX$ is the projection map.
It follows from \lemref{lem:independenceofstackmap} that $\vartheta_{\sX}$ is 
independent of the choice of the presentation of the stack.
Furthermore, it follows from \thmref{thm:GRR-PROD} that $\vartheta_{\sX}$
is an isomorphism. 
\end{defn}
Our main result on the Atiyah-Segal correspondence for separated quotient
stacks is the following.

\begin{thm}\label{thm:ASCS}
Let $f: \sX \to \sY$ be a proper morphism of separated quotient 
stacks of finite type over $\C$. Let $i \ge 0$ be an integer.
Then there is a commutative diagram
\begin{equation}\label{eqn:ASCS-0}
\xymatrix@C1pc{
G_i(\sX) \ar[r]^-{\vartheta_{\sX}} \ar[d]_{f_*} & G_i(I_{\sX})_{\fm_{I_{\sX}}} 
\ar[d]^{f^I_*} \\
G_i(\sY) \ar[r]^-{\vartheta_{\sY}} & G_i(I_{\sY})_{\fm_{I_{\sY}}}}
\end{equation}
such that the horizontal arrows are isomorphisms.
\end{thm}
\begin{proof}
We have already seen in Definition~\eqref{defn:RR-QDM} that the horizontal
arrows are isomorphisms. We only need to show the commutativity.
We let $\sX = [X/H]$ and $\sY = [Y/F]$. Let $Z$ be as in
~\eqref{eqn:Fin-coh-0-new*}. It follows from \lemref{lem:GRR-QDM-stacks}
that $G = H \times F$ acts on $Z$ such that $t$ is $G$-equivariant and an
$F$-torsor. 
By \lemref{lem:GRR-QDM-stacks}, we have proper maps 
$\sX = [Z/G] \xrightarrow{p} \sW = [W/F] 
\xrightarrow{q} \sY = [Y/F]$, where $W = Z/H$.

We consider the diagram
\begin{equation}\label{eqn:ASCS-1}
\xymatrix@C1pc{
G_i(\sX) \ar[r]^-{p_Z^*} \ar[d]_{p_*} & 
G_i(G,Z) \ar[r]^-{\vartheta^G_Z} \ar[d]^{\inv^H_Z}  &
G_i(G, I_Z)_{\fm_{1_G}} \ar[d]^{\upsilon_* \circ \inv^H_{I_Z}} \ar[r]^{v_Z} &
G_i(I_Z)_{\fm_{I_{\sZ}}} \ar[d]^{\upsilon_* \circ \inv^H_{I_Z}} \\
G_i(\sW) \ar[r]^-{p^*_W} \ar[d]_{q_*} & 
G_i(F, W) \ar[r]^-{\vartheta^F_W} \ar[d]^{q_*} & G_i(F,I_W)_{\fm_{1_F}} 
\ar[d]^{q^I_*} \ar[r]^-{v_W} & G_i(I_{\sW})_{\fm_{I_{\sW}}} \ar[d]^{q^I_*} \\
G_i(\sY) \ar[r]^{p^*_Y} & G_i(F,Y) \ar[r]^-{\vartheta^F_Y} &
G_i(I_Y)_{\fm_{1_F}} \ar[r]^-{v_Y} & G_i(I_{\sY})_{\fm_{I_{\sY}}}.}
\end{equation}

The top and the bottom squares on the left commute by the
commutativity of proper push-forward and flat pull-back maps on $K$-theory of 
stacks, noting that $\inv^{(-)}_{(-)}$ is just the proper push-forward map
(see \thmref{thm:invariants}). The top and the bottom squares on the right
commute by definition of the various maps. 
The bottom square in the middle commutes by
\propref{prop:riemannrochmapfunct} and the top square in the middle
commutes by \thmref{thm:GRR-PROD}. It follows that
~\eqref{eqn:ASCS-0} commutes. 
\end{proof}

\section{The Riemann-Roch theorem for quotient 
stacks}\label{sec:RRT}
In this section, we shall apply the Atiyah-Segal isomorphism and
the equivariant Riemann-Roch theorem of \cite{AK1} to prove the
Grothendieck Riemann-Roch theorem for separated quotient stacks.
In order to do so, we need to recall the higher Chow groups of
such stacks and prove some properties of these groups.

\subsection{Equivariant higher Chow groups}\label{sec:ECG}
Let $G$ be a linear algebraic group.
Since most of the basic properties of equivariant higher Chow groups 
require us to assume that the underlying algebraic space with $G$-action is 
$G$-quasi-projective, we shall have to mostly work under this
assumption in this section. Of course, all group actions on such 
schemes will be linear.
However, one can check that all the properties that we need or prove below,
hold for algebraic spaces too if we are only interested in the
equivariant (not higher) Chow groups $\CH^G_*(X)$ and Riemann-Roch for these.

The equivariant higher Chow groups of quasi-projective schemes
were defined in \cite{EG1} and their fundamental properties were described in 
\cite{AK3}. We refer to these references without recalling the definition
of equivariant higher Chow groups. We only recall that they share most of
the properties of equivariant $K$-theory. In particular, they satisfy
localization, homotopy invariance, flat pull-back, proper push-forward and
Morita isomorphisms. Furthermore, they coincide with Bloch's higher Chow groups
of quotients under a free action. 
These properties will be used frequently in this section.
We shall let $\CH^G_*(X,i) = \oplus_{j \in \Z} \CH^G_i(X,i)$ for $i \ge 0$
and $\CH^G_*(X) = \CH^G_*(X,0)$. Recall also that these are all $\C$-vector
spaces under our convention.

In this text, we shall use the following convention for
differentiating the maps on $K$-theory and Chow groups, induced by a
morphism of spaces $f: X \to Y$. We shall denote the maps on
$K$-theory by $f^*$ (resp. $f_*$) and on Chow groups by 
$\bar{f}^*$ (resp. $\bar{f}_*$).

Let $H \trianglelefteq G$ be a normal subgroup of an algebraic group $G$ with
quotient $F$.
Let $G$ act properly on quasi-projective $\C$-schemes $X$ and $X'$ whose
quotients $X/G$ and ${X'}/G$ are quasi-projective.
Let $f: X' \to X$ be a $G$-equivariant proper map.
It follows from Lemmas~\ref{lem:Basic} and ~\ref{lem:quot*} 
that under this hypothesis, 
the quotients $Y = X/H$ and $Y' = {X'}/H$ are quasi-projective with proper and 
linear $F$-action. Let $f': Y' \to Y$ be the induced map between partial
quotients.

We recall the following result from
\cite[\S~4]{EG1} and \cite[\S~6.4]{EG2} 
on the functor of invariants for equivariant higher Chow groups.

\begin{thm}\label{thm:pcpu}
There is a natural map
\begin{equation}\label{diag:pcpu}
\ov{\inv}^H_X : \CH_*^G(X,i) \to \CH_*^F(Y,i)
\end{equation}

and a commutative diagram 
\begin{equation}\label{diag:pcpufunct}
\xymatrix@C1pc{ 
\CH_*^G(X',i) \ar[r]^-{\bar{f}_*} \ar[d]_{\ov{\inv}^H_{X'}} & 
\CH_*^G(X,i) \ar[d]^{\ov{\inv}^H_X} \\
\CH_*^F(Y',i) \ar[r]_{\bar{f}'_*} & \CH^F_*(Y,i)}
\end{equation}
such that $\ov{\inv}^G_X = \ov{\inv}^F_Y \circ \ov{\inv}^H_X$.

If $X$ and $X'$ are only algebraic spaces, then the above hold for
$\CH_*^G(-)$. 
\end{thm}
\begin{proof}
Let $X \xrightarrow{\pi_H} Y \xrightarrow{\pi_F} Y/F = X/G$ denote 
the quotient maps and let $\pi_G = \pi_F \circ \pi_H$.
Then the map $\pi^*_G: \CH_*(X/G,i) \to \CH^G_*(X,i)$ is defined in
the proof of \cite[Theorem~3]{EG1} and this is an isomorphism.
The map $\ov{\inv}^G_X$ is defined in \cite[\S~6.4]{EG2} for Chow
groups $\CH^G_*(X)$. However, under the quasi-projective assumption, 
this extends to higher Chow groups as well which we now explain.

We first define $\ov{\inv}^G_X$ as the inverse of $\pi^*_G$.
In particular, $\ov{\inv}^F_Y$ is the inverse of $\pi^*_F$.
We let $\ov{\inv}^H_X: \CH_*^G(X,i) \to \CH_*^F(Y,i)$ be the unique homomorphism
such that the diagram
\begin{equation}\label{eqn:Chow-pf}
\xymatrix@C1pc{
\CH_*^G(X,i) \ar[r]^{\ov{\inv}^H_X} \ar@/_2pc/[rr]^{\ov{\inv}^G_X} & 
\CH_*^F(Y,i) \ar[r]^{\ov{\inv}^F_Y} & \CH_*(X/G,i)}
\end{equation}

commutes. To show that \ref{diag:pcpufunct} commutes, we consider the diagram 
\begin{equation}\label{eqn:Chow-pf-0}
\xymatrix@C1pc{ 
\CH_*^G(X',i) \ar[d]_{\bar{f}_*} \ar[r]^-{\ov{\inv}^H_{X'}} & 
\CH_*^F(Y',i) \ar[r]^-{\ov{\inv}^F_{Y'}} \ar[d]^{\bar{f}'_*} & 
\CH_*^G(X',i) \ar[d]^{\bar{f}_*} \\
\CH_*^G(X,i) \ar[r]^-{\ov{\inv}^H_{X}}  & 
\CH_*^F(Y,i)\ar[r]^-{\ov{\inv}^F_{Y}} &
\CH_*^G(X,i).}
\end{equation}

It follows from \cite[Proposition~11(a)]{EG1}
that the right square and the big outer rectangle commute.
Since $\ov{\inv}^F_{Y}$ is an isomorphism, it follows that the
left square also commutes. For $i = 0$, an identical construction and proof
work for algebraic spaces as well.
\end{proof}

\begin{defn}\label{defn:chowgroupofstack}
Let $\sX$ be a separated quotient stack with coarse moduli 
space $M$. We define the Chow group of $\sX$ as
$\CH_*(\sX): = \CH_*(M)$. If $M$ is quasi-projective,
then we define the higher Chow groups of $\sX$ as 
$\CH_*(\sX,i): = \CH_*(M,i)$.

Let $f:\sX=[X/G] \to \sY=[Y/F]$ be a proper map of separated quotient 
stacks with coarse moduli spaces $M$ and $N$, respectively. 
Let $\tilde{f}$ be the map induced on the coarse moduli spaces. 
It follows from \lemref{lem:properpushforward} that $\tilde{f}$ is proper.

We define $f_*: \CH_*(\sX) \to \CH_*(\sY)$ to be the proper push-forward map
$\tilde{f}_* :  \CH_*(M) \to \CH_*(N)$. If $M$ and $N$ 
are quasi-projective, then we define the push-forward 
map $f_*: \CH(\sX, i) \to \CH_*(\sY,i)$ to be the corresponding proper 
push-forward map $\tilde{f}_* : \CH_*(M,i) \to \CH_*(N,i)$.  
\end{defn}

%\enlargethispage{40pt}

\subsection{Some consequences of the equivariant 
Riemann-Roch theorem}\label{sec:RR-Final}
Let $G$ be a linear algebraic group acting properly on an algebraic space
$X$. The Riemann-Roch map $G_0(G,X) \to \CH^G_*(X)$ was defined in \cite{EG5}.
If $X$ is quasi-projective, the Riemann-Roch map $G_i(G,X) \to \CH^G_*(X,i)$
was defined in \cite{AK1}. We recall the following.

\begin{thm}$($\cite[Theorem~3.1]{EG5}, \cite[Theorem~1.4]{AK1}$)$
\label{thm:AKRR}
Let $G$ act properly on a quasi-projective scheme $X$. Then for every 
$i \ge 0$, there is a Riemann-Roch map 
\begin{equation}\label{eqn:AKRR-0}
t^G_X : G_i(G,X)\rightarrow \CH_*^G(X,i)
\end{equation}
which is covariant for proper $G$-equivariant maps and the induced
map $G_i(G,X)_{\fm_1} \xrightarrow{t^G_X} \CH^G_*(X,i)$ is an isomorphism.
If $X$ is only an algebraic space, the same holds for $i = 0$.
\end{thm}

When $G$ is trivial, then the Riemann-Roch map $t^G_X$ of \thmref{thm:AKRR}
coincides with the Riemann-Roch map in the non-equivariant $K$-theory,
established by Bloch \cite{Bloch}. We shall denote this map by $\tau_X$
in the sequel.

\begin{lem}\label{lem:augmentation ideal}
Let $G$ act properly on an algebraic space $X$ with quotient $Y$ and
let $i \ge 0$ be an integer. 
If $Y$ is quasi-projective, there is a commutative diagram
\begin{equation}\label{eqn:Aug-0}
\xymatrix@C1pc{
G_i(G,X)_{\fm_1} \ar[r]^-{t^G_X} \ar[d]_{\inv^G_X} & 
\CH_*^G(X,i) \ar[d]^{\overline{\inv}^G_X} \\
G_i(Y) \ar[r]^-{\tau_Y} & \CH^*(Y,i).}
\end{equation}

If $Y$ is not necessarily quasi-projective, this diagram is commutative for
$i = 0$.
\end{lem}
\begin{proof} 
If $G$ freely acts on $X$, then $\Sigma^G_X = \{1\}$ and hence
$G_i(G,X) \simeq G_i(G,X)_{\fm_1}$. The lemma then follows from 
\cite[Theorem~4.6]{AK1} when $Y$ is quasi-projective  and
from \cite[\S~3.2]{EG5} when $Y$ is just an algebraic space.

We now prove the general case when $Y$ is quasi-projective.
Note in this case that $X$ is also quasi-projective.
By \cite[Theorem~1]{KV}, we can find a finite, flat and surjective 
$G$-equivariant map $f: X' \to X$ such that $G$ acts freely on $X'$ with 
$G$-torsor $X' \to Y':= {X'}/G$ and the induced map $u: Y' \to Y$
is finite and surjective. 
We consider the diagram
\begin{equation}\label{eqn:Aug-1}
\xymatrix@C1pc{
G_i(G,X')_{\fm_1} \ar[rr]^-{\tau^G_{X'}} \ar[dr]_{\inv^G_{X'}} \ar[dd]_{f_*} & &
\CH_*^G(X',i) \ar[dd]^/.3cm/{\bar{f}_*} \ar[dr]^-{\overline{\inv}^G_{X'}} & \\
& G_i(Y') \ar[rr]^-{\tau_{Y'}} \ar[dd]_/.3cm/{u_*} & & \CH_*(Y', i) 
\ar[dd]^{\bar{u}_*} \\ 
G_i(G,X)_{\fm_1} \ar[rr]^-{\tau^G_{X}} \ar[dr]_{\inv^G_{X}} 
& &  \CH_*^G(X,i) \ar[dr]^-{\overline{\inv}^G_X} & \\
& G_i(Y) \ar[rr]^-{\tau_Y} & & \CH_*(Y,i).}
\end{equation}

The back and the front faces of this cube commute by 
By \thmref{thm:AKRR} and its non-equivariant version. The left face commutes
by the functoriality of the push-forward maps in $K$-theory of stacks
because $\inv^G_{X}$ and $\inv^G_{X'}$ are just the push-forward maps
induced by the proper maps $[{X'}/G] \to Y'$ and $[X/G] \to Y$. The right
face commutes by \thmref{thm:pcpu} and the top face commutes 
because of the free action. Since $f$ is flat, 
the map $\bar{f}_*$ is surjective by
\cite[Lemma~11.1]{AK1}. It follows from \thmref{thm:AKRR} that $f_*$ is 
also surjective. But this easily implies that the bottom face of the cube
commutes. This finishes the proof when $Y$ is quasi-projective.

If $Y$ is only an algebraic space, then we can use \cite[Proposition~10]{EG1}
to find a finite and surjective $G$-equivariant map $f: X' \to X$ such that 
$G$ acts freely on $X'$ with $G$-torsor $X' \to Y':= {X'}/G$ and the induced 
map $u: Y' \to Y$ is finite and surjective.
It is then easy to check that the map $\bar{f}_*: \CH^G_*(X') \to \CH^G_*(X)$
is surjective (we are working with $\C$-coefficients). Now, the 
above argument works verbatim to prove that ~\eqref{eqn:Aug-0} commutes
for $i =0$.   
\end{proof}

\begin{rem}\label{rem:separatednesshypothesis}
	We note here that our assumption that the group action is proper (equivalently, the stack is separated)
	is necessary in the proof of \lemref{lem:augmentation ideal} as we rely on \cite[Theorem~1]{KV} where this 
	hypothesis is made.
\end{rem}	

\begin{cor}\label{cor:G-surj}
With notations of \lemref{lem:augmentation ideal}, the
map $\inv^G_X: G_i(G,X)_{\fm_1} \to G_i(Y)$ is an isomorphism for every
$i \ge 0$ if $Y$ is quasi-projective. 
%This map is an isomorphism for all algebraic spaces for $i =0$.
\end{cor}
\begin{proof}
The corollary is an immediate consequence of \lemref{lem:augmentation ideal}
and \thmref{thm:AKRR} because the map $\ov{\inv}^G_X$ is an isomorphism
(see \thmref{thm:pcpu}).
\end{proof}

\begin{lem}\label{lem:aiprod}
Let $G = H \times F$ act properly on an algebraic space $X$ and let 
$Y = X/H$. Let $i \ge 0$ be any integer. If $X/G$ is quasi-projective,
there is a commutative diagram
\begin{equation}\label{diag:aiprod}
\xymatrix@C1pc{
G_i(G,X)_{\fm_1} \ar[r]^-{t^G_X} \ar[d]_{\inv^H_X} & 
\CH_*^G(X,i) \ar[d]^{\ov{\inv}^H_X} \\
G_i(F,Y)_{\fm_1} \ar[r]^-{t^F_Y} & \CH_*^F(Y,i).}
\end{equation}
If $i = 0$, this diagram is commutative without any assumption on $X/G$. 
\end{lem}
\begin{proof}
We let $Z = X/G = Y/F$ and consider the diagram
\begin{equation}\label{eqn:aiprod-0}
\xymatrix@C1pc{
G_i(G,X)_{\fm_1} \ar[d]_{t^G_X} \ar[r]^-{\inv^H_X} &
G_i(F,Y)_{\fm_1} \ar[r]^-{\inv^F_Y} \ar[d]^{t^F_Y} &
G_i(Z) \ar[d]^-{\tau_Z} \\
\CH_*^G(X,i) \ar[r]^-{\ov{\inv}^H_X} & \CH_*^F(Y,i) \ar[r]^{\ov{\inv}^F_Y} &
\CH_*(Z,i).}
\end{equation}
It follows from Theorem~\ref{thm:invariants} that $\inv_Y^F \circ \inv^H_X 
= \inv^G_X$ and it follows from \thmref{thm:pcpu} that $\ov{\inv}^G_X = 
\ov{\inv}^F_Y \circ \ov{\inv}^H_X$. 
The outer rectangle and the right square both commute by
\lemref{lem:augmentation ideal}. Since the map $\ov{\inv}^F_Y$
is an isomorphism (see the proof of \thmref{thm:pcpu}), it follows that 
the left square commutes. For $i = 0$, exactly the same proof works for
algebraic spaces.
\end{proof}

\subsection{Riemann-Roch for quotient stacks}\label{sec:RRQS*}
We now define the Riemann-Roch transformation for separated quotient stacks
and prove the Riemann-Roch theorem for them.

\begin{defn}\label{defn:RR-Final-defn}
Let $\sX = [X/G]$ be a separated quotient stack of finite type over $\C$
and let $M$ denote its coarse moduli space. Let $i \ge 0$ be an integer.

%\begin{enumerate}
%\item
(1) If $M$ is quasi-projective, we let the Riemann-Roch maps be
the composite arrows 
\begin{equation}\label{eqn:RR-Final*-01}
\tau^G_X: G_i(G,X) \xrightarrow{\vartheta^G_X}
G_i(G,I_X)_{\fm_1} \xrightarrow{t^G_{I_X}} \CH^G_*(I_X,i).
\end{equation}
\begin{equation}\label{eqn:RR-Final*-0}
\tau_{\sX}: G_i(\sX) \xrightarrow{p^*_X} G_i(G,X) \xrightarrow{\tau^G_X}
\CH^G_*(I_X,i) \xrightarrow{\ov{\inv}^G_{I_X}} \CH_*(I_{\sX},i).
\end{equation}
%\item

(2) If $M$ is only an algebraic space and $i = 0$, we let $\tau^G_X$ and
$\tau_{\sX}$ be the maps as defined in ~\eqref{eqn:RR-Final*-01} and
~\eqref{eqn:RR-Final*-0}, respectively.
%\end{enumerate}
\end{defn}

It follows immediately from \lemref{lem:independenceofstackmap} and
\thmref{thm:pcpu} that $\tau_{\sX}$ is independent of the presentation
of $\sX$ as a quotient stack.

\begin{exm}\label{exm:Diag}
If $G$ is a diagonalizable group, then the semi-simple conjugacy classes of $G$
are all the singleton subsets $\{h\}$ of $G$. Since the map $\mu^h: I^h_X \to X^h$ is
an isomorphism, it easily follows from the above constructions that
$\tau^G_X$ in this case is simply the direct sum of the Riemann-Roch maps
$t^G_{X^h}: G_i(G, X^h)_{\fm_1} \xrightarrow{\simeq} \CH^G_*(X^h, i)$ of \cite{AK1}.
\end{exm}

\begin{remk}\label{remk:Comp-EGRR}
If $X$ is smooth and $i = 0$, then an equivariant Riemann-Roch map was
constructed by Edidin and Graham \cite[Theorem~6.7]{EG2}.
Using the non-abelian localization theorem of \cite{EG4}, we can check that 
the map $\tau^G_X$ of ~\eqref{diag:RR-mapfull} coincides
with that of \cite{EG2}.
 
It is enough to check this at every semi-simple conjugacy class $\psi$.
Let $p_g$ denote the projection $G_*(Z_g, X^g)_{\fm_{\psi}} \surj 
G_*(Z_g, X^g)_{\fm_g}$ and let $l_f = \lambda_{-1}(N_f^*)^{-1}$
(see \cite[\S~5.4]{EG4}). To check the agreement, all we need to show is that
$l_f \circ p_g \circ \mu^{-1}_g \circ f^* = \omega^{-1}_g \circ (j^{\psi}_*)^{-1}$,
where $j^\psi: X_{\psi} \inj X$ is the inclusion,
$j^{\psi}_*: G_*(G, X_{\psi})_{\fm_{\psi}} \to  G_*(G, X)_{\fm_{\psi}}$
is the isomorphism of \thmref{thm:closedimmiso} and 
$f = j^\psi \circ \mu^\psi: I^\psi_X \to X$ is a map of smooth algebraic
spaces. Equivalently, we need to show that
$j^{\psi}_* \circ \omega_g \circ l_f \circ p_g \circ \mu^{-1}_g \circ f^*$
is identity. Using our definition of $\omega_g$ and $j^{\psi}_* \circ \mu^\psi_* 
= f_*$, this is equivalent to showing that 
$f_* \circ \mu^\psi_g \circ l_f \circ p_g \circ \mu^{-1}_g \circ f^*$ 
is identity. In the notations of \cite[\S~5.4]{EG4}, this is equivalent to
saying that $f_*((l_f \cap f^*(\alpha))_{c_\psi}) = \alpha$ for
$\alpha \in  G_*(G, X)_{\fm_{\psi}}$. But this is precisely the
identity (13) of \cite{EG4}.
\end{remk}

%\vskip.3cm

%\subsection{The main theorems}\label{sec:MTP}
The Grothendieck-Riemann-Roch theorems that we wish to prove in this text for 
separated quotient stacks are the following. 

\begin{thm}\label{thm:GRR-QDMStacks}
Let $f: \sX \to \sY$ be a proper morphism of separated quotient 
stacks of finite type over $\C$ with quasi-projective 
coarse moduli spaces. Let $i \ge 0$ be an integer.
Then there is a commutative diagram
\begin{equation}\label{eqn:GRR-QDMStacks-0}
\xymatrix@C1pc{
G_i(\sX) \ar[r]^-{\tau_{\sX}} \ar[d]_{f_*} & \CH_*(I_{\sX},i) \ar[d]^{\bar{f^I_*}} \\
G_i(\sY) \ar[r]^-{\tau_{\sY}} & \CH_*(I_{\sY},i)}
\end{equation}
such that the horizontal arrows are isomorphisms.
\end{thm}
\begin{proof}
We let $\sX = [X/H]$ and $\sY = [Y/F]$. Let $Z$ be as in
~\eqref{eqn:Fin-coh-0-new*}. It follows from \lemref{lem:GRR-QDM-stacks}
that $G = H \times F$ acts on $Z$ such that $t$ is $G$-equivariant and
an $F$-torsor. 
By \lemref{lem:GRR-QDM-stacks}, we have proper maps maps 
$\sX = [Z/G] \xrightarrow{p} \sW = [W/F] 
\xrightarrow{q} \sY = [Y/F]$, where $W = Z/H$.
Let $\upsilon: {I_Z}/H \to I_W$ denote the canonical map induced on the
inertia schemes.
We consider the diagram
\begin{equation}\label{eqn:GRR-QDMStacks-1}
\xymatrix@C1pc{
G_i(\sX) \ar[r]^-{p_Z^*} \ar[d]_{p_*} & 
G_i(G,Z) \ar[r]^-{\vartheta^G_Z} \ar[d]^{\inv^H_Z} &
G_i(G, I_Z)_{\fm_1} \ar[r]^-{t^G_{I_X}} \ar[d]^{\upsilon_* \circ \inv^H_{I_Z}}  
& \CH^G_*(I_Z,i) \ar[d]^{\ov{\upsilon}_* \circ \ov{\inv}^H_{I_Z}} 
\ar[r]^-{\ov{\inv}^G_{I_Z}} & \CH_*(I_{\sX}, i) \ar[d]^{\ov{p}^I_*} \\
G_i(\sW) \ar[r]^-{p^*_W} \ar[d]_{q_*} & 
G_i(F, W) \ar[r]^-{\vartheta^F_W} \ar[d]^{q_*} & 
G_i(F, I_W)_{\fm_1} \ar[r]^-{t^F_W} \ar[d]^{q^I_*} &
\CH_*^F(I_W,i) \ar[d]^{\ov{q}^I_*} 
\ar[r]^-{\ov{\inv}^F_{I_W}} &  \CH_*(I_{\sW}, i) \ar[d]^{\ov{q}^I_*}  \\
G_i(\sY) \ar[r]^{p^*_Y} & G_i(F,Y) \ar[r]^-{\vartheta^F_Y} &
G_i(F, I_Y)_{\fm_1} \ar[r]^-{t^F_Y} &
\CH_*^F(I_Y,i) \ar[r]^-{\ov{\inv}^F_{I_Y}} & \CH_*(I_{\sY}, i).}
\end{equation}

The top and the bottom squares on the left commute by the
commutativity of flat pull-back and proper push-forward maps in 
$K$-theory of stacks, noting that $\inv^{(-)}_{(-)}$ is just the proper 
push-forward map (see \thmref{thm:invariants}). 
The top and the bottom squares on the right
commute by the functoriality of the proper push-forward maps
on the higher Chow groups of stacks (see \S~\ref{sec:ECG}).
%\lemref{lem:aiprod}
The second square from the left on the top commutes by \thmref{thm:GRR-PROD}
and the corresponding bottom square commutes by 
\propref{prop:riemannrochmapfunct}.
The second square from the right on the top commutes using
\lemref{lem:aiprod} and \thmref{thm:AKRR}. The corresponding bottom
square commutes by \thmref{thm:AKRR}. It follows that
~\eqref{eqn:GRR-QDMStacks-0} commutes. 
We have already shown before that all horizontal arrows in
~\eqref{eqn:GRR-QDMStacks-1} are isomorphisms. This finishes the proof.
\end{proof}

An identical proof also gives us the following Riemann-Roch theorem 
without any assumption on the coarse moduli spaces. 
In the special case where we
assume $\sX$ to be smooth and $f: \sX \to \sY = M$ the
coarse moduli space map for $\sX$, this result is due to 
Edidin and Graham (see \cite[Theorem~6.8]{EG2} and \cite[Theorem~5.1]{G1}).

\begin{thm}\label{thm:GRR$G_0$algspaces}
Let $f: \sX \to \sY$ be a proper morphism of separated quotient 
stacks. Then there is a commutative diagram
\begin{equation}\label{eqn:GRR$G_0$algspaces-0}
\xymatrix@C1pc{
G_0(\sX) \ar[r]^-{\tau_{\sX}} \ar[d]_{f_*} & \CH_*(I_\sX) \ar[d]^{\bar{f}^I_*} \\
G_0(\sY) \ar[r]^-{\tau_{\sY}} & \CH_*(I_\sY)}
\end{equation}
in which the horizontal arrows are isomorphisms.
\end{thm}

\begin{comment}
\begin{proof}
The proof of the commutativity of ~\eqref{eqn:GRR$G_0$algspaces-0}
is identical to that of \thmref{thm:GRR-QDMStacks} when we take $i =0$. 
The quasi-projectivity assumption in \thmref{thm:GRR-QDMStacks}
was only used when we have to deal with higher Chow groups.
The horizontal arrows in ~\eqref{eqn:GRR$G_0$algspaces-0}
are isomorphism because each of the map in the definition
of $\tau_{\sX}$ (see Definition~\ref{defn:RR-QDM}) is an isomorphism 
by \thmref{thm:GRR-PROD} and Definition~\ref{defn:chowgroupofstack}.
\end{proof}
\end{comment}

\noindent\emph{Acknowledgments.} The main results of this paper were
largely suggested by Dan Edidin during his conversations
with the first author on the equivariant Riemann-Roch problem.
The authors are grateful to him for his suggestion and for providing
many valuable inputs during this work.   
The authors would also like to thank Edidin and Graham for writing 
\cite{EG2} whose
influence is ubiquitous in this work.

This work was partly completed when the first author was visiting
IMS at National university of Singapore 
during the program on `Higher dimensional algebraic Geometry, holomorphic 
dynamics and their interactions' in January 2017. He would like to thank
the institute and Professor De-Qi Zhang for invitation and support. 
The authors would like thank the referee for providing many helpful
suggestions.

%\enlargethispage{20pt}

\end{document}